\documentclass[11pt]{amsart}
\usepackage[legalpaper, margin=1.2in]{geometry}
\usepackage{amssymb}
\usepackage{amsfonts}
\usepackage{amsmath}
\usepackage{dsfont}
\usepackage{mathrsfs}
\usepackage{amsthm}
\usepackage{graphicx}
\usepackage{mathtools}
\usepackage[dvipsnames]{xcolor}
\usepackage{soul}
\usepackage{mathrsfs}
\usepackage{latexsym}
\usepackage{stmaryrd}
\usepackage{tikz-cd}
\usetikzlibrary{matrix,arrows}
\usepackage{comment}
\usepackage{hyperref}
\usepackage{wasysym}
\usepackage{listings}
\usepackage{csquotes}
\usepackage{xspace}
\usepackage{hepunits}

\numberwithin{equation}{section}

\newtheorem{thm}[equation]{Theorem}
\newtheorem{prop}[equation]{Proposition}
\newtheorem{lem}[equation]{Lemma}
\newtheorem{cor}[equation]{Corollary}

\newtheorem*{statement*}{Statement}

\theoremstyle{definition}
\newtheorem{definition}[equation]{Definition}
\newtheorem{example}[equation]{Example}
\newtheorem*{question*}{Question}

\newtheorem*{conclusion*}{Conclusion}
\newtheorem*{goal*}{Goal}
\newtheorem*{idea*}{Idea}

\theoremstyle{remark}

\newtheorem{remark}[equation]{Remark}
\newtheorem{assumption}[equation]{Assumption}

\numberwithin{equation}{subsection}

\newcommand{\N}{\mathbf{N}}
\newcommand{\Q}{\mathbf{Q}}
\newcommand{\Z}{{\mathbf{Z}}}
\newcommand{\C}{\mathbb{C}}
\newcommand{\F}{\mathbf{F}}

\newcommand{\Qpbar}{{\overline{\Q}_p}}

\newcommand{\Fpbar}{{\overline{\F}_p}}
\newcommand{\Zhat}{{\widehat{\Z}}}

\newcommand{\Gal}{{\rm Gal}}
\newcommand{\Frob}{{\rm Frob}}

\newcommand{\Oh}{{\mathcal{O}}}
\newcommand{\cyc}{{\rm cyc}}

\newcommand{\GL}{{\rm GL}}

\newcommand{\ad}{{\rm ad}}
\newcommand{\Ind}{{\rm Ind}}

\newcommand{\Sym}{{\rm Sym}}
\newcommand{\rhobar}{{\overline{\rho}}}

\newcommand{\rbar}{{\overline{r}}}

\newcommand{\St}{{\rm St}}
\newcommand{\semisimple}{{\rm ss}}

\newcommand{\Image}{{\rm Im\,}}
\newcommand{\Hom}{{\rm Hom}}
\newcommand{\End}{{\rm End}}

\newcommand{\id}{{\rm id}}

\newcommand{\tr}{{\rm tr\,}}

\newcommand{\ord}{{\rm ord}}

\newcommand{\T}{{\mathbb{T}}}

\newcommand{\Tate}{{\rm Tate}}

\newcommand{\Spec}{{\rm Spec}}

\newcommand*{\sheafhom}{\mathcal{H}\kern -.5pt om}

\newcommand{\Gm}{{\mathbb{G}_m}}

\newcommand{\A}{{\mathbb{A}}}

\newcommand{\univ}{{\rm univ}}

\newcommand{\cris}{{\rm cris}}

\newcommand{\WD}{{\rm WD}}

\newcommand{\rec}{{\rm rec}}

\newcommand{\Spf}{{\rm Spf}}

\newcommand{\surj}{\twoheadrightarrow}

\newcommand{\rbarss}{{\overline{r}^{\rm ss}}}
\newcommand{\ver}{{\rm ver}}
\newcommand{\irr}{{\rm irr}}
\newcommand{\no}{{\rm no}}

\begin{document}
\title{Applications of Patching the Coherent Cohomology of Modular Curves}
\author{Chengyang Bao}
\address{Department of Mathematics, Huxley Building, South Kensington Campus, Imperial College London, London, SW7 2AZ, UK}
\email{c.bao@imperial.ac.uk}

\begin{abstract}
In this paper, we apply the Taylor--Wiles--Kisin patching method to the coherent cohomology of modular curves at minimal level.  
We establish a multiplicity-one result for the patched module by the $q$-expansion principle and show that a certain partial normalization of the crystalline deformation ring is Cohen--Macaulay. As applications, we prove new cases where crystalline deformation rings are Cohen--Macaulay, establish a Zariski density result for crystalline points in characteristic $p$, and prove a multiplicity-one result for Serre's mod-$p$ quaternionic modular forms.
\end{abstract}

\maketitle
\tableofcontents

\section{Introduction}

The main goal of this paper is to study the properties of certain crystalline deformation rings using the Taylor--Wiles--Kisin patching method applied to the coherent cohomology of modular curves. To make this precise, we begin by introducing notation and stating the assumptions of our main theorems. 

Fix a prime $p\ge 3$.  Let $L$ be a finite extension of $\Q_p$ with ring of integers $\Oh$ and residue field $\F$, and denote by $\varpi$ the maximal ideal of $\Oh$. Let $$\rbar:\Gal(\overline{\Q}_p/\Q_p)\to \GL_2(\F)$$ be a local residual representation. 
Fix a crystalline character $\psi$ of $G(\overline{\Q}_p/\Q_p)$ of Hodge--Tate weight $k-1$ for some integer $k\ge 1$; our convention is that the $p$-adic cyclotomic character $\chi_{\cyc}$ has Hodge--Tate weight $1$. We write $R_\rbar^{\Box,\psi}$ for the universal lifting ring of $\rbar$ whose points parametrize liftings of $\rbar$ with a fixed determinant $\psi.$ Kisin proved in \cite{Kisin08Potentially} that it has a quotient $R_{\rbar}^{\Box,\psi}(k)$ whose generic fiber parametrizes crystalline representations of $G_{\Q_p}$ with Hodge--Tate weights $\{0,k-1\}$ and determinant $\psi$ lifting $\rbar$. Inside $R_{\rbar}^{\Box,\psi}(k)[1/p]$ there is an element $\alpha_p$ characterized by the property that, for any $L$-algebra homomorphism
\[
    x : R_{\rbar}^{\Box,\psi}(k)[1/p] \to \overline{\Q}_p,
\]
the value $\alpha_p(x)$ is the trace of crystalline Frobenius on the corresponding crystalline representation. These two rings, $R_{\rbar}^{\Box,\psi}(k)$ and $R_{\rbar}^{\Box,\psi}(k)[\alpha_p]$, are the crystalline deformation rings we will study in this paper.

For simplicity, we state the main results in the introduction under the following assumption. 
All the theorems in \S \ref{subsec_main_theorems} concerning only a local representation $\rbar$ remain valid if this assumption is replaced by Assumption~\ref{assumption_weaker}. 

\begin{assumption}\label{assumption_semisimplification}
    The representation $\rbar$ is either semisimple, or it takes the form $$\rbar\sim \begin{pmatrix}
    \chi_1 & \chi_2b\\
    0 & \chi_2
\end{pmatrix}$$ where $\chi_i:\Gal(\overline{\Q}_p/\Q_p)\to \F^{\times}$ are characters such that $\chi_1/\chi_2\neq 1,\epsilon^{\pm 1}$ and $b$ defines a nontrivial class in $H^1(\Gal(\overline{\Q}_p/\Q_p),\chi_1/\chi_2).$
\end{assumption}

\subsection{Main Theorems}\label{subsec_main_theorems}
Our first result concerns the Cohen--Macaulay property of crystalline deformation rings.  
When all local deformation rings appearing in the Taylor--Wiles--Kisin method are Cohen--Macaulay, Kisin's result 
$
    R[1/p] = \mathbb{T}[1/p]
$ 
can be upgraded to an integral statement 
$
    R = \mathbb{T};
$
see \cite[\S 8]{HuPaskunas2019crystabelline} for details.  
For $\GL_2/\Q$, Shotton \cite{shotton2016local} gives a complete description of the local deformation rings away from $p$.
At $\ell = p$, crystalline deformation rings are regular (and hence Cohen--Macaulay) when the difference of Hodge--Tate weights $k-1$ is small relative to $p$, as described by Fontaine--Laffaille theory \cite{FontaineLaffaille1982construction}.  
For weights in the so-called \emph{poids moyen} range, roughly between $p$ and $2p$, the Cohen--Macaulay property is known by the following theorem of Hu--Pa\v{s}k\"unas \cite[Theorem 7.5]{HuPaskunas2019crystabelline}. 

\begin{thm}[Hu--Pa\v{s}k\"unas]
    Assume that $\rbar$ is Schur and $p\ge 3$. If $p=3$, further assume that $\rbar$ is not a twist of $1$ extended by the mod-$p$ cyclotomic character $\epsilon$. Then $R^{\Box,\psi}_\rbar(k)$ is Cohen--Macaulay for $2\le k\le 2p+1$.
\end{thm} 
We prove that crystalline deformation rings are Cohen--Macaulay for a larger range of weights.
\begin{thm}\label{thm_crystalline_deformation_ring_Cohen_Macaulay}
Let $k(\rbar) \in \mathbb{N}$ denote the Serre weight associated to $\rbar$ via the weight part of Serre's conjectures (see Definition \ref{def_Serre_weight}).  
Under Assumption \ref{assumption_semisimplification}, if 
\[
    k(\rbar) \le k < pk(\rbar),
\] 
then the crystalline deformation ring $R_{\rbar}^{\Box,\psi}(k)$ is Cohen--Macaulay.
\end{thm}

\begin{remark}
Hu--Pa\v{s}k\"unas \cite[Proposition 7.13]{HuPaskunas2019crystabelline} go further in the range $2 \le k \le 2p+1$: if $\rbar$ is additionally generic, they are complete intersections. 
Our theorem, on the other hand, applies under weaker restrictions on the weight $k$ and does not assume $\rbar$ is Schur. For example, the Serre weight $k(\rbar)$ is an integer that can be as large as $p^2-1.$ When $k(\rbar) > 3$ or $\rbar$ is a sum of two characters, our result covers new cases not addressed by \cite{HuPaskunas2019crystabelline}.
\end{remark}

\begin{remark}
    If $\rbar$ is an unramified representation, i.e. $k(\rbar) = 1$, then $R_\rbar^{\Box,\psi}(p)$ is not Cohen--Macaulay. This is well-known if $\rbar$ is $p$-distinguished and the non-$p$-distinguished case was proved in 
\cite{Sander2012HS}. On the other hand, when $\rbar$ is ramified, computations in the author's thesis suggest that $R_{\rbar}^{\Box,\psi}(k)$ might still be Cohen--Macaulay when $k \ge pk(\rbar)$ in some examples. The author hopes to return to this question in future work. 
\end{remark}

The second result describes density of crystalline points in characteristic $p.$ Density of crystalline points in the generic fiber of the universal deformation ring of $\rbar$ were first proved by Colmez \cite{colmez2008representations} and Kisin \cite{kisin2010deformations}. The subtlety in characteristic $p$ is that crystalline representations are by definition characteristic $0$ points so it does not really make sense to talk about crystalline points in characteristic $p.$ If one naively takes the reductions of the crystalline points, then they are all $\rbar$ which is a single point. However, we can study whether the special fibers of crystalline deformation rings as $k$ varies fill out the special fiber of the universal lifting ring of $\rbar$ in the following sense. 

Fix an unramified character $\eta:\Gal(\overline{\Q}_p/\Q_p)\to \F^{\times}$. By abuse of notation, we denote its Teichm\"uller lift by $\eta$ as well. We let $\psi_k$ be the product of $\eta$ and $\chi_\cyc^{k-1}$ for every integer $k\ge 1$ and $k\equiv k(\rbar)\pmod{p-1}.$ Then the characters $\psi_{k}$ and $\psi_{k(\rbar)}$ are congruent modulo $p,$ the deformation rings $R^{\Box,\psi_k}_\rbar\otimes_\Oh\F$ can be identified with $R^{\Box,\psi_{k(\rbar)}}_\rbar\otimes_\Oh\F.$ Therefore, the special fiber of the crystalline deformation ring $R^{\Box,\psi_{k}}_\rbar(k)$ can  be viewed as quotients of $R^{\Box,\psi_{k(\rbar)}}_\rbar\otimes_\Oh\F$, cut out by an ideal $I_k\subseteq R^{\Box,\psi_{k(\rbar)}}_\rbar\otimes_\Oh\F.$ For now, we will write all the determinants as $\psi.$ Denote by $\mathfrak{m}_{R_\rbar^{\Box,\psi}}$ the maximal ideal of $R_\rbar^{\Box,\psi}.$ 

\begin{thm} \label{thm_density_of_crystalline_points_in_char_p}
    Assume that $p\ge 5$ and Assumption \ref{assumption_semisimplification} holds. Then $$\bigcap_{k\equiv k(\rbar)\pmod{p-1}} I_k = \{0\}.$$ 
\end{thm}

Let us explain the idea behind the theorem. From a global perspective, multiplication by the Hasse invariant identifies spaces of mod-$p$ modular forms in low weight with those in high weight. One might hope to express this phenomenon purely in terms of Galois deformation theory, which heuristically suggests that the special fiber of the crystalline deformation ring $R_{\bar r}^{\Box,\psi}(k)$ ``grows'' as $k$ increases by $p-1$. Ideally, one would like to construct natural maps
\[
    R_{\bar r}^{\Box,\psi}(k+p-1)\otimes_{\Oh}\F \longrightarrow
    R_{\bar r}^{\Box,\psi}(k)\otimes_{\Oh}\F
\]
for every $k\ge 1$, and then show that these maps are surjective. Such surjections would produce a surjective inverse system
\[
    \{\, R_{\bar r}^{\Box,\psi}(k)\otimes_{\Oh}\F \,\}_{k\equiv k(\bar r)\!\!\!\pmod{p-1}},
\]
whose inverse limit could then be expected to coincide with the special fiber of the universal lifting ring $R_{\bar r}^{\Box,\psi}\otimes_{\Oh}\F$. While this conceptual idea motivates our approach, we are not able to construct these maps using the Hasse invariant alone, and we do not know how to produce them at this point. However, we are able to prove the following:

\begin{thm} \label{thm_crystalline_stronger_topology}
    Assume that $p\ge 5$ and that Assumption \ref{assumption_semisimplification} holds. Then there exists a descending sequence $\{J_k\}_{k\equiv k(\rbar)\pmod{p-1}}$ of ideals $J_k$ of $R_\rbar^{\Box,\psi}\otimes_\Oh\F$ such that \begin{enumerate}
        \item $I_k\subseteq J_k$ for every $k\equiv k(\rbar)\pmod{p-1}$ 
        \item and $\{J_k\}_{k\equiv k(\rbar)\pmod{p-1}}$ defines a stronger topology than the $\mathfrak{m}_{R_\rbar^{\Box,\psi}}$-adic topology on $R_\rbar^{\Box,\psi}\otimes_\Oh\F.$ 
    \end{enumerate} 
\end{thm}

By a purely commutative algebra fact \ref{prop_strong_top_inverse_limit_equivalence}, (2) is equivalent to $$R_\rbar^{\Box,\psi}\otimes_\Oh\F \to \varprojlim_{k\equiv k(\rbar)\pmod{p-1}}R_\rbar^{\Box,\psi}\otimes_\Oh\F/J_k.$$
Theorem \ref{thm_density_of_crystalline_points_in_char_p} now follows from the theorem above. 
\begin{proof}[Proof of Theorem \ref{thm_density_of_crystalline_points_in_char_p}]
By (1) in Theorem \ref{thm_crystalline_stronger_topology}, it suffices to prove the claim with $I_k$ replaced by $J_k$. 
In other words, we need to show that the $\{J_k\}_{k \equiv k(\rbar) \pmod{p-1}}$-topology is separated. 
But this topology is stronger than the $\mathfrak{m}_{R_\rbar^{\Box,\psi}}$-adic topology on $R_\rbar^{\Box,\psi} \otimes_\Oh \F$, which is separated because $R_\rbar^{\Box,\psi} \otimes_\Oh \F$ is $\mathfrak{m}_{R_\rbar^{\Box,\psi}}$-adically complete. 
It follows that the $\{J_k\}$-topology is separated, and the proof is complete.
\end{proof}

Using Theorem \ref{thm_crystalline_stronger_topology}, we can also explain
some data collected in the author's thesis. There we developed an algorithm computing
the presentations of approximations of crystalline deformation rings. In particular,
we computed several examples of the approximations of the Hilbert functions of
$R_\rbar^{\Box,\psi}(k)\otimes_\Oh\F$ as $k\to \infty$ among those integers
congruent to $k(\rbar)\pmod{p-1}$. Recall that the Hilbert function of a Noetherian local ring $(R,\mathfrak{m})$ with 
residue field $\F$ is
\[
    H_R(x) := \sum_{n\ge 0} \dim_{\F}\!\left(\mathfrak{m}^{n}R / \mathfrak{m}^{n+1}R\right) x^n.
\]
The data we collected suggested the following limit of Hilbert
functions, which we can now prove.

\begin{thm} \label{thm_Hilbert_function}
    Assume that $p\ge 5$ and Assumption \ref{assumption_semisimplification} holds. Then we have $$\lim_{\substack{k\to \infty\\k\equiv k(\rbar)\pmod{p-1}}} H_{R^{\Box,\psi}_\rbar(k)\otimes_\Oh\F}(x)  = H_{R^{\Box,\psi}_\rbar\otimes_\Oh\F}(x),$$ where the limit is understood as such: the $i$-th coefficient of $H_{R^{\Box,\psi}_\rbar(k)\otimes_\Oh\F}(x)$ converges to the $i$-th coefficient of $H_{R^{\Box,\psi}_\rbar\otimes_\Oh\F}(x)$ for every integer $i\ge 0.$
\end{thm}

\begin{proof}
By definition, it suffices to prove the statement modulo $\mathfrak{m}^i_{R_\rbar^{\Box,\psi}}$ for every integer $i\ge 0$. 
In this case, the Hilbert functions involved are simply polynomials. 
We have surjections
\[
R^{\Box,\psi}_\rbar \otimes_\Oh \F / \mathfrak{m}^i_{R_\rbar^{\Box,\psi}} 
\;\;\surj\;\; 
R^{\Box,\psi}_\rbar(k) \otimes_\Oh \F / \mathfrak{m}^i_{R_\rbar^{\Box,\psi}} 
\;\;\surj\;\; 
R^{\Box,\psi}_\rbar \otimes_\Oh \F / (J_k + \mathfrak{m}^i_{R_\rbar^{\Box,\psi}}),
\]
which imply the following inequalities of Hilbert functions:
\[
H_{R^{\Box,\psi}_\rbar \otimes_\Oh \F / \mathfrak{m}^i_{R_\rbar^{\Box,\psi}}}(x) 
\;\ge\; 
H_{R^{\Box,\psi}_\rbar(k) \otimes_\Oh \F / \mathfrak{m}^i_{R_\rbar^{\Box,\psi}}}(x) 
\;\ge\; 
H_{R^{\Box,\psi}_\rbar \otimes_\Oh \F / (J_k + \mathfrak{m}^i_{R_\rbar^{\Box,\psi}})}(x).
\]
By Theorem \ref{thm_crystalline_stronger_topology}, for $k$ sufficiently large we have an isomorphism
\[
R^{\Box,\psi}_\rbar \otimes_\Oh \F / \mathfrak{m}^i_{R_\rbar^{\Box,\psi}} 
\;\xrightarrow{\sim}\; 
R^{\Box,\psi}_\rbar \otimes_\Oh \F / (J_k + \mathfrak{m}^i_{R_\rbar^{\Box,\psi}}),
\]
forcing all the inequalities above to be equalities. 
This completes the proof.
\end{proof}

We point out that neither the theorem nor its proof gives information on the speed of convergence; see Remark~\ref{rmk_speed_of_convergence} for more details.

At last, we prove a multiplicity one result for Serre's modular forms. For this, we need a global representation $$\rhobar:\Gal(\overline{\Q}/\Q)\to \GL_2(\F)$$ and we identify $\rhobar|_{\Gal(\overline{\Q}_p/\Q_p)}$ with $\rbar.$ We assume that the Artin
conductor $N(\rhobar)$ is at least $5.$ 
For convenience, we temporarily adopt notation for level subgroups, modular curves, and automorphic bundles that differs from that in \S\ref{subsubsec_modular_forms}. 
The formulation using the standard notation appears later as Theorem~\ref{thm_Serre_mult_one}.

Set
\[
    U := \left\{
    \begin{pmatrix} a & b \\ c & d \end{pmatrix} \in \GL_2(\hat{\Z}) :
    \begin{pmatrix} a & b \\ c & d \end{pmatrix}
    \equiv
    \begin{pmatrix} * & * \\ 0 & 1 \end{pmatrix}
    \pmod{N(\rhobar)}
    \right\}.
\]
Let $X_{U,\F}$ be the compact modular curve over $\F$ parametrizing generalized elliptic
curves with level structure $U$.  
Write $\pi : \mathcal{E}_{U,\F} \to X_{U,\F}$ for the universal generalized elliptic curve,
and set
$
    \omega := \pi_*\bigl(\Omega^1_{\mathcal{E}_{U,\F}/X_{U,\F}}\bigr).
$
In characteristic~$p$, multiplication by the Hasse invariant $A$ induces, for every
integer $k$, a map of line bundles
\[
    A \cdot \omega^{\otimes (k-(p-1))}
        \;\longrightarrow\;
    \omega^{\otimes k}.
\]
Define $\mathcal{S}^k$ to be the cokernel, and put
\(S(k,U) := H^0(X_{U,\F}, \mathcal{S}^k)\).
Let $\mathfrak{m}_\rhobar$ be the maximal ideal of
\(\T^{pN(\rhobar)} := \F[T_\ell, \langle \ell\rangle]_{\ell \nmid pN(\rhobar)}\)
generated by
\[
    \bigl(T_\ell - \tr\rhobar(\Frob_\ell),\; 
          \langle \ell\rangle - \det(\rhobar)/\chi_\cyc^{k(\rhobar)-1}
    \bigr)_{\ell \nmid pN(\rhobar)}.
\]

\begin{thm}
Assume the following:
\begin{enumerate}
    \item $p \ge 5$;
    \item $\rhobar$ is odd;
    \item $\rhobar|_{\Gal(\Q(\zeta_p)/\Q)}$ is absolutely irreducible;
    \item $N(\rhobar) \ge 5$;
    \item $\rbar$ is not a twist of 
    \(
        \begin{pmatrix}
            \epsilon & * \\
            0 & 1
        \end{pmatrix}
    \)
    with $*$ trivial or peu ramifi\'ee.
\end{enumerate}
Then for every weight $k \ge 1$, the $\F$-vector space
\(
    S(k,U)[\mathfrak{m}_\rhobar]
\)
is at most one-dimensional. The same conclusion holds when $\rbar$ is a twist of the above shape, provided that there are no vexing primes.
\end{thm}

\begin{remark} \label{rmk_vexing_primes}
A prime $\ell$ is called ``vexing'' for $\rhobar$ if
$
\ell \mid N(\rhobar), \ell \equiv -1 \pmod p,
$
and $\rhobar|_{\Gal(\overline{\Q}_\ell/\Q_\ell)}$ is an absolutely irreducible representation induced from the absolute Galois group of the unique unramified quadratic extension of $\Q_\ell$.  
Their effect is explained in more detail in \S\ref{subsubsection_vexing_primes}.
\end{remark}

\begin{remark}
    The multiplicity one theorem implies that some patched module
    is free of rank one over a formal power series ring over the completed tensor
    product of certain local deformation rings.
    This was first announced in \cite[\S7.5.13]{EGH_2022_IHES_notes}. There it was used
    to describe the image of the conjectural categorical $p$-adic local Langlands functor
    \(\mathfrak{A}_{D^\times}\) attached to a non-split quaternion algebra
    \(D/\Q_p\).
    Roughly, the functor is expected to send compact inductions of characters of the unit group of the maximal order of $D$ to sheaves supported on the underlying reduced of potentially crystalline Emerton--Gee stacks of cuspidal inertial type; 
    multiplicity one and
    compatibility with patching lead to the expectation that this sheaf is a line bundle.
    See~\cite[\S7.5.6]{EGH_2022_IHES_notes} for details.
    Andrea Dotto and Le Hung Bao have given an independent proof of multiplicity one using a different method in \cite{DottoLeHung_GKdimension_modp}. 
\end{remark}

\subsection{Overview of the Proof} 
Our main construction and arguments are based on the Taylor--Wiles--Kisin patching method applied to the coherent cohomology of modular curves. 
Accordingly, we impose the following global assumption, and explain in \S \ref{subsubsec_remove_global_assumption} how to replace it with Assumption \ref{assumption_semisimplification}.
\begin{assumption} \label{assumption_global_rhobar}
    The local residual representation $\rbar:\Gal(\overline{\Q}_p/\Q_p)\to \GL_2(\F)$ is the restriction of a global Galois representation $\rhobar:\Gal(\overline{\Q}/\Q)\to \GL_2(\F)$ such that
\begin{enumerate}
    \item $\rhobar$ is odd, i.e., $\det(\rhobar(c))=-1$ for a complex conjugation $c;$
    \item \label{assumption_Taylor_Wiles}(Taylor--Wiles condition) the restriction $\rhobar|_{\Gal(\Q(\zeta_p)/\Q)}$ is absolutely irreducible; 
    \item either $p \ge 5$, or that the Artin conductor $N(\rhobar)$ of $\rhobar$ satisfies $N(\rhobar) \ge 5$. 
\end{enumerate}
\end{assumption} 

\begin{remark} \label{rmk_assumption_explanation}
Conditions (1) and (2) imply that the representation $\rhobar$ arises from a Katz modular form in characteristic~$p$.  
Such a form lifts to a classical characteristic $0$ modular form of the same weight, level, and Nebentypus character provided that the corresponding modular curve represents the relevant moduli problem and that the weight is at least~$2$.  
Condition (3) ensures this representability, see \S \ref{subsubsection_Galois_representations} for a detailed discussion.  
Condition (2) additionally guarantees that the Taylor--Wiles--Kisin patching method can be applied.
\end{remark}

\begin{remark} \label{rmk_semisimple_reps_satisfy_assumption_global}
    When $\rbar$ is semisimple, the assumption holds true after possibly replacing $\F$ by a finite extension. In this case,  one can take $\rhobar$ to be an absolutely irreducible induced from a suitable quadratic imaginary field, i.e., the residual representation attached to a CM eigenform. 
\end{remark}

\subsubsection{Main Construction}\label{subsubsection_main_construction}
We explain the main construction, which is due to Matthew Emerton, in a simplified setting: we assume that there are no vexing primes and no extra framing variables, i.e., we consider the case where the patched ring is exactly the crystalline deformation ring $R_\rbar^{\Box,\psi}(k)$. 

Let $Q_n$ be a set of Taylor--Wiles primes that are congruent to $1$ modulo $p^n$ and let $N_{Q_n}$ be the product of the Artin conductor $N(\rhobar)$ and all the primes in $Q_n.$ We patch the dual of the weight $k$ coherent cohomology $M(k,Q_n)$ of a modular curve of level $N_{Q_n}$. By the $q$-expansion principle, $M(k,Q_n)$ can be identified with the full Hecke algebra $\T(M(k,Q_n))$ acting on $M(k,Q_n)$. (More generally, we denote by $\T^N(M(k,Q_n))$ the Hecke algebra generated by Hecke operators away from the integer $N.$) Modularity theorems of Kisin \cite{Kisin09FontaineMazur} imply that the patched module is a faithful $R_\rbar^{\Box,\psi}$-module.

On the other hand, the action of $R_\rbar^{\Box,\psi}$ on the patched module extends faithfully to $R_\rbar^{\Box,\psi}[\alpha_p]$ \cite{Caraiani_Emerton_Gee_Geraghty_PaskunasShin_2018_GL2Qp}. 
For each $n$, the map
\[
R_\rbar^{\Box,\psi}(k)[\alpha_p] \longrightarrow \T(M(k,Q_n)),
\]
extends the map $R_\rbar^{\Box,\psi}(k) \to \T^{pN_{Q_n}}(M(k,Q_n))$ and sends $\alpha_p$ to the Hecke operator $T_p$. Since we patch at the minimal level $N(\rhobar)$, Lemma \ref{lem_U_ell_in_anemic} (proved by Wiles in {wiles1995modular}) implies that 
$
\T^{pN_{Q_n}}(M(k,Q_n)) = \T^{p}(M(k,Q_n)),
$
so the above map is surjective at the minimal level. By the right exactness of the patching functor, it follows that the patched module is a cyclic $R_\rbar^{\Box,\psi}(k)[\alpha_p]$-module.

\begin{thm}[Emerton]\label{thm_intro_patched_module_free_of_rank_one}
    The patched module is finite free of rank one over $R_\rbar^{\Box,\psi}(k)[\alpha_p]$. 
\end{thm}

Since the patched module is maximal Cohen--Macaulay over $R_\rbar^{\Box,\psi}(k)$ and $R_\rbar^{\Box,\psi}(k)[\alpha_p]$ is finite over $R_\rbar^{\Box,\psi}(k)$, it is also maximal Cohen--Macaulay over $R_\rbar^{\Box,\psi}(k)[\alpha_p]$. We thus obtain:

\begin{thm}[Emerton] \label{thm_intro_Rap_CM}
    Under Assumption \ref{assumption_global_rhobar}, the ring $R_\rbar^{\Box,\psi}(k)[\alpha_p]$ is Cohen--Macaulay. 
\end{thm}

\begin{remark}
A more general version of Theorem \ref{thm_intro_patched_module_free_of_rank_one} is stated in Lemma~\ref{lem_Hecke_algebras_patch_to_deformation_rings}(1), where we allow extra framing variables and vexing primes, and single out the ordinary and non-ordinary components of $R_\rbar^{\Box,\psi}(k)[\alpha_p]$. Similarly, Theorem~\ref{thm_intro_Rap_CM} holds under the weaker Assumption~\ref{assumption_weaker}, as stated in Theorem~\ref{thm_R_ap_cohen_macaulay}.
\end{remark}

\subsubsection{Proof of the Main Theorems} \label{subsubsection_proof_of_main_theorems}
To prove Theorem \ref{thm_crystalline_deformation_ring_Cohen_Macaulay}, one wants to understand the difference between $R_\rbar^{\Box,\psi}(k)[\alpha_p]$ and $R_\rbar^{\Box,\psi}(k)$. At a finite level, this reduces to analyzing the cokernel $C(M(k,Q_n))$ of the map 
$$\T^{p}(M(k,Q_n))\to \T(M(k,Q_n)).$$
If $C(M(k,Q_n))\otimes_\Oh\F$ is trivial, then $R_\rbar^{\Box,\psi}(k)[\alpha_p] = R_\rbar^{\Box,\psi}(k)$, so the Cohen--Macaulay property of the former immediately implies that for the latter. Using the $q$-expansion principle, in Corollary \ref{cor_anemic_equal_full_criterion}, we identify $C(M(k,Q_n))\otimes_\Oh\F$ with the dual of the kernel of the $\theta$ operator in characteristic $p$, which is studied by Katz \cite{katz1977result}. This gives a concrete criterion in terms of $k$ for when $C(M(k,Q_n)) = 0$.

Write $R_\rbar^{\Box,\psi}(k)^\no$ for the quotient of the crystalline deformation ring $R_\rbar^{\Box,\psi}(k)$ corresponding to its non-ordinary components. We prove the following theorem, from which Theorem~\ref{thm_crystalline_stronger_topology} follows (under Assumption~\ref{assumption_global_rhobar}) by taking $$J_k := \ker(R_\rbar^{\Box,\psi}\otimes_\Oh\F\to \Image(R_\rbar^{\Box,\psi}(k)^\no\otimes_\Oh\F 
    \to R_\rbar^{\Box,\psi}(k)^\no[\alpha_p]\otimes_\Oh\F)).$$
\begin{thm} \label{thm_liminf_of_crystalline_is_everything}
    Assume that $p\ge 5$ and that Assumption \ref{assumption_global_rhobar} holds. Then there is a surjection \begin{align*}
        \Image(R_\rbar^{\Box,\psi}(k + p-1)^\no\otimes_\Oh\F \to R_\rbar^{\Box,\psi}(k + p-1)^\no[\alpha_p]\otimes_\Oh\F)\\\surj \Image(R_\rbar^{\Box,\psi}(k)^\no\otimes_\Oh\F \to R_\rbar^{\Box,\psi}(k)^\no[\alpha_p]\otimes_\Oh\F)
    \end{align*} for every $k\ge 2.$ With these surjections as transition maps, the inverse limit satisfies
\[
    \varprojlim_k 
    \Image(R_\rbar^{\Box,\psi}(k)^\no\otimes_\Oh\F 
    \to R_\rbar^{\Box,\psi}(k)^\no[\alpha_p]\otimes_\Oh\F)
    \xrightarrow{\sim} R_\rbar^{\Box,\psi}\otimes_\Oh\F.
\]  
\end{thm}

Let $M(k,Q_n)^{\no}$ denote the non-ordinary part of $M(k,Q_n)$, i.e.\ the submodule on which $T_p$ acts topologically nilpotently.The main idea in proving Theorem \ref{thm_liminf_of_crystalline_is_everything} is to patch the Hecke algebra
\(\T^{p}(M(k,Q_n)^\no\otimes_\Oh\F)\) and identify the patched ring with
\[
\Image\big(R_\rbar^{\Box,\psi}(k)^\no\otimes_\Oh\F \to R_\rbar^{\Box,\psi}(k)^\no[\alpha_p]\otimes_\Oh\F\big)
\]
in Lemma \ref{lem_Hecke_algebras_patch_to_deformation_rings}.
This identification allows one to construct the surjections in Theorem \ref{thm_liminf_of_crystalline_is_everything} using the Hasse invariant. 
To conclude that the inverse limit is the special fiber of the universal lifting ring \(R_\rbar^{\Box,\psi}\), we use the fact that \(R_\rbar^{\Box,\psi}\otimes_\Oh\F\) is a domain, which is proved in \cite{Bockle_Iyengar_Paskunas_2023_local_deformation_rings} and reduce the problem to counting the dimension of the inverse limit. 
Ultimately, this reduces to bounding the Krull dimension of 
$\T^{p}(M(k,Q_n)^\no \otimes_\Oh \F)$ from below by two. 
This bound is established in \cite{Bellaiche_Khare_2015_Hecke} for level one and arbitrary $p$, 
and in \cite{deo_2017_structure_of_hecke_algebras} for general levels with $p \ge 5$. 
Hence, the restriction $p \ge 5$ is imposed in all the theorems that rely on Deo's result.

For the multiplicity one result of Serre's modular forms, we again use the fact that $q$-expansions determine modular forms, together with certain smoothness properties of local deformation rings, as explained in \cite[\S7.5.13]{EGH_2022_IHES_notes}.

\subsubsection{Vexing Primes} \label{subsubsection_vexing_primes}
We now explain the strategy when there are vexing primes, as suggested by Chandrashekhar Khare. In this case, the patched ring is a formal power series ring over the completed tensor product of the crystalline deformation ring $R_\rbar^{\Box,\psi}(k)$ and the universal lifting rings of local representations at vexing primes. Fortunately, these rings are thoroughly studied in \cite{shotton2016local}, which allows us to carry over much of the previous strategy. The key point in proving Theorems \ref{thm_crystalline_deformation_ring_Cohen_Macaulay} and \ref{thm_liminf_of_crystalline_is_everything} is that certain local deformation rings in characteristic $p$ are Cohen--Macaulay and have the desired dimensions. These properties are still satisfied by the special fibers of the universal lifting rings of local representations at vexing primes. 

The proof of the multiplicity one result largely carries over, except in the following case. When $\rbar$ is, up to twist, isomorphic to 
\(
\begin{pmatrix}
    \epsilon & *\\
    0 & 1
\end{pmatrix}
\)
with $*$ peu ramifi\'ee, we need a quotient of the patched module to have finite projective dimension over a ring of the form $\F[x]/(x^n)$ for some integer $n\ge 2$. For these rings, all modules have projective dimension either $0$ or $\infty$. But $\F$ as an $\F[x]/(x^n)$-module has infinite projective dimension $\infty.$ It is not clear how to conclude the result from commutative algebra alone. It is plausible that one could instead use type theory to patch modular forms of a specified type, following \cite[\S 3.9]{calegari_geraghty_2018_beyond} to fully address the problems arising from vexing primes. But for convenience, we simply assume that no vexing primes occur in this case.

\subsubsection{Removing the Global Assumption} \label{subsubsec_remove_global_assumption}
We now explain how to replace Assumption~\ref{assumption_global_rhobar} with Assumption~\ref{assumption_semisimplification}. The idea follows the approach of Kisin in \cite{Kisin09FontaineMazur}. 
As observed in Remark~\ref{rmk_semisimple_reps_satisfy_assumption_global}, the global assumption is automatically satisfied when $\rbar$ is semisimple. 
We may therefore assume that $\rbar$ is not semisimple and satisfies Assumption~\ref{assumption_semisimplification}. In this case, the deformation theory and pseudo-deformation theory of $\rbar$ coincide. 
Moreover, by definition, the pseudo-deformation theories of $\rbar$ and its semisimplification $\rbarss$ agree. 
It thus remains to compare the deformation and pseudo-deformation theories of $\rbarss$, which is thoroughly studied in \cite{paskunas_2017_2_adic_deformations}. Using these results, we deduce the Cohen--Macaulay property and the density of crystalline points for $\rbar$ from the corresponding statements for $\rbarss$ by commutative algebra arguments carried out in \S\ref{subsubsec_semisimplification}. In particular, all the results concerning a local representation $\rbar$ stated in the introduction remain valid under Assumption~\ref{assumption_semisimplification}, and more generally under the following weaker condition:
\begin{assumption}\label{assumption_weaker}
Either Assumption~\ref{assumption_semisimplification} or Assumption~\ref{assumption_global_rhobar} holds.
\end{assumption}

We note that an alternative approach to removing Assumption~\ref{assumption_global_rhobar} would be to apply the Moret--Bailly theorem (see \cite[\S 3]{calegari_2012_even_II} and \cite[Proposition 3.2.1]{EmertonGee2014Geometric}) to obtain a global Galois representation over a totally real field $F$ rather than $\Q$. Then one could carry out a patching argument using Hilbert modular forms after checking all the relevant results on modular forms still hold in the Hilbert setting, but this is beyond the scope of the current paper.

\subsubsection{Organization of the Paper} \label{subsubsection_organization_of_the_paper}
The paper is organized as follows. In \S \ref{section_prelim}, we recall some results in commutative algebra, modular forms and Galois deformation theory that we will use in the sequel. In \S\ref{section_patching}, we carry out the main construction described in \S \ref{subsubsection_main_construction}. We choose to use Scholze's ultrapatching functor \cite{Scholze_2018_Lubin_Tate} to describe the Taylor--Wiles--Kisin patching process. This is because we need to study maps between patched modules and patched rings, and the ultrapatching functor naturally constructs these maps while allowing the patching process to be written in a fairly clean way. In the last section, we prove the main theorems following the proof strategy outlined above.

\subsection*{Notation and Conventions}\label{subsubsection_notation_and_conventions}
Throughout this article we fix an odd rational prime $p>2.$ Let $L$ be a finite extension of $\Q_p$ with ring of integers $\Oh$ , maximal ideal $(\varpi)$ and residue field $\F$. We assume that $\Oh$ is large enough so that all scalars appearing in the paper are in $L$ or $\F.$ Unless stated otherwise, an irreducible representation is understood to mean an absolutely irreducible one.

For any field $F,$ we denote its the algebraic closure by $\overline{\F}$ and write $G_F$ for the absolute Galois group $\Gal(\overline{F}/F).$ 
For a prime $\ell,$ we denote by $I_\ell \triangleleft G_{\Q_\ell}$ the inertia subgroup and by $P_\ell\triangleleft I_\ell$ the wild inertia subgroup. The quotient $G_{\Q_\ell}/I_\ell$ is topologically generated by the arithmetic Frobenius element $\Frob_\ell,$ whose action on $\overline{\F}_\ell$ is given by $x\mapsto x^\ell.$ 
The quotient $I_\ell/P_\ell$ is isomorphic to the pro-cyclic group $\prod_{v\neq \ell}\Z_v.$ 
If $S$ is a finite set of finite places of $\Q,$ we write $\Q_S$ for the maximal extension of $\Q$ unramified outside $S\cup \{\infty\}$ and set $G_{\Q,S} = \Gal(\Q_S/\Q).$

We let $\chi_{\cyc}:G_{\Q}\to \Z_p^{\times}$ be the $p$-adic cyclotomic character. For $n\ge 1,$ we let $W(\F_{p^n})$ be the ring of Witt vectors of $\F_{p^n}$ and let $\epsilon_n:W(\F_{p^n})^{\times} \to \F_{p^n}^{\times}$ be the character induced by reduction modulo $p.$ It can be extended to a character of $W(\F_{p^n})[1/p]^{\times}$ by mapping $p$ to $1.$ By local class field theory, this defines a character of $G_{W(\F_{p^n})[1/p]}.$ When $n=1,$ $\epsilon_1$ coincides with the mod-$p$ reduction of $\chi_\cyc,$ which we denote by $\epsilon.$ 

Let $\mathcal{C}_\Oh$ be the category of completed local Noetherian $\Oh$-algebras with residue field $\F$. Let $G = G_{\Q_\ell}$ for some prime $\ell$ or $G_{\Q,S}$ for a finite set of primes $S.$  
Let $\rhobar:G\to \GL_2(\F)$ be a continuous representation of $G$ over $\F.$ The functor sending an object of $\mathcal{C}_{\Oh}$ to the set of liftings of $\rhobar$ to $\Oh$ is represented by the universal lifting ring $R_\rhobar^{\Box}$ in $\mathcal{C}_{\Oh},$ equipped with a universal lift $\rho^{\Box}:G\to \GL_2(R^{\Box}_\rhobar).$ For every character $\psi:G\to \Oh^\times$ with $\psi \equiv \det(\rhobar)\pmod \varpi$, there is the fixed determinant universal lifting ring $R_\rhobar^{\Box,\psi}$, which is the quotient of $R_\rhobar^{\Box}$ parametrizing deformations of $\rhobar$ with determinant $\psi$. One can check from the universal property that $$R^{\Box,\psi}_\rhobar\xrightarrow{\sim} R^{\Box}_\rhobar/(\det \rho^{\Box}(g)-\psi(g))_{g\in G}.$$ When $p\ge 3$ and $\psi$ and $\psi'$ are two such characters, the ring $R_\rhobar^{\Box,\psi}$ and $R_\rhobar^{\Box,\psi'}$ are isomorphic: there is a unique character $\delta: G \to 1 + \varpi\Oh$ such that $\delta^2 = \psi^{-1}\psi'.$ Twisting by $\delta$ induces the desired isomorphism.

Given a ring homomorphism $f :S\to R$ and an ideal $\mathfrak{a}\subset S,$ we write $\mathfrak{a}^e$ for its extension $\mathfrak{a}R$ to $R$; for an ideal $\mathfrak{b}\subset S$, we write $\mathfrak{b}^c$ for its contraction $f^{-1}(\mathfrak{b})$ to $S$. The extensions and contractions depend on $f$, though we suppress it from the notation when there is ambiguity. By definition, we always have $\mathfrak{a} \subseteq \mathfrak{a}^{ec}$ for every ideal $\mathfrak{a}$ of $S.$

We write $v_p$ for the additive $p$-adic valuation, normalized so that $v_p(p)=1.$ For a positive integer $n$, we write $\zeta_n$ for a fixed choice of a primitive $n$-th root of unity. 

\section{Preliminaries} \label{section_prelim}
In this section, we recall some preliminaries on modular forms and Galois representations. We claim no originality in this section. But we sometimes include proofs for completeness. 

\subsection{Some Commutative Algebra}

We first recall some commutative algebra concerning inverse limits of directed inverse systems over a countable index set $I$. These are used in dealing with completed tensor products and ultrapatching. 

\begin{lem}\label{lem_inverse_limit_exactness}
    Let 
    \[
        0 \to \{A_i\} \to \{B_i\} \to \{C_i\} \to 0
    \]
    be a short exact sequence of directed inverse systems of abelian groups. Then 
    \[
        0 \to \varprojlim_i A_i \to \varprojlim_i B_i \to \varprojlim_i C_i
    \]
    is exact. If in addition the system $\{A_i\}$ satisfies the Mittag--Leffler condition (i.e.\ the transition maps eventually stabilize), then 
    \[
        0 \to \varprojlim_i A_i \to \varprojlim_i B_i \to \varprojlim_i C_i \to 0
    \]
    is also exact.
\end{lem}

\begin{remark}
For a proof, see for example \cite[\href{https://stacks.math.columbia.edu/tag/02N1}{Lemma~02N1}]{stacks-project}.  
In practice, the Mittag--Leffler condition holds whenever $\{A_i\}$ is a surjective inverse system. More generally, if all maps in the sequence are $R$-module homomorphisms for some ring $R$, then it also holds whenever each $A_i$ has finite length as an $R$-module. We will sometimes use the following convenient form of the right exactness.
\end{remark}
\begin{cor}\label{cor_mittag_leffler_right_exact}
Let $R$ be a ring and let
\[
\{A_i\}_{i\in I}\xrightarrow{\{f_i\}} \{B_i\}_{i\in I}\xrightarrow{} \{C_i\}_{i\in I}\to 0
\]
be a right exact sequence of inverse systems of $R$-modules indexed by a countable set \(I\).
If both $\{\ker f_i\}_{i\in I}$ and $\{\Image f_i\}_{i\in I}$ satisfy the Mittag--Leffler condition,
then
\[
\varprojlim_i A_i\longrightarrow \varprojlim_i B_i\longrightarrow \varprojlim_i C_i\longrightarrow 0
\]
is right exact.
\end{cor}

\begin{proof}
    Consider the short exact sequences $$0\to \{\Image f_i\}\to \{B_i\}\to \{C_i\}\to 0$$ and $$0\to \{\ker f_i\}\to \{A_i\}\to \{\Image f_i\}\to 0.$$ 
    By the previous lemma, we then have short exact sequences $$0\to \varprojlim_i \Image f_i\to \varprojlim_i B_i\to \varprojlim_i C_i\to 0$$ and $$0\to \varprojlim_i \ker f_i \to \varprojlim_i A_i\to \varprojlim_i \Image f_i\to 0$$ and the proof is complete. 
\end{proof}

\begin{lem} \label{lem_inverse_limit_tensor_product}
    Let $R$ be a Noetherian ring and let $M$ be a finitely generated $R$-module. Suppose that either one of the following conditions is satisfied:
    \begin{enumerate}
        \item[(a)] $\{A_i\}$ is an inverse system of finite length $R$-modules;
        \item[(b)] $\{A_i\}$ is a surjective inverse system of free $R$-modules.
    \end{enumerate}
    Then
    \[
    \varprojlim_i (A_i\otimes_R M) \xrightarrow{\;\sim\;} \left(\varprojlim_i A_i\right)\otimes_R M.
    \]
\end{lem}

\begin{proof}
    Since $M$ is finitely generated over the Noetherian ring $R$, it is finitely presented. We may therefore fix an exact sequence
    \[
    0\to N\to R^{\oplus n}\to R^{\oplus m}\to M\to 0
    \]
    for some $m,n$ and some finitely generated $R$-module $N$.  

    In case~(a), tensoring with $A_i$ gives a right exact sequence
    \[
    A_i^{\oplus n}\to A_i^{\oplus m}\to A_i\otimes_R M\to 0.
    \]
    Since each $A_i$ has finite length, so do all subquotients of $A_i^{\oplus n}$. Hence the inverse systems $\{\ker A_i^{\oplus n}\to A_i^{\oplus m}\}$ and $\{\Image A_i^{\oplus n}\to A_i^{\oplus m}\}$ satisfy the Mittag--Leffler condition. By Corollary~\ref{cor_mittag_leffler_right_exact}, we have the following right exact sequence:
    \[
    \varprojlim_i A_i^{\oplus n}\to \varprojlim_i A_i^{\oplus m}\to \varprojlim_i (A_i\otimes_R M)\to 0.
    \]
    As inverse limit commutes with finite direct sums, this is
    \[
    \bigl(\varprojlim_i A_i\bigr)\otimes_R R^{\oplus n}\to
    \bigl(\varprojlim_i A_i\bigr)\otimes_R R^{\oplus m}\to
    \varprojlim_i(A_i\otimes_R M)\to 0,
    \]
    which identifies the rightmost term with $(\varprojlim_i A_i)\otimes_R M$.  

    In case~(b), tensoring with the free modules $A_i$ yields an exact sequence
\[
0\to A_i\otimes_R N \to A_i^{\oplus n}\to A_i^{\oplus m}\to A_i\otimes_R M\to 0.
\]
Since $\{A_i\}$ is a surjective system of free modules, the inverse systems $\{A_i\otimes_R N\}$ and $\{A_i^{\oplus m}\}$ are also surjective and hence satisfy the Mittag--Leffler condition. 
The rest of the proof is identical to case~(a).
\end{proof}

We next recall some facts about the category $\mathcal{C}_{\Oh}.$ If $(A,\mathfrak{m}_A)$ and $(B,\mathfrak{m}_B)$ are two objects of $\mathcal{C}_{\Oh},$ then $$A\hat{\otimes}_{\Oh} B := \varprojlim_{a,b}A/\mathfrak{m}_A^a\otimes_{\Oh}B/\mathfrak{m}_B^{b}$$ is an object of $\mathcal{C}_{\Oh}.$ When $A$ and $B$ are $\F$-algebras, the completed tensor product over $\Oh$ is the same as that over $\F.$ If $M$ and $N$ are finitely generated $A$-module and $B$-module respectively, we define the completed tensor product $$M\hat{\otimes}_{\Oh}N := \varprojlim_{a,b}M/\mathfrak{m}_A^a\otimes_{\Oh}N/\mathfrak{m}_B^{b}.$$ 

\begin{prop}\label{prop_completed_tensor_product_ideal}
    Let $(A,\mathfrak{m}_A)$ and $(B,\mathfrak{m}_B)$ be two objects of $\mathcal{C}_\Oh.$ Let $\mathfrak{a}$ be an ideal of $A.$ Then we have the short exact sequence $$0\to \mathfrak{a}(A\widehat{\otimes}_\Oh B)\to A\widehat{\otimes}_\Oh B \to (A/\mathfrak{a})\widehat{\otimes}_\Oh B\to 0.$$
\end{prop}

\begin{proof}
    It suffices to show that $$(A\widehat{\otimes}_\Oh B)\otimes_AA/\mathfrak{a} \xrightarrow{\sim} (A/\mathfrak{a})\widehat{\otimes}_\Oh B.$$ Unravelling the definition of completed tensor product, we see that this is equivalent to showing $$\left(\varprojlim_{m,n} A/\mathfrak{m}_A^m\otimes_\Oh B/\mathfrak{m}_B^n \right)\otimes_A A/\mathfrak{a}\xrightarrow{\sim} \varprojlim_{m,n} \left(A/(\mathfrak{a} + \mathfrak{m}_A^m) \otimes B/\mathfrak{m}_B^n\right).$$ Since $A/\mathfrak{m}_A^m\otimes_\Oh B/\mathfrak{m}_B^n$ are finite length $A$-modules for all $m,n$, this follows from Lemma \ref{lem_inverse_limit_tensor_product}. 
\end{proof}

The following lemma is part of \cite[Lemma 3.3]{Barnet-Lamb_Geraghty_Harris_Taylor_2011_Calabi_Yau}.
\begin{lem}[Taylor]\label{lem_Taylor}
    \leavevmode
    \begin{enumerate}
        \item If $A$ is an object of $\mathcal{C}_{\Oh}$ and $\mathfrak{p}$ is a maximal ideal of $A[1/p],$ then the residue field $A[1/p]/\mathfrak{p}$ is a finite extension over $L$ and the image of $A$ in $L$ is an order. Furthermore, $A[1/p]$ is a Jacobson ring.
        \item If $A$ and $B$ are geometrically irreducible $\F$-algebras in $\mathcal{C}_\Oh,$ then $A\widehat{\otimes}_\F B$ is also geometrically irreducible. 
    \end{enumerate}
\end{lem}

We also have the following lemma \cite[Lemma 3.4.12]{kisin2009moduli}.
\begin{lem}[Kisin] \label{lem_kisin_completed_tensor_product}
    If $A$ and $B$ be two objects of $\mathcal{C}_\Oh$ that are flat over $\Oh,$ so is $A\hat{\otimes}_{\Oh} B.$
    
\end{lem}

\begin{prop} \label{prop_completed_tensor_product_in_char_p}
  Let \((A,\mathfrak{m}_A),(B,\mathfrak{m}_B)\in\mathcal C_{\Oh}\) be \(\F\)-algebras. \begin{enumerate}
    \item If $A\hookrightarrow B$ is an injective morphism, then the induced map \(A\widehat\otimes_{\F}C \to B\widehat\otimes_{\F}C\) is injective for every \(\F\)-algebra \((C,\mathfrak{m}_C)\) in \(\mathcal{C}_{\Oh}\). 
    \item If $A$ and $B$ are Cohen--Macaulay (resp. local complete intersections), then the completed tensor product \(A\widehat\otimes_{\F}B\) is Cohen--Macaulay (resp. local complete intersections).
    \item If the completed tensor product \(A\widehat\otimes_{\F}B\) is Cohen--Macaulay and \(B\) itself is Cohen--Macaulay, then \(A\) is Cohen--Macaulay.
  \end{enumerate}
\end{prop}

\begin{proof}
    
\begin{enumerate}
    \item Since every $\F$-algebra is $\F$-flat, the induced map $$A\otimes_\F C/\mathfrak{m}_C^i\to B\otimes_\F C/\mathfrak{m}_C^i$$ remains injective for every positive integer $i.$ Passing to the limit and using Lemma \ref{lem_inverse_limit_exactness}, we deduce that $$A\widehat{\otimes}_\F C \to B\widehat{\otimes}_\F C$$ is still injective.
\end{enumerate}
For the rest of the proof, we let $a_1,\ldots,a_s$ and $b_1,\ldots,b_t$ be systems of parameters of $A$ and $B$, respectively. Their union is a system of parameters for $A\widehat{\otimes}_\F B.$ 
\begin{enumerate}
    \item [(2)] If $A$ is Cohen--Macaulay, then the sequence $a_1,\ldots,a_s$ is regular. Thus $A\xrightarrow{\cdot a_1}A$ is an injection. By ~(1), we have that $$A\widehat{\otimes}_\F B\xrightarrow{\cdot a_1}A\widehat{\otimes}_\F B$$ remains injective and so $a_1$ is a non-zero-divisor in $A\widehat{\otimes}_\F B.$ By induction on the length of the parameter sequence, we deduce that $a_1,\ldots,a_s$ is regular in $A\widehat{\otimes}_\F B.$ Hence, it suffices to prove that $b_1,\ldots,b_t$ is a regular sequence in $A/(a_1,\ldots,a_s)\otimes_\F B.$ But then the same reasoning applies to $A/(a_1,\ldots,a_s)$ and $B$ and we deduce that $b_1,\ldots,b_t$ is a regular sequence in $B.$ Thus $A\widehat{\otimes}_\F B$ is Cohen--Macaulay. The complete intersection case follows from choosing presentations of $A$ and $B.$ 
    \item [(3)] Again by ~(1), the inclusion \(\F\hookrightarrow B\) induces an injective map 
  \[
    A \xrightarrow{\sim} A\widehat{\otimes}_\F\F\hookrightarrow A\widehat{\otimes}_\F B .
  \]
  If \(A\widehat{\otimes}_\F B\) is \(a_s\)-torsion-free, then \(a_s\) is a non-zero-divisor in \(A\). Now we apply the same argument to $A/(a_s).$ Note that $A/(a_s)\widehat{\otimes}_\F B \xrightarrow{\sim} A\widehat{\otimes}_\F B / (a_s)$ is $a_{s-1}$-torsion-free. We conclude that $a_{s-1}$ is a non-zero-divisor in $A/(a_s)$ and hence $a_{s-1},a_s$ is a regular sequence in $A.$ Repeating the argument, we see that $a_1,\ldots,a_{s}$ is a regular sequence in $A,$ proving that $A$ is Cohen--Macaulay. 
\end{enumerate}
\end{proof}

\begin{prop} \label{prop_strong_top_inverse_limit_equivalence}
    Let $(R,\mathfrak{m})$ be an object in $\mathcal{C}_\Oh$, and let $\{J_n\}_{n\ge 1}$ be a descending sequence of ideals of $R$. The following are equivalent:
    \begin{enumerate}
        \item $R$ is complete with respect to the topology defined by the ideals $\{J_n\}_{n\ge 1}$.
        \item The $\{J_n\}_{n\ge 1}$-topology is stronger than the $\mathfrak{m}$-adic topology; that is, for every $k \ge 1$, there exists $n \ge 1$ such that
        \[
            J_n \subseteq \mathfrak{m}^k.
        \]
    \end{enumerate}
\end{prop}

\begin{proof}
Suppose that (1) holds, i.e., $R$ is complete with respect to the $\{J_n\}$-topology. To show (2), it suffices to prove that
\[
    R/\mathfrak{m}^k \xrightarrow{\sim} \varprojlim_n R/(J_n + \mathfrak{m}^k)
\]
for every $k \ge 1$. Indeed, since $R/\mathfrak{m}^k$ is a finite Artinian ring, this isomorphism implies that for $n$ large enough, 
\[
    R/\mathfrak{m}^k \xrightarrow{\sim} R/(J_n + \mathfrak{m}^k),
\] 
and therefore $J_n \subseteq \mathfrak{m}^k$ for $n$ sufficiently large. Consider the short exact sequence of inverse systems:
\[
    0 \to \{\mathfrak{m}^k(R/J_n)\}_n \to \{R/J_n\}_n \to \{R/(J_n+\mathfrak{m}^k)\}_n \to 0.
\]
Since $\{J_n\}_n$ is descending, the system $\{\mathfrak{m}^k(R/J_n)\}_n$ is surjective. By Lemma~\ref{lem_inverse_limit_exactness}, taking inverse limits preserves exactness, giving
\[
    0 \to \varprojlim_n \mathfrak{m}^k(R/J_n) \to \varprojlim_n R/J_n \to \varprojlim_n R/(J_n+\mathfrak{m}^k) \to 0.
\]
By (1), $\varprojlim_n R/J_n \simeq R$, so we obtain
\[
    R / \varprojlim_n \mathfrak{m}^k(R/J_n) \xrightarrow{\sim} \varprojlim_n R/(J_n + \mathfrak{m}^k).
\]
It remains to show that 
\[
    \mathfrak{m}^k \to \varprojlim_n \mathfrak{m}^k(R/J_n)
\]
is surjective. Let 
\[
    0 \to N \to R^s \to \mathfrak{m}^k \to 0
\]
be a presentation of $\mathfrak{m}^k$ for some integer $s \ge 1$, with $N \subseteq R^s$. Equip $N$ with the subspace topology induced by $\{N \cap J_n R^s\}_n$. Then, for each $n$, we have a short exact sequence
\[
    0 \to N / (N \cap J_n R^s) \to R^s / J_n R^s \to \mathfrak{m}^k(R/J_n) \to 0.
\]
Again by Lemma~\ref{lem_inverse_limit_exactness}, the corresponding sequence of inverse limits is exact:
\[
    0 \to \varprojlim_n N / (N \cap J_n R^s) \to \varprojlim_n (R/J_n)^s \to \varprojlim_n \mathfrak{m}^k(R/J_n) \to 0.
\]
We then obtain a commutative diagram:
\[
\begin{tikzcd}
    R^s \arrow[r,two heads] \arrow[d,"\wr"] & \mathfrak{m}^k \arrow[d] \\
    \varprojlim_n (R/J_n)^s \arrow[r,two heads] & \varprojlim_n \mathfrak{m}^k(R/J_n)
\end{tikzcd}
\]
where the left vertical map is an isomorphism. It follows that the right vertical map is surjective, as desired. Hence, (1) implies (2).

Conversely, suppose that (2) holds. By passing to a subsequence if necessary, we may assume that
\[
    J_n \subseteq \mathfrak{m}^n \quad \text{for all } n \ge 1.
\]
We then have natural maps
\[
    R \longrightarrow \varprojlim_n R/J_n \longrightarrow \varprojlim_n R/\mathfrak{m}^n \xrightarrow{\sim} R,
\]
where the last isomorphism holds because $R$ is an object in $\mathcal{C}_\Oh$. It follows that 
\[
    \varprojlim_n R/J_n \longrightarrow \varprojlim_n R/\mathfrak{m}^n
\] 
is surjective. Thus, it suffices to show injectivity. Let $(a_n)_{n \ge 1}$ be a coherent sequence in $\varprojlim_n R/J_n$. For every $n_0 \ge 1$, there exists $m_0 \gg 0$ such that
\[
    a_{m_0+k} - a_{m_0} \in J_{n_0} \quad \text{for all } k \ge 1.
\]
Thus 
\[
    a_{m_0+k} \in a_{m_0} + J_{n_0}.
\] Suppose $(a_n)$ maps to $0$ in $\varprojlim_n R/\mathfrak{m}^n$. Taking the limit as $k \to \infty$, we obtain
\[
    0 \in a_{m_0} + \overline{J}_{n_0},
\]
where $\overline{J}_{n_0}$ denotes the closure of $J_{n_0}$ in the $\mathfrak{m}$-adic topology. On the other hand, since $R/\mathfrak{m}^n$ have finite length, Lemma~\ref{lem_inverse_limit_tensor_product} implies
\[
    R/J_{n_0} \xrightarrow{\sim} (\varprojlim_n R/\mathfrak{m}^n) \otimes_R R/J_{n_0} \xrightarrow{\sim} \varprojlim_n R/(\mathfrak{m}^n + J_{n_0}),
\]
so $R/J_{n_0}$ is Hausdorff. In particular, $J_{n_0}$ is closed in the $\mathfrak{m}$-adic topology, and we conclude that for every $n_0\ge 1,$ there exists $m_0$ such that 
\[
    a_{m_0} \in \overline{J}_{n_0} = J_{n_0}.
\]
It follows that there exists a subsequence of $\{a_n\}$ converging to $0$ in the $\{J_n\}$-topology. But $(a_n)$ is a Cauchy sequence in the $\{J_n\}$-topology, so it must already be $0$. This proves that the map is injective, and therefore (2) implies (1).
\end{proof}

\subsection{Modular Forms}
In this subsection, we begin by recalling the definition of modular forms as the zero-th coherent cohomology of modular curves, together with the Hecke action on their $q$-expansions. We then review mod-$p$ modular forms and Serre's modular forms, since these play a key role in the proofs of the main theorems outlined in \S\ref{subsubsection_proof_of_main_theorems}. Finally, we discuss how to fix the Nebentypus character of modular forms, as this is what will be required in the patching process.

\subsubsection{Modular Forms} \label{subsubsec_modular_forms}
Let $N$ be a positive integer and let $r\ge 5$ be a divisor of $N.$ We follow the notion of a general level structure introduced in \cite[IV.3]{deligne_rapoport_1973}. Let $U$ be an open subgroup of $\GL_2(\hat{\Z})$ that contains $$U(N) := \ker\left(\GL_2(\hat{\Z})\surj \GL_2(\Z/N\Z)\right)$$ and is contained in $$U_1(r):=\left\{\begin{pmatrix}
    a & b\\
    c & d
    \end{pmatrix}\in \GL_2(\hat{\Z}): \begin{pmatrix}
a & b\\
c & d
\end{pmatrix} \equiv \begin{pmatrix}
* & *\\
0 & 1
\end{pmatrix}\pmod r\right\}$$ for some integer $r\ge 5.$ Let $S$ be a scheme over $\Z[1/N]$ and let $E$ be an elliptic curve over $S.$ A $U$-level structure is the equivalence class $[\alpha]$ of $S$-isomorphisms $\alpha$ identifying the $N$-torsion $E[N]$ of $E$ with the constant group scheme $\left(\underline{\Z/N\Z}\right)^2_S$ over $S$ where two isomorphisms $\alpha$ and $\alpha'$ are equivalent if \'etale locally on $S$, there is some $g\in U$ such that $\alpha' = g\circ\alpha$ (with the right action of $U$ on $\left(\Z/N\Z \right)^2_S$). Due to our assumption on $U,$ the functor that sends a $\Z[1/N]$-scheme $S$ to the set of isomorphism classes of pairs $(E,[\alpha])$ of elliptic curves $E$ over $S$ and $U$-level structures is represented by an affine smooth curve $Y_U$ over $\Z[1/N]$. We let $X_U$ denote the compactification of $Y_U$ by adjoining cusps; it is proper smooth over $\Z[1/N]$. The curve $X_U$ may not be geometrically connected. Its connected components are in bijection with $\hat{\Z}^{\times}/\det U.$ For a $Z[1/N]$-algebra $R$, we write $X_{U,R}$ for the base-change of $X_U$ along $\Z[1/N]\to R.$

Let $\pi_U:\mathcal{E}_U\to X_U$ be the universal generalized elliptic curve over $X_U$ with zero section $$e_U:X_U\to \mathcal{E}_U.$$ Let $\omega_{U} := e^*_{U}\Omega^1_{\mathcal{E}_U/X_{U}}.$ For a $\Z[1/N]$-algebra $R,$ the subscript $R$ indicates that the corresponding object (e.g., $X_{U,R}$ or $\omega_{U,R}$) is obtained by base-changing along $Z[1/N]\to R.$ We define a weight $k$ level $U$ modular form over $R$ to be an element of $H^0(X_{U,A},\omega_{U,R}^{\otimes k})$. Let $\infty$ be the reduced divisor supported on the cusps of $X_{U}.$ By a weight $k$ level $U$ cusp form over $R$, we mean an element of $H^0(X_{U,R},\omega_{U,R}^{\otimes k}(-\infty)).$ In most of the situations considered in this paper, the ring $R$ will simply be either $\Oh$ or $\F.$

The following base-change theorem due to Katz and Mazur.

\begin{thm}[Katz, Mazur]\label{thm_base_change}
    If $R$ is flat over $\Z[1/N]$ or if $k\ge 2,$ then the natural map $$H^0(X_{U},\mathscr{L})\otimes_{\Z[1/N]}R\to H^0(X_{U,R},\mathscr{L}_R)$$ is an isomorphism where $\mathscr{L}$ is $\omega^{\otimes k}_U$ or $\omega^{\otimes k}_U(-\infty).$
\end{thm}

We let $U = U_1(N)$ for the discussion of the $q$-expansion of a modular form $f$ of weight $k$ and level $U_1(N)$ that is defined over a $\Z[1/N]$-algebra $R.$ In this case, the curve $X_{U_1(N)}$ is geometrically connected. The $U_1(N)$-level structure is a choice of an order $N$-point for an elliptic curve $E/S,$ which is the same as fixing an embedding of group schemes $\mu_{N,S}\hookrightarrow E[N],$ which we still denote by $\alpha$ by abuse of notation. Denote by $\Tate(q)$ the Tate curve ${\Gm}_{,t}/q^{\Z}$, which is a generalized elliptic curve over $\Z\llbracket q\rrbracket$, and fix the natural embedding $$\id_N : \mu_N\hookrightarrow \Tate(q)[N].$$ By evaluating a modular form $f$ at the $\infty$ cusp $(\Tate(q),\id_N)$ over $\Z[1/N]\llbracket q\rrbracket$, we obtain the the $q$-expansion of $f$ given by $$f(\Tate(q),\id_N) = \left(\sum_n a_n(f)q^n\right)\left(\frac{dt}{t}\right)^{\otimes k},$$ which defines a homomorphism $$\phi_{\infty,R}: H^0(X_{U_1(N),R},\omega_{U_1(N),R}^{\otimes k})\to R\llbracket q\rrbracket.$$ By abuse of notation, we write $f(q)$ for $\phi_{\infty,R}(f)$ and write $a_n(f)$ for the coefficient of $q^n$ in the $q$-expansion of $f.$ 
One can also evaluate a modular form at other cusps by choosing another level structure on $\Tate(q)$, as long as $\zeta_N\in R.$ By definition, the cusp forms are the ones whose $q$-expansion at all cusps do not have constant terms. In particular, they are mapped into $qR\llbracket q\rrbracket$ under $\phi_{\infty,R}.$ The following $q$-expansion principle is proved by Katz in \cite[Theorem 1.6.1, Corollary 1.6.2]{katz1973p} so that we actually only need to consider one cusp. 

\begin{thm}[Katz]\label{thm_q_expansion_principle}\leavevmode
    \begin{enumerate}
        \item The $q$-expansion homomorphism $\phi_{\infty,R}$ is injective.
        \item If $R_0$ is a $\Z[1/N]$ sub-algebra of $R,$ and if $f$ is a weight $k$ level $U_1(N)$ modular form over $R$ whose $q$-expansion has coefficients in $R_0,$ then $f$ is a modular form over $R_0.$
    \end{enumerate}
\end{thm}

\subsubsection{Hecke Action} \label{subsubsec_Hecke_action}
Let $A$ be a $\Z[1/N]$-algebra and let $d\in (\Z/N\Z)^{\times}.$ The diamond operator $\langle d\rangle$ is defined to be $$\left(f\langle d\rangle\right)(E,\alpha) := f(E,d\alpha)$$ for $f\in H^0(X_{U_1(N),A},\omega^{\otimes k}_{U_1(N),A})$ where $E$ is a generalized elliptic curve and $\alpha$ a level-structure. 

The Hecke operators $T_\ell$ if $\ell\nmid N$ and $U_\ell$ if $\ell\mid N$ are defined in \cite[\S 3]{Gross1990Tameness} and their effects on the $q$-expansions are also calculated there which we now recall. Let $f(q) = \sum_{n=0}^{\infty}a_nq^n$ be the $q$-expansion of $f\in H^0(X_{U_1(N)},\omega^{\otimes k})$ and let $(f\langle \ell\rangle)(q) = \sum_{n=0}^{\infty}b_nq^n$ be the $q$-expansion of $f\langle \ell\rangle$ when $\ell\nmid N,$ Then we have $$\left(f T_\ell\right)(q) = \sum_{n=0}^\infty a_{n\ell}q^n + \ell^{k-1}\sum_{n=0}^\infty b_{n/\ell}q^n\quad \ell\nmid N $$ and $$\left(f U_\ell\right)(q)=\sum_{n=0}^\infty a_{n\ell}q^n\quad \ell|N.$$ 

For $R=\Oh$ or $L$, the ring $R$ is flat over $\Z[1/N]$, so the Hecke operators 
$T_\ell$ and $U_\ell$ are defined on
\[
H^0(X_{U_1(N),R},\omega^{\otimes k}_{U_1(N),R})
\]
by the base-change theorem~\ref{thm_base_change}. 
When $R=\F$, the same construction applies provided $k\ge 2$. In this case we write 
$T_p$ as $U_p$ to reflect its effect on $q$-expansions. When $k=1$, the 
base-change theorem does not apply directly, but the operators are defined 
in~\cite[\S 4]{Gross1990Tameness}, and they satisfy the same $q$-expansion formulas.

Let $\T$ be the polynomial ring over $\Oh$ in variables $T_n$, where $n$ ranges
over all positive integers. For a positive integer $M$, let $\T^M$ denote the
subring of $\T$ generated by the variables $T_n$ with $(n,M)=1$. Let 
$R=\Oh, L,$ or $\F$. Then $\T$ acts on
\[
H^0(X_{U_1(N),R},\omega^{\otimes k}_{U_1(N),R}(-\infty))
\]
as follows. The variable $T_1$ acts as the identity. If a prime $\ell\nmid Np$, then
$T_\ell\in\T$ acts via the Hecke operator $T_\ell$; if $\ell\mid N$, then
$T_\ell$ acts via the operator $U_\ell$. When $\ell=p$ and $k\ge 2$, the variable $T_p$ acts by the operator $T_p$, which in characteristic $p$ is denoted $U_p$ because of its effect on $q$-expansions. Suppose the action of the variables $T_{\ell^t}$ has been defined for integers
$1\le t\le s$. Then the action of $T_{\ell^{s+1}}$ is given recursively by
\[
T_{\ell^{s+1}} = T_\ell T_{\ell^s} - \ell^{k-1}\langle \ell\rangle\, T_{\ell^{s-1}}.
\]
More generally, if $(m,n)=1$, then
\[
T_{mn} = T_m T_n.
\] The following proposition can be found in {\cite[Proposition 5.3.1]{diamondshurman2005first}}.

\begin{prop}\label{prop_a1Tnf}
    For every cusp form $f\in H^0(X_{U_1(N),R},\omega^{\otimes k}_{U_1(N),R}(-\infty)),$ we have $$a_1(fT_n) = a_n(f).$$
\end{prop}

\subsubsection{Classical Modular Forms}
When $R=L$, we may view $L$ as a subfield of the complex numbers $\C$, and the theory of modular forms over $L$
coincides with the classical theory,
provided $L$ is sufficiently large. We briefly recall some standard facts
from the Atkin--Lehner theory of classical modular forms; see
\cite[\S 5.8]{diamondshurman2005first} for details.

We call $f\in H^0(X_{U_1(N),L},\omega^{\otimes k}_{U_1(N),L}(-\infty))$ an
\emph{eigenform} if it is an eigenvector for $T_\ell$ and $\langle \ell\rangle$
for all primes $\ell\nmid N$. It then follows that $f$ is an eigenvector for all
$T_n$ with $\gcd(n,N)=1$. We define $\lambda_f(n)\in \Oh$ by
\[
f T_n = \lambda_f(n)\, f
\] for all such $n.$
Since the Hecke actions are defined integrally, the eigenvalues $\lambda_f(n)$
are algebraic integers over $\Z[1/N]$; as usual, we enlarge $\Oh$ so that all
of them lie in $\Oh$.

By \cite[Proposition~5.8.4]{diamondshurman2005first}, there exists a unique
newform $g$ of level $N_g\mid N$ whose Hecke eigenvalues away from $N$ agree
with those of $f$. We call this $g$ the \emph{newform associated to $f$}. We
say that $f$ is \emph{new at} a prime $q\mid N$ if $v_q(N)=v_q(N_g)$;
otherwise, we say that $f$ is \emph{old at} $q$. The following theorem summarizes standard properties of newforms; see \cite[\S 5.8]{diamondshurman2005first}.
\begin{thm}[Strong multiplicity one]\label{thm_strong_multiplicity_one}
    Let $g = \sum_{n=1}^\infty a_n(g)q^n$ be a newform of level $N_g\ge 1.$ 
    Then it is an eigenvector for the full Hecke algebra $\T$ it satisfies $$\lambda_g(n) = a_n(g)$$ for every integer $n\ge 1.$
\end{thm}

For a positive integer $d$, there is an operator
\[
V_d: H^0(X_{U_1(N),L},\omega^{\otimes k}_{U_1(N),L})
   \to H^0(X_{U_1(Nd),L},\omega^{\otimes k}_{U_1(Nd),L})
\]
whose effect on $q$-expansions is given by
\[
(fV_d)(q) := \sum_{n=0}^{\infty} a_{n/d}(f)\, q^n.
\]
Using the $q$-expansion formula, one checks that $V_d$ commutes with $T_\ell$
and $\langle \ell\rangle$ for all primes $\ell\nmid Nd$, and that
\[
V_q U_q = \id.
\]
Moreover, we may formally write
\[
T_\ell = U_\ell + \ell^{k-1}\langle \ell\rangle V_\ell.
\]

\begin{thm}[{\cite[Theorem 5.8.3]{diamondshurman2005first}}]\label{thm_oldforms_as_newform_span}
    Assume that $L$ is large enough. 
    For each newform $g\in H^0(X_{U_1(N),L},\omega^{\otimes k}_{U_1(N),L}(-\infty))$, let
\[
E_g := H^0(X_{U_1(N),L},\omega^{\otimes k}_{U_1(N),L}(-\infty))[\lambda_g],
\]
the eigenspace of $H^0(X_{U_1(N),L},\omega^{\otimes k}_{U_1(N),L}(-\infty))$ corresponding to the system of Hecke eigenvalues of $g$ away from $N$.  
    Then we have a decomposition $$H^0(X_{U_1(N),L},\omega^{\otimes k}_{U_1(N),L}(-\infty)) \xrightarrow{\sim} \bigoplus_{\substack{g \text{ newform}\\ g\in H^0(X_{U_1(N),L},\omega^{\otimes k}_{U_1(N),L}(-\infty))}} E_g.$$ And each $E_g$ has a basis consisting of eigenforms $$\{ g V_d : d N_g \mid N \}.$$ In particular, if $f$ is an eigenform of weight $k$ and level $U_1(N)$, then it is a linear combination of forms $gV_d$ where $g$ is the newform associated to $f$ and $d$ runs over divisors of $N/N_g.$
\end{thm}

\begin{prop}\label{prop_eigenform_new_old_at_q}
Let $f$ be an eigenform of weight $k$ and level $U_1(N)$, and let $g$ be the newform associated to $f$ with Nebentypus character $\chi_g$.
\begin{enumerate}
    \item If $f$ is new at a prime $q \mid N$, then $f$ is an eigenvector of $U_q$ with eigenvalue $a_q(g)$.
    \item If $q \parallel N$ and $f$ is old at $q$, then $f$ is annihilated by
    \[
        U_q^2 - a_q(g) U_q + \chi_g(q) q^{k-1}.
    \]
\end{enumerate}
\end{prop}

\begin{proof}
    Since $f$ is new at a prime $q \mid N$, we have $q \mid N_g$ and $\gcd(q, N/N_g) = 1$.  
It follows by Theorem \ref{thm_strong_multiplicity_one} that $g$, the newform associated to $f$, is an eigenvector for $U_q$ (rather than $T_q$) with eigenvalue $a_q(g)$.  
Moreover, for any $d \mid N/N_g$, we have $q \nmid d$, so $U_q$ commutes with $V_d$.  
Hence, for each summand $g V_d$ in the decomposition of $f$, we have
\[
g V_d U_q = g U_q V_d = a_q(g) g V_d,
\]
and therefore $f$ itself is an eigenvector of $U_q$ with eigenvalue $a_q(g)$ by Theorem \ref{thm_oldforms_as_newform_span}.

For the second part, since $q \parallel N/N_g$, we can write
\[
f = h_1 + h_2 V_q
\]
for some eigenforms $h_1$ and $h_2$ of level $U_1(N/q)$ for $\T^{N/q}$, whose associated newform is also $g$.  Using the formulas
\[
T_q = U_q + q^{k-1} \langle q \rangle V_q, \quad
h_i \langle q \rangle = \chi_g(q) h_i, \quad
T_q h_i = a_q(g) h_i, \quad
V_q U_q = \id,
\]  
we conclude
\[
h_1 \,(U_q^2 - a_q(g) U_q + \chi_g(q) q^{k-1}) = 0, \quad
h_2 V_q \,(U_q^2 - a_q(g) U_q + \chi_g(q) q^{k-1}) = 0.
\]  
Therefore,
\[
f \,(U_q^2 - a_q(g) U_q + \chi_g(q) q^{k-1}) = (h_1 + h_2 V_q) \,(U_q^2 - a_q(g) U_q + \chi_g(q) q^{k-1}) = 0.
\]
\end{proof}

\subsubsection{\texorpdfstring{Mod-$p$ Modular Forms}{Mod-p Modular Forms}}
In this part, 
we recall facts on modular forms over $\F$ that we will later use.  

In characteristic $p,$ we have the Hasse invariant $A \in H^0(X_{U_1(N),\F},\omega^{\otimes p-1}_{U_1(N),\F})$, whose $q$-expansion at every cusp is constantly $1.$ 
Multiplication by $A$ induces an injection $$H^0(X_{U_1(N),\F},\omega^{\otimes k-(p-1)}_{U_1(N),\F})\hookrightarrow H^0(X_{U_1(N),\F},\omega^{\otimes k}_{U_1(N),\F})$$ that restricts to an injection on the subspace of cusp forms. 
The filtration $w(f)$ of a mod-$p$ modular form $f$ is the smallest integer $k_0$ such that there is a mod-$p$ modular form $f_0$ such that $A^{n} f_0 = f.$ We necessarily have $k\equiv k_0\mod p-1.$ In fact, the kernel of $$\bigoplus_{k\ge 0}H^0(X_{U_1(N),\F},\omega^{\otimes k}_{U_1(N),\F})\to \overline{\F}_p\llbracket q\rrbracket$$ is exactly $A-1.$ By checking the effect of Hecke operators on $q$-expansions of mod-$p$ modular forms, we see that multiplication by the Hasse invariant $A$ commutes with all the Hecke operators except when $k=1$ and the Hecke operator is $T_p.$  The filtration satisfies $w(fT)\le w(f)$ for $f$ a mod-$p$ modular form and $T$ a Hecke operator.

We can decompose 

In characteristic $p$, there is an operator
\[
V_p : H^0(X_{U_1(N),\F}, \omega_{U_1(N),\F}^{\otimes k}) \to H^0(X_{U_1(N),\F}, \omega_{U_1(N),\F}^{\otimes pk}),
\]
whose effect on $q$-expansions is
\[
(f V_p)(q) = \sum_{n=0}^\infty a_n(f) q^{pn} \quad \text{if} \quad f(q) = \sum_{n=0}^{\infty} a_n(f) q^n.
\]
Unlike the $V_p$ operator in characteristic $0$, the operator in characteristic $p$ does not change the level structure.  
From now on, $V_p$ always refers to this characteristic $p$ operator.  
It preserves cusp forms and commutes with Hecke operators away from $p$.

\begin{prop}[{\cite[Fact 1.7]{Jochnowitz1982Congruences}}] \label{prop_Vp_filtration}
    We have $$w(fV_p) = p w(f)\quad{\rm and}\quad V_pU_p = \id.$$ In particular, $V_p$ is injective. 
\end{prop}

When $k \ge 2$ and $R = \Oh$ or $\F$, the $R$-module 
$H^0(X_{U_1(N),R}, \omega_{U_1(N),R}^{\otimes k})$ decomposes under the action of the $T_p$ or $U_p$ operator into an ordinary part, 
\[
H^0(X_{U_1(N),R}, \omega_{U_1(N),R}^{\otimes k})^\ord,
\] 
where $T_p$ (or $U_p$) is invertible, and a non-ordinary part, 
\[
H^0(X_{U_1(N),R}, \omega_{U_1(N),R}^{\otimes k})^\no,
\] 
where $T_p$ (or $U_p$) acts nilpotently.

\begin{prop} \label{prop_Vp_preserves_ordinary_nonordinary forms}
    When $k\ge 2,$ $V_p$ preserves ordinary forms and non-ordinary forms in characteristic $p,$ i.e., for $*\in \{\ord,\no\},$ $$V_p : H^0(X_{U_1(N),\F}, \omega_{U_1(N),\F}^{\otimes k})^* \to H^0(X_{U_1(N),\F}, \omega_{U_1(N),\F}^{\otimes pk})^*.$$  
\end{prop}

\begin{prop} \label{prop_Vp_preserves_ordinary_nonordinary_forms}
    Let $k \ge 2$. The operator $V_p$ preserves the ordinary and non-ordinary subspaces in characteristic $p$. More precisely, for $* \in \{\ord, \no\}$, we have
    \[
        V_p : H^0(X_{U_1(N),\F}, \omega_{U_1(N),\F}^{\otimes k})^* 
            \longrightarrow 
        H^0(X_{U_1(N),\F}, \omega_{U_1(N),\F}^{\otimes pk})^*.
    \]
\end{prop}

\begin{proof}
Let $f$ be an ordinary form. For $n \gg 1$, we have
\[
(f V_p) U_p^n = f (V_p U_p) U_p^{n-1} = (A^k f) U_p^{n-1},
\]
which is still ordinary. The same reasoning applies to non-ordinary forms.
\end{proof}

In \cite{katz1977result}, Katz constructed the $\theta$-operator $$\theta: H^0(X_{U_1(N),\F},\omega^{\otimes k}_{U_1(N),\F})\to H^0(X_{U_1(N),\F},\omega^{\otimes k+p+1}_{U_1(N),\F})$$ for a general level structure whose effect on the $q$-expansion is given by 
$$(\theta f)(q) = \sum_{n=0}^{\infty}na_nq^n \quad{\rm if}\quad f(q)=\sum_{n=0}^{\infty}a_nq^n.$$ 

From now on, whenever we use $V_p$ and $\theta,$ we view them as operators on the space of $q$-expansions. Katz proved the following in \cite{katz1977result}.
\begin{thm}[Katz]\label{thm_katz_theta}
    The kernel of $\theta$ is exactly the image of $V_p.$
\end{thm}

\subsubsection{A duality theorem}
In this part, we recall a duality theorem for the coherent cohomology of the modular curves using $q$-expansions so that we can identify the dual of modular forms with the full Hecke algebra. Following the idea of the theorem, we study the difference between the Hecke algebras $\T^p$ and $\T$ which will be used to determine when $R_\rbar^{\Box,\psi_k}$ and $R_\rbar^{\Box,\psi_k}[\alpha_p]$ differ as mentioned in \S\ref{subsubsection_proof_of_main_theorems} . In this part, we assume the weight $k\ge 2.$

Given an $R$-submodule $M$ of $H^0(X_{U_1(N),R},\omega^{\otimes k}_{U_1(N),R}(-\infty))$, we write $\T(M)$ (resp. $\T^{N}(M)$ and $\T^{Np}(M)$) for the image of $\T$ (resp. $\T^N$ and $\T^{Np}$) in $\End_R(M).$ We have the following theorem first proved in \cite[Theorem 1.6, Theorem 2.2]{Ribet83Mod}:
\begin{thm}[Ribet]\label{thm_ribet_duality_full}
    Suppose $k\ge 2.$ The pairing \begin{align*}
        \T(M) \times M &\to R\\
        (T,f) &\mapsto a_1(Tf)
    \end{align*} is Hecke equivariant and perfect. 
\end{thm}

\begin{cor}\label{cor_finite_level_mod_p_commute}
    Suppose $k\ge 2.$ Let $M$ be a $\T$-submodule of $H^0(X_{U_1(N),\Oh},\omega^{\otimes k}_{U_1(N),\Oh}(-\infty))$ that satisfies the base-change theorem \ref{thm_base_change}. 
    Then the Hecke algebra $\T(M\otimes_{\Oh}\F)$ is isomorphic to $\T(M)\otimes_{\Oh}\F.$
\end{cor}

\begin{proof}
    We have $$\T(M\otimes_\Oh\F) \xrightarrow{\sim}\Hom_{\F}(M\otimes_{\Oh}\F,\F)\xrightarrow{\sim}\Hom_{\Oh}(M,\Oh)\otimes_{\Oh}\F\xrightarrow{\sim}\T(M)\otimes_{\Oh}\F.$$ Note that the condition on $M$ guarantees that the first isomorphism is compatible with the Hecke action on both sides.
\end{proof}

\begin{remark} \label{rmk_Hecke_mod_p_surjects_onto_Hecke_of_mod_p_mod_forms}
    If we replace $\T$ by $\T^{N'}$ for some integer $N',$ then we still have a $\T^{N'}$-equivariant surjection $$\T^{N'}(M)\otimes_{\Oh}\F\surj \T^{N'}(M\otimes_{\Oh}\F).$$ But this map can have nontrivial kernel when $N'>1.$
\end{remark}

\begin{cor}\label{cor_finite_level_surjection}
    Let $M_1$ (resp. $M_2$) be a $\T$-submodule of $H^0(X_{U_1(N),\Oh},\omega^{\otimes k_1}_{U_1(N),\Oh}(-\infty))$ (resp. $H^0(X_{U_1(N),\Oh},\omega^{\otimes k_2}_{U_1(N),\Oh}(-\infty))$) that satisfies the base-change theorem \ref{thm_base_change}. 
    If $$A^{n}(M_1\otimes_{\Oh}\F)\subseteq M_2\otimes_{\Oh}\F$$ for some integer $n\ge 0,$ then there is a $\T$-equivariant surjection $$\T(M_2\otimes_{\Oh}\F)\surj \T(M_1\otimes_{\Oh}\F).$$ The same conclusion also holds if we replace $\T$ by $\T^{N'}$ if $N'$ is a divisor of $Np.$
\end{cor}

\begin{proof}
    The second assertion follows from the first because the surjection is $\T$-equivariant. The first assertion holds by identifying $\T(M_i\otimes_{\Oh}\F)$ with $\Hom_{\F}(M_i\otimes_{\Oh}\F,\F)$ for $i=1,2$ and that multiplication by $A$ is $\T$-equivariant when the weights are at least 2. 
\end{proof}

\begin{remark}
    In the context of Corollary \ref{cor_finite_level_surjection}, we can invoke Corollary \ref{cor_finite_level_mod_p_commute} to conclude the existence of a $\T$-equivariant surjection $$\T(M_2)\otimes_{\Oh}\F\surj \T(M_1)\otimes_{\Oh}\F.$$ 
\end{remark}

We now develop an analogous duality result for the Hecke algebra $\T^{p}.$ Consider the $q$-expansion homomorphism $$\phi_{\infty,L}:H^0(X_{U_1(N),L},\omega^{\otimes k}_{U_1(N),L}(-\infty))\to L\llbracket q\rrbracket. $$ Inside $L\llbracket q\rrbracket$ we project away all coefficients $a_n$ with $p|n:$ $${\rm pr}^{\{p\}}:L\llbracket q\rrbracket\surj \prod_{\gcd(n,p)=1} Lq^n.$$ We then set $$\phi^{\{p\}}_{\infty,L} = {\rm pr}^{\{p\}}\circ\phi_{\infty,L}.$$
\begin{prop}\label{prop_main_lemma}
    The homomorphism $\phi^{\{p\}}_{\infty,L}$ is injective.
\end{prop}

\begin{proof}
    This follows the base-change theorem \ref{thm_base_change} and Theorem \cite[Theorem 5.7.1]{diamondshurman2005first}.
\end{proof}

Now let $M$ be a $\T$-submodule of $H^0(X_{U_1(N),\Oh},\omega^{\otimes k}_{U_1(N),\Oh}(-\infty))$ that satisfies the base-change theorem \ref{thm_base_change}. We let $M^{\{p\}}$ be the $$\left(M\otimes_{\Oh}L\right) \bigcap \phi^{\{p\},-1}_{\infty,L}\left(\prod_{\gcd(n,p)=1}\Oh q^n\right).$$ One can check that $\T^{p}(M)$ acts on $M^{\{p\}}$ by the $q$-expansion formulas in \S \ref{subsubsec_Hecke_action}.

\begin{prop}\label{prop_coker_qexp_away_from_p_torsion_free}
    The cokernel of the map $M^{\{p\}}\hookrightarrow \prod_{\gcd(p,n)=1}\Oh q^n$ is $\Oh$-torsion free. 
\end{prop}

\begin{proof}
    Let $(a_n)_{\gcd(n,p)=1}\in \prod_{\gcd(n,p)=1}\Oh q^n$ be such that $(a_n)_{\gcd(n,p)=1}$ comes from a modular form $f\in M^{\{p\}}$ for some $a\in \Oh\setminus \{0\}.$ Then $f/a$ is in $M\otimes_{\Oh}L$ such that its $q$-expansion away from $p$ is $(a_n)_{\gcd(n,p)=1}\in \prod_{\gcd(n,p)=1}\Oh q^n.$ By definition of $M^{\{p\}},$ we have $f/a\in M^{\{p\}}$ and so $(a_n)_{\gcd(n,p)=1}$ is in the image of $\phi_{\infty,L}^{\{p\}}(M^{\{p\}}).$
\end{proof}

\begin{lem}\label{lem_duality_anemic}
    The pairing \begin{align*}
        \T^{p}(M) \times M^{\{p\}} &\to \Oh\\
        (T,f) &\mapsto a_1(Tf)
    \end{align*} is $\T^{p}$-equivariant and perfect.
\end{lem}

\begin{proof}
    By Proposition \ref{prop_coker_qexp_away_from_p_torsion_free}, the same argument as in Ribet’s proof of Theorem \ref{thm_ribet_duality_full} applies once we know \(M^{\{p\}}\) is finite free over \(\Oh\). Consider the \(\Oh\)\nobreakdash-linear map
    \[
      M^{\{p\}}
      \;\to\;
      \Hom_{\Oh}\!\bigl(\T^{p}(M),\,\Oh\bigr),
      \qquad
      f\;\mapsto\;\bigl(T\mapsto a_1(Tf)\bigr).
    \]
    If \(a_n(f)=0\) for every \(n\) with \(\gcd(n,p)=1\), then \(f=0\) by Proposition \ref{prop_main_lemma}. Meanwhile, \(\T^p(M)\subseteq \End_{\Oh}(M)\) is a finite free \(\Oh\)\nobreakdash-module, so its dual \(\Hom_{\Oh}(\T^p(M),\Oh)\) is also finite free.  Hence \(M^{\{p\}}\) is finite free.
\end{proof}

\begin{lem}
    The kernel of the map $$M\otimes_{\Oh}\F \to M^{\{p\}}\otimes_{\Oh}\F$$ is $\left(M\otimes_{\Oh}\F\right)\cap \Image V_p,$ where $V_p$ is viewed as an operator on the space of $q$-expansions. 
\end{lem}

\begin{proof}
    Consider the commutative diagram:
    $$\begin{tikzcd}
        &M\otimes_\Oh\F \arrow[r]\arrow[d,hook]& M^{\{p\}}\otimes_\Oh\F \arrow[d,hook]\\
        &\F\llbracket q\rrbracket \arrow[r,two heads]& \prod_{\gcd(n,p)= 1}\F q^n,
    \end{tikzcd}$$ where the vertical maps are induced by the $q$-expansion homomorphisms $\phi_{\infty,\Oh}|_{M}:M\to \Oh\llbracket q\rrbracket $ and $\phi_{\infty,L}^{\{p\}}|_{M^{\{p\}}}:M^{\{p\}}\to \prod_{\gcd(n,p)=1}\Oh q^n$ respectively. These $q$-expansion maps remain injective modulo $\varpi$ because they have torsion-free, hence $\Oh$-flat cokernels (by Theorem \ref{thm_q_expansion_principle} and Proposition \ref{prop_coker_qexp_away_from_p_torsion_free} respectively). So the kernel of the map \(M\otimes_{\Oh}\F\to M^{\{p\}}\otimes_{\Oh}\F\) consists of elements whose $q$-expansions are supported only at powers of \(q^p\), and are thus annihilated by the $\theta$ operator. It then follows from Theorem \ref{thm_katz_theta} that this kernel is exactly the image of \(V_p\) in \(M\otimes_{\Oh}\F\).
\end{proof}

For a $\T$-equivariant submodule $M$ of $H^0(X_{U_1(N),R},\omega^{\otimes k}_{U_1(N),R}(-\infty))$, we write $C(M)$ for the cokernel of the inclusion \begin{equation} \label{eqn_anemic_into_full}
    \T^{p}(M)\hookrightarrow\T(M).
\end{equation}

\begin{cor} \label{cor_anemic_equal_full_criterion}
    The module $C(M)\otimes_\Oh\F$ is dual to $(M\otimes_\Oh\F) \cap \Image V_p.$ In particular, the inclusion \eqref{eqn_anemic_into_full} is an isomorphism if and only if $(M\otimes_{\Oh}\F) \cap \Image V_p=\{0\}.$
\end{cor}

\begin{proof}
    Let $C(M)$ be the cokernel of the inclusion. It suffices to show $C(M)\otimes_{\Oh}\F=\{0\}$ if and only if $M\otimes_{\Oh}\F \cap \Image V_p=\{0\}$. Reducing modulo $\varpi,$ we obtain an exact sequence $$0\to C(M)[\varpi]\to \T^{p}(M)\otimes_{\Oh}\F\to \T(M)\otimes_{\Oh}\F\to C(M)\otimes_{\Oh}\F\to 0.$$ Taking $\F$-duals and applying Theorem \ref{thm_ribet_duality_full} and Lemma \ref{lem_duality_anemic} we identify \((\T^{p}(M)\otimes_{\Oh}\F)^\vee\) with \(M^{\{p\}}\otimes_{\Oh}\F\) and \((\T(M)\otimes_{\Oh}\F)^\vee\) with \(M\otimes_{\Oh}\F\). By the previous lemma, the term $C(M)\otimes_\Oh\F$ is dual to to $\left(M\otimes_\Oh\F\right)\cap \Image V_p,$ completing the proof. 
\end{proof}

\subsubsection{Serre's modular forms} We briefly recall Serre’s modular forms here; all statements in what follows are taken from \cite{Serre1996TwoLetters}.
Define $\mathcal{S}^k_{U_1(N)}$ to be the quotient sheaf $$\mathcal{S}^k_{U_1(N)} := \omega^{\otimes k}_{U_1(N),\F}/A\cdot \omega^{\otimes k-(p-1)}_{U_1(N),\F}$$ on the closed modular curve $X_{U_1(N),\F}$ of level $U_1(N)$ over $\F.$ In fact, since the Hasse invariant $A$ vanishes at supersingular locus with multiplicity 1, the sheaf $\mathcal{S}^k_{U_1(N)}$ is supported on the supersingular locus, which is a finite set of points. By Grothendieck's vanishing theorem \cite[Chapter III, Theorem 2.7]{hartshorne_1977_algebraic_geometry}, all but the zero-th cohomology groups vanish. Serre's modular forms $S(k,U_1(N))$ of weight $k$ and level $U_1(N)$ are defined to be the global sections of $\mathcal{S}^k_{U_1(N)}:$ $$S(k,U_1(N)) := H^0(X_{U_1(N),\F},\mathcal{S}^k_{U_1(N)}).$$ The short exact sequence of sheaves $$0\to \omega_{U_1(N),\F}^{\otimes k-(p-1)}\xrightarrow{\cdot A}\omega_{U_1(N),\F}^{\otimes k} \to \mathcal{S}^k_{U_1(N)}\to 0$$ induces a long exact sequence of cohomology groups \begin{multline*}
0\to H^0(X_{U_1(N),\F},\omega^{\otimes k-(p-1)}_{U_1(N),\F})\xrightarrow{\cdot A}H^0(X_{U_1(N),\F},\omega^{\otimes k}_{U_1(N),\F}) \to H^0(X_{U_1(N),\F},\mathcal{S}^k_{U_1(N)})\\\to H^1(X_{U_1(N),\F},\omega^{\otimes k-(p-1)}_{U_1(N),\F})\to H^1(X_{U_1(N),\F},\omega^{\otimes k}_{U_1(N),\F})\to0 .
\end{multline*} 
By Serre's duality and Kodaira--Spencer isomorphism, we have \begin{align*}
\dim_\F H^1(X_{U_1(N),\F},\omega^{\otimes k-(p-1)}_{U_1(N),\F}) &= \dim_\F H^0(X_{U_1(N),\F},\omega^{\otimes p-1-k}\otimes \Omega^1_{X_{U_1(N),\F}})\\
& = \dim_\F H^0(X_{U_1(N),\F},\omega^{\otimes p+1-k}(-\infty)),
\end{align*} which is zero as long as $k\ge p+1.$ When $\ell\neq p,$ the Hecke action of $\langle\ell\rangle$ and $T_\ell$ when $\ell\nmid N$ and $U_\ell$ when $\ell|N$ are defined using the same formula as \cite[(3.2),(3.3),(3.6)]{Gross1990Tameness}. All the aforementioned maps are $\T^{p}$-equivariant.

\begin{prop}[Serre] \label{prop_Sk_weight_periodicity}
    For a given prime $p$ and a level structure $U_1(N),$ the $\T^{p}$-module $S(k,U_1(N))$ depends only on the residue of $k$ modulo $p^2-1$, i.e., there is a $\T^{p}$-equivariant isomorphism $$S(k,U_1(N))\to S_{k+p^2-1}(k,U_1(N)).$$
\end{prop}

Let ${\rm Nrd}:D^{\times}(\A) \to \A^\times$ be the reduced norm map and let $\epsilon_D$ be the composition $\epsilon\circ {\rm Nrd}.$ By interpreting elements in $S(k,U_1(N))$ as $\Fpbar$-valued continuous locally constant functions of $D^\times(\Q)\backslash D^\times(\A)$ with a suitable level structure (see \cite[(8)--(10)]{Serre1996TwoLetters}), one obtain a map $$S(k,U_1(N))\xrightarrow{\cdot \epsilon_D} S(k+p+1,U_1(N))$$ induced by multiplication by $\epsilon_D.$

\begin{prop}[Serre] \label{prop_Sk_theta_twist}
    The map $$S(k,U_1(N))(1)\xrightarrow{\cdot \epsilon_D}S(k+p+1,U_1(N))$$ is an isomorphism between $\T^{p}$-modules for all weights $k.$
\end{prop} 

\begin{prop}[Serre] \label{prop_Serre_mod_forms_k_pk}
    There $\T^{p}$-modules $S(k,U_1(N))$ and $S(pk,U_1(N))$ are isomorphic. 
\end{prop}

\subsubsection{Nebentypus characters}
Let $R$ be a $\Z[1/N]$-algebra. Let $U_0(N)$ be the subgroup of $\GL_2(\Zhat)$ consisting of matrices that are upper triangular modulo $N.$ Then $U_1(N)$ is a normal subgroup of $U_0(N)$ with the quotient $U_1(N)\backslash U_0(N)\xrightarrow{\sim}\left(\Z/N\Z\right)^{\times}.$
In fact, the right action of $U_0(N)$ on $H^0(X_{U_1(N),R},\omega_{U_1(N),R}^{\otimes k})$ coincides with the diamond operators discussed in \S \ref{subsubsec_Hecke_action}. Let $H$ be a subgroup of $\left(\Z/N\Z\right)^{\times}$ and let $U_H(N)$ be the preimage of $H$ in $U_0(N).$ Fix a character $$\chi:U_H(N)\surj H\to \overline{\Q}^{\times}$$ to be a left Dirichlet character. Let us consider the module $$\left(H^0(X_{U_1(N),R},\omega_{U_1(N),R}^{\otimes k})\otimes_R R(\chi)\right)^{H}.$$

\begin{remark}
    If $U_H(N)$ is torsion free and $R$ is an $\Oh$-algebra, then the module is isomorphic to $$H^0(X_{U_H(N),R},\omega^{\otimes k}_{U_H(N),R}\otimes \chi),$$ which coincides with the twisted version of modular forms in \cite[\S3.9.1]{calegari_geraghty_2018_beyond}.
\end{remark}

We see that it naturally embeds into $H^0(X_{U_1(N),R},\omega^{\otimes k}_{U_1(N),R})$ because every element in the module takes the form $f\otimes_R 1$ for some $f\in H^0(X_{U_1(N),R},\omega^{\otimes k}_{U_1(N),R}).$ We have $$f\otimes_R 1 = (f\otimes_R 1)\langle d\rangle = (f\langle d\rangle)\otimes_R \chi(d)^{-1},$$ i.e.,$$(f\langle d\rangle)\otimes_R 1 = \chi(d) f\otimes_R 1$$ for every $d\in H.$ We say $\chi$ is the Nebentypus character of modular forms $f\otimes_R 1$ in this space. (We warn the reader that this is not the standard convention of a Nebentypus character unless $H = (\Z/N\Z)^\times$).

\begin{prop}\label{prop_Nebentypus_base_change}
    If either\begin{itemize}
    \item $R$ is flat or
    \item $k\ge 2$, $R=\F$ and the $p$-part $P$ of $H$ satisfies that $U_P(N)$ is torsion free,
    \end{itemize}then $$\left(H^0(X_{U_1(N)},\mathscr{L})\otimes \Z\left[\frac{1}{N}\right](\chi)\right)^{H}\otimes R \xrightarrow{\sim}\left(H^0(X_{U_1(N),R},\mathscr{L}_R)\otimes R(\chi)\right)^{H}$$ for $\mathscr{L} = \omega_{U_1(N)}^{\otimes k}$.
\end{prop}

\begin{proof}
    Denote by $M$ the $\Z[1/N][H]$-module $$H^0(X_{U_1(N)},\omega_{U_1(N)}^{\otimes k})\otimes \Z\left[\frac{1}{N}\right](\chi).$$ It then suffices to prove $$M^H\otimes_{\Z[1/N]}R\xrightarrow{\sim}(M\otimes_{\Z[1/N]}R)^{H}$$ by Theorem \ref{thm_base_change}. If $R$ is flat, this follows from Proposition \ref{prop_invariant_commute_flat_base_change} below. To prove the other case, we let $M_{\Oh}$ be its base-change to $\Oh$. Then $M_\Oh^P$ has no ambiguity as $\Oh$ is flat over $\Z[1/N].$ Now we aim to show $$\left(M_\Oh^{P}\right)^{H/P}\otimes_\Oh\F\xrightarrow{\sim}\left(\left(M_\Oh\otimes_\Oh \F\right)^P\right)^{H/P}.$$ From the short exact sequence $$0\to M_\Oh^P\xrightarrow{\cdot \varpi}M_{\Oh}^P\to M_{\Oh}^P\otimes_{\Oh}\F\to 0,$$ we deduce the exact sequence $$0\to \left(M_\Oh^{P}\right)^{H/P}\otimes_{\Oh}\F\to (M_{\Oh}^P\otimes_{\Oh}\F)^{H/P}\to H^1(H/P,M_\Oh^{P}).$$ The last term is zero because the order of $H/P$ is invertible on the $\Oh$-module $M_{\Oh}^P.$ It remains to show $$M_\Oh^P\otimes_{\Oh}\F\xrightarrow{\sim}\left(M_{\Oh}\otimes_{\Oh}\F \right)^P$$ which requires further inputs from geometry of modular curves as follows. We have \begin{align*}
        M_\Oh^P\otimes_{\Oh}\F \xrightarrow{\sim} \left(H^0(X_{U_1(N),\Oh},\omega^{\otimes k}_{U_1(N),\Oh})\otimes\chi\right)^{P}\otimes_{\Oh}\F\xrightarrow{\sim}H^0(X_{U_P(N),\F},\omega^{\otimes k}_{U_H(N),\F}\otimes\chi)\\
        \xrightarrow{\sim} \left(H^0(X_{U_1(N),\F},\omega^{\otimes k}_{U_1(N),\F})\otimes\chi\right)^{P}\xrightarrow{\sim} \left(M_{\Oh}\otimes \F\right)^P,\\
    \end{align*} where the second isomorphism follows from \cite[Corollary 3.29]{calegari_geraghty_2018_beyond} when $k\ge 2$ and the third isomorphism follows from \cite[Lemma 3.27(2)]{calegari_geraghty_2018_beyond}.
\end{proof}

\begin{prop}{\label{prop_invariant_commute_flat_base_change}}
    Suppose that $H$ is a finite abelian group and $S$ is a Noetherian ring. Let $R$ be a flat $S$-algebra with trivial $H$-action.
    Then for every $S$-module $M,$ we have $$M^{H}\otimes_S R\xrightarrow{\sim} (M\otimes_S R)^{H}.$$ 
\end{prop}

\begin{proof}
    Let $F_\bullet$ be a free resolution of $S$ as and $S[H]$-module where each term is a finite free $S[H]$-module (This is because $S[H]$ is a Noetherian ring by our assumption and finitely generated modules over a Noetherian ring are automatically finitely presented). 
    The cochain complex $$\Hom_{S[H]}(F_\bullet, M\otimes_SR) \xrightarrow{\sim} \Hom_{S[H]}(F_\bullet,M)\otimes_SR$$ computes the group cohomology of $H$ with coefficients in $M\otimes_S R.$ 
    Thus we have $$H^0(G,M\otimes_S R) = \ker(\Hom_{S[H]}(F_0,M)\otimes_S R\to \Hom_{S[H]}(F_1,M)\otimes_S R).$$ Since $R$ is flat over $S$, this is isomorphic to $$\ker(\Hom_{S[H]}(F_0,M)\to \Hom_{S[H]}(F_1,M))\otimes_S R=M^{H}\otimes_S R.$$
\end{proof}

It is not hard to check all other discussions on the $q$-expansion principle, the Hecke action, the mod-$p$ modular forms, the duality theorems and Serre's modular forms hold for modular forms with Nebentypus characters. From now on, we denote by 
$$S(k,U_1(N),\chi) := \left(H^0(X_{U_1(N),\F},\mathcal{S}^{k}_{U_1(N)})\otimes \chi\right)^{H}.$$

\subsection{Galois Representations and Galois Deformations}
\subsubsection{Galois Representations}\label{subsubsection_Galois_representations}
First, we briefly recall the definition of the Artin conductor for two-dimensional $p$-adic and mod-$p$ representations of $G_\Q$.  
Let $\rho: G_{\Q} \to \GL_2(L)$ be a Galois representation that is almost everywhere unramified.  
For a prime $\ell \neq p$, the \emph{local Artin conductor} of $\rho$ at $\ell$ is defined by
\[
n_\ell(\rho) := 2 - \dim \rho^{I_\ell} + sw(\rho|_{G_{\Q_\ell}}),
\]
where $sw(\rho)\in \Z_{\ge 0}$ denotes the Swan conductor of $\rho$.  
The representation $\rho$ is unramified at $\ell$ if and only if $n_\ell(\rho) = 0$, and it is tamely ramified at $\ell$ if and only if $sw(\rho) = 0$. The \emph{Artin conductor} of $\rho$ is then
\[
N(\rho) := \prod_{\ell \neq p} \ell^{n_\ell(\rho)},
\]
where the product is finite, since $\rho|_{G_{\Q_\ell}}$ is unramified for all but finitely many primes.  

Similarly, if $\rhobar: G_\Q \to \GL_2(\F)$ is a residual representation, its Artin conductor is defined by the same formula with $\rho$ replaced by $\rhobar$ and $L$ replaced by $\F$.  
When $\rhobar$ is the semisimplification of the reduction of a lattice in $\rho$ and is absolutely irreducible (so that, by Brauer--Nesbitt's theorem, $\rhobar$ is independent of the choice of lattice), we have $sw(\rho)=sw(\rhobar),$ since the Swan conductor only depends on the restriction to the wild inertia subgroup $P_\ell$, which is a pro-$\ell$ group.  
It follows that 
\[
n_\ell(\rhobar) \le n_\ell(\rho) \quad \text{and hence} \quad N(\rhobar) \mid N(\rho).
\]

The following will be used in \S \ref{subsection_hecke_modules_and_algebras}.
\begin{prop} \label{prop_dim_rhobar_I_ell_is_one}
    Let $\ell\neq p$ be a prime number. Suppose that $\rhobar:I_\ell\to \GL_2(\F)$ is a continuous representation such that $\dim_\F \rhobar^{I_\ell} = 1.$ Then either \begin{enumerate}
        \item $\rhobar$ is a non-split extension of $1$ by $1$, or
        \item $\rhobar$ is a direct sum of $1$ and a nontrivial character of $I_\ell.$
    \end{enumerate} 
\end{prop}

\begin{proof}
    Since $\dim_\F \bar\rho^{I_\ell}=1$, in the non-split case it suffices to show 
    that any non-split extension of characters must be an extension of $1$ by $1$. To see this, we calculate $H^1(I_\ell,\chi)$ for every character $\chi:I_\ell\to \F^{\times}.$ In the inflation-restriction exact sequence, we have $$0\to H^1(I_\ell/P_\ell,\chi^{P_\ell})\to H^1(I_\ell,\chi)\to H^1(P_\ell,\chi),$$ the last term vanishes because $P_\ell$ is a pro-$\ell$-group but $\F$ has characteristic $p\neq \ell$. Let $\tau$ be a topological generator of $I_\ell/P_\ell.$ The group cohomology of a pro-cyclic group gives $$H^1(I_\ell,\chi)\xrightarrow{\sim}H^1(I_\ell/P_\ell,\chi^{P_\ell})\xrightarrow{\sim} \frac{\chi^{P_\ell}}{(\chi^{P_\ell}(\tau)-1)\F}.$$ This is nonzero if and only if $\chi$ is trivial.
\end{proof}

From now on, we let $\rhobar: G_{\Q} \to \mathrm{GL}_2(\F)$ be a representation for which conditions (1) and (2) in Assumption~\ref{assumption_global_rhobar} hold. To carry out minimal-level patching, one would normally patch at level $U_1(N(\rhobar))$. However, if $N(\rhobar) < 5$, the moduli problem of elliptic curves with level $U_1(N(\rhobar))$ is not represented by a scheme, as mentioned at the beginning of \S\ref{subsubsec_modular_forms}. To avoid this issue, following the standard trick of Diamond--Taylor, we pass to level $U_1(N(\rhobar)\ell)$ for some prime $\ell \nmid N(\rhobar),\ell \ge 5$ that meets the conditions in the proposition below.

\begin{prop}[Diamond--Taylor]\label{prop_unramified_lifting}
    Let $\ell\neq p$ be a rational prime and $\rhobar_\ell:G_{\ell}\to \GL_2(\F)$ be an unramified representation of the local Galois group $G_\ell.$ Suppose that \begin{enumerate}
        \item $\ell\not \equiv 1\pmod p$ and
        \item the ratio of the eigenvalues of $\rhobar_\ell(\Frob_\ell)$ is not in $\{\ell,\ell^{-1}\}.$
    \end{enumerate} If $\rho_\ell:G_{\Q_\ell}\to \GL_2(\Oh)$ is a lift of $\rhobar_\ell$, it is still unramified. 
\end{prop}

\begin{remark}
This was first proved in the analysis preceding \cite[Lemma~2]{diamond_taylor_1994_lifting}. It can also be deduced from the structure of the universal lifting ring of $\rhobar$ computed in \cite{shotton2016local}.
\end{remark}

Under conditions (1) and (2) in Assumption~\ref{assumption_global_rhobar}, the existence of such a prime is proved in \cite[Lemma~2]{diamond_taylor_1994_lifting}. However, we would like to carry out patching in a multiplicity-one situation so that the patched module is cyclic over $R_\rhobar^{\Box,\psi}[\alpha_p]$, as mentioned in \S\ref{subsubsection_main_construction}. Therefore, we require $\ell$ to satisfy a stronger version of condition (2) in Proposition~\ref{prop_unramified_lifting}:

\begin{enumerate}
    \item[(2')] the ration of the eigenvalues of $\rhobar_\ell(\Frob_\ell)$ is not in $\{1,\ell,\ell^{-1}\}.$ 
\end{enumerate} 
Under this stronger assumption, $\rhobar_\ell(\Frob_\ell)$ has distinct eigenvalues. By further localizing at one of these eigenvalues, we are in the desired multiplicity-one situation. This is also explained in \cite[Remark~3.10]{calegari_geraghty_2018_beyond}.

Kisin proved the existence of a prime $\ell$ satisfying the stronger conditions as long as $p\ge 5$. This is first announced by Kisin in \cite[Lemma 2.2.1]{Kisin09FontaineMazur} and later corrected in \cite[B.4]{GeeKisin2014breuil}. 
\begin{lem}[Kisin]\label{lem_unramified_prime}
    Suppose that $p\ge 5.$ Let $\rhobar:G_{\Q}\to \GL_2(\F)$ be a global Galois representation such that $\rhobar|_{G_{\Q(\zeta_p)}}$ is absolutely irreducible. Then there are infinitely many primes $\ell$ such that conditions (1) and (2') in Proposition \ref{prop_unramified_lifting} hold. 
\end{lem}

\begin{remark}
    As a result of Kisin's lemma, in condition (3) of Assumption \ref{assumption_global_rhobar}, we ask $p\ge 5$ or $N(\rhobar)\ge 5.$
\end{remark}
\subsubsection{Universal Lifting Rings}
Let $G = G_{\Q_\ell}$ and let $\rhobar_\ell:G_{\Q_\ell}\to \GL_2(\F)$ be a continuous representation. When $\ell\neq p,$ it is Shotton's theorem that $R^{\Box}_{\rhobar_\ell}$ and $R^{\Box,\psi}_{\rhobar_\ell}$ are reduced complete intersections, flat and equidimensional over $\Oh.$ See {\cite[Theorem 2.5]{shotton_2018_bm_gln}} for the unrestricted version and {\cite[Proposition 4.3]{BockleKhareManning2021wiles}} for the fixed determinant version.

In fact, the only case we are interested in here is when $\ell$ is a vexing prime (see \ref{rmk_vexing_primes}). In this case, the representation is isomorphic to $\Ind_{G_{\Q_{\ell^2}}}^{G_{\Q_\ell}}\phi$ where $\Q_{\ell^2}$ is the unique unramified quadratic extension of $\Q_\ell$ and $\phi$ is a character of $G_{\Q_{\ell^2}}$ that does not extend to $G_{\Q_\ell}.$ The presentation of the unrestricted deformation ring is explicitly calculated. 

\begin{lem}[{\cite[Proposition 5.1]{shotton2016local}}]
    The universal lifting ring $R^{\Box}_{\rhobar_\ell}$ of $\rhobar_\ell\xrightarrow{\sim} \Ind_{G_{\Q_{\ell^2}}}^{G_{\Q_\ell}}\phi$ is isomorphic to $$\Oh\llbracket X,Y,Z_1,Z_2,Z_3\rrbracket/((1+Y)^{p^{v_p(\ell^2-1)}}-1).$$
\end{lem}

\begin{remark} \label{rmk_vexing_prime_local_deformation_ring} 
    Following Shotton's idea, one can show that the fixed-determinant lifting ring of $\rhobar_\ell \xrightarrow{\sim} \Ind_{G_{\Q_{\ell^2}}}^{G_{\Q_\ell}}\phi$ is isomorphic to $$\Oh\llbracket Y,Z_1,Z_2,Z_3\rrbracket /((1+Y)^{p^{v_p(\ell^2-1)}}-1).$$ In particular, it is a reduced complete intersection, flat and equidimensional of dimension $3$ over $\Oh.$ 
    Its special fiber is a geometrically irreducible local complete intersection of Krull dimension $3.$
\end{remark}

When $\ell = p,$ let us write $\rbar$ for the representation $\rhobar_p$ to be consistent with the notation in the later sections. The following theorem is obtained in \cite{Bockle_Iyengar_Paskunas_2023_local_deformation_rings}. 

\begin{thm}[B\"ockle--Iyengar--Pa\v{s}k\=unas] \label{thm_unrestricted_deformation_ring_at_p}
    The special fiber $R^{\Box,\psi}_\rbar\otimes_\Oh\F$ is a local complete intersection of Krull dimension $6$ and it is an integral domain. 
\end{thm}

\begin{remark}
    Since the statement of the theorem is irrelevant to the coefficient ring $\Oh,$ from the moduli interpretation of $R_\rbar^{\Box,\psi}\otimes_\Oh\F,$ it in fact implies that the ring $R^{\Box,\psi}_\rbar\otimes_\Oh\F$ is geometrically integral. 
\end{remark}
\subsubsection{Crystalline Representations}\label{subsubsection_crystalline_representations}
In this part, we recall properties of crystalline representations and crystalline deformation rings. 
We first clarify our conventions for Hodge--Tate weights and the crystalline Frobenius, as they vary across the literature. Along the way, we indicate how our choices compare to those in other sources.
Let $r:G_{\Q_p}\to \GL_n(L)$ be a $p$-adic representation of $G_{\Q_p}.$ For each integer $i,$ define $$d_i := \dim_F \left(r \otimes_{\Q_p} \C_p(-i)\right)^{G_{\Q_p}}.$$
We say that $i$ is a Hodge--Tate weight of $r$ with multiplicity $d_i$ if $\sum_i d_i = n.$  With this normalization, the cyclotomic character $\chi_\cyc$ has Hodge--Tate weight $1$ (rather than $-1$).Define the covariant crystalline functor $$D_\cris(r):=\left(r\otimes_{\Q_p}B_{\cris}\right)^{G_{\Q_p}}$$ where $B_\cris$ is the Fontaine's crystalline period ring. We also define the contravariant version: $$D_\cris^*(r) := D_\cris(r^\vee) = \left(r^\vee\otimes_{\Q_p}B_\cris\right)^{G_{\Q_p}} = \Hom_{G_{\Q_p}}(r,B_{\cris}).$$ We say that $r$ is \emph{crystalline} if $\dim_L D_\cris(r) =n,$ or equivalently, $\dim_L D^*_\cris(r) = n.$
Kisin \cite{Kisin08Potentially,Kisin09FontaineMazur} uses the contravariant functor $D_\cris^*$, while the six-author patching paper \cite{Caraiani_Emerton_Gee_Geraghty_PaskunasShin_2018_GL2Qp} adopts the covariant $D_\cris.$ Both are weakly admissible filtered $\varphi$-modules. The filtration jumps in $D_\cris(r)$ occur at the negatives of the Hodge--Tate weights of $r,$ whereas for $D^*_\cris(r),$ they occur at the Hodge--Tate weights themselves. 

In the two-dimensional case, we explicitly have $$r^\vee \xrightarrow{\sim} r\otimes \det(r)^{-1}$$ and hence: $$D^*_\cris(r)\xrightarrow{\sim}D_\cris(r)\otimes_L D_\cris(\det(r)^{-1}).$$ The crystalline Frobenius on $D_\cris^*(r)$ is given by $\varphi\det(\varphi)^{-1}$ where $\varphi$ is the crystalline Frobenius on $D_\cris(r).$ 

We can associate a Weil--Deligne representation $\WD(r)$ to an $n$-dimensional representation $r$ by taking the underlying vector space $D_\cris(r),$ letting the arithmetic Frobenius to act as $\varphi^{-1}$ and setting the monodromy operator to zero.

In what follows, we refer to the Frobenius on $D_\cris^*(r)$ the crystalline Frobenius associated to $r.$ In the two dimensional case, since $\varphi\det(\varphi)^{-1}$ and $\varphi^{-1}$ have the same characteristic polynomials, the characteristic polynomial of the crystalline Frobenius of $r$ coincides with that of the arithmetic Frobenius on $\WD(r).$ 

The following proposition is a direct consequence of the weak admissibility of $D_\cris(r^*)$ and the $p$-adic Hodge theory for characters. 
\begin{prop}\label{prop_crystalline_characteristic_polynomial}
    Suppose that $r:G_{\Q_p}\to \GL_2(L)$ is a 2-dimensional crystalline representation of Hodge--Tate weights $0$ and $k-1$ for some integer $k\ge 2.$ Let $X^2-aX + b$ be the characteristic polynomial of the crystalline Frobenius acting on $D_\cris^*(r).$
    \begin{enumerate}
        \item The $p$-adic valuation $v_p(a)$ of $a$ is non-negative. The representation $r$ is absolutely irreducible if and only if $v_p(a)>0$. When $a$ is a $p$-adic unit, the representation $r$ is an extension of an unramified character $\eta$ by a crystalline character of Hodge--Tate weight $k-1$ and $v_p(\eta(\Frob_p)-a)>0.$
        \item We have $$b=(\det(r)/\chi_\cyc^{k-1})(\Frob_p)p^{k-1},$$ where $\Frob_p$ is the arithmetic Frobenius at $p.$
    \end{enumerate}
\end{prop}

The crystalline representations in this paper are those coming from modular forms:
\begin{thm}[{\cite[Theorem 1.2.4]{Scholl1990motives}}] \label{thm_Scholl_rhof_crystalline_at_p}
    Let $f=\sum_n a_np^n$ be a normalized cuspidal eigenform of weight $k\ge 2$, level $U_1(N)$ with Nebentypus character $\chi.$ Let $p\nmid N$ be a prime and let $\rho_f:G_{\Q}\to \GL_2(L)$ be the $p$-adic Galois representation attached to $f.$ Then $\rho_f|_{G_{\Q_p}}$ is crystalline with Hodge--Tate weights $0$ and $k-1$ and the characteristic polynomial of the crystalline Frobenius is $$X^2-a_p(f)X + \chi(p)p^{k-1}.$$ 
\end{thm}

Crystalline deformation rings are defined by {\cite[Corollary 2.7.7]{Kisin08Potentially}} by taking the Zariski closure of crystalline lifts of $\rbar$ satisfying the desired weight condition. We include here the version we need. 
\begin{thm}[Kisin]\label{thm_Kisin_crystalline_generic_fiber}
    Suppose that $k \ge 2.$ There is a unique (possibly trivial) quotient $R^{\Box,\psi}_\rbar(k)$ with the following properties:
    \begin{enumerate}
        \item $R^{\Box,\psi}_\rbar(k)$ is $p$-torsion free, $R^{\Box,\psi}_\rbar(k)[1/p]$ is formally smooth and equidimensional of dimension $4$. 
        \item If $L'/L$ is a finite extension, then a map $x:R^{\Box}_\rbar\to L'$ factors through $R^{\Box,\psi}_\rbar(k)$ if and only if the corresponding $E'$-representation is crystalline of Hodge--Tate weights $(0,k-1)$ and has determinant $\psi.$
    \end{enumerate}
\end{thm}

\begin{remark} \label{rmk_change_det_of_deformation_ring}
        In \cite{Kisin09FontaineMazur}, the notation $R^{\Box,\psi}(k,\id,\rbar)$ refers to the framed crystalline deformation ring parametrizing lifts of $\rbar$ with Hodge--Tate weights $\{0,k-1\}$ and determinant equal to $\psi\chi_\cyc$. In contrast, our notation $R^{\Box,\psi}_\rbar(k)$ denotes the framed crystalline deformation ring where the determinant is $\psi$ (i.e., without the cyclotomic twist). If two characters $\psi$ and $\psi'$ differ by an unramified character, then the corresponding crystalline deformation rings are isomorphic by the twisting argument in \hyperref[subsubsection_notation_and_conventions]{Notation and Conventions}. 
\end{remark}

Let $\alpha_p$ be the element in $R^{\Box,\psi}_\rbar(k)[1/p]$ such that if $x:R^{\Box,\psi}_\rbar(k)[1/p]\to \Qpbar$ is an $L$-algebra homomorphism, then $\alpha_p(x)$ is the trace of the crystalline Frobenius element. Since $R^{\Box,\psi}_\rbar(k)[1/p]$ is a Jacobson ring by Lemma \ref{lem_Taylor}, this uniquely determines $\alpha_p.$ 

\begin{example} \label{example_crystalline_deformation_rings}
    When $\rbar|_{I_p} \sim \begin{pmatrix}
        \epsilon_2^{k_0-1} & 0\\
        0 & \epsilon_2^{p(k_0-1)}
    \end{pmatrix}$ for some integer $2\le k_0\le p$, the explicit presentations of $R^{\Box,\psi}_\rbar(k)$ are calculated in the following cases: $$R^{\Box,\psi}_\rbar(k) \xrightarrow{\sim} \begin{cases}
        \Oh\llbracket x_1,x_2,x_3,\alpha_p\rrbracket & 2\le k\le p+1 {\rm \ and\ }k \equiv k_0\pmod{p-1}\\
        \Oh\llbracket x_1,x_2,x_3,y,\alpha_p\rrbracket/(y\alpha_p - p) & p+2\le k\le 2p-1 {\rm \ and\ }k \equiv k_0\pmod{p-1}
    \end{cases}.$$  
    When $\rbar|_{I_p}\sim \begin{pmatrix}
        \epsilon & *\\
        0 & 1
    \end{pmatrix}$ but $\rbar$ itself is not up to twist an extension of $1$ by $\epsilon,$ we have $$R_\rbar^{\Box,\psi}(k) \xrightarrow{\sim} \Oh\llbracket x_1,x_2,x_3,\alpha_p\rrbracket$$ for $k = 2$ or $p+1.$ When $2\le k\le p+1,$ the presentation is due to the Fontaine--Laffaille theory \cite{FontaineLaffaille1982construction}.
    When $p+2\le k\le 2p-1,$ these are computed by Kisin in \cite[(3.13.1)]{kisin2007modularity}, building on the work of \cite{BB05Poids} and \cite{BLZ03Construction}. 
\end{example}

In general, the element $\alpha_p$ may not be in $R_\rbar^{\Box,\psi}(k).$ But we have the following: 

\begin{lem}[{\cite[Corollary 2.11]{Caraiani_Emerton_Gee_Geraghty_PaskunasShin_2018_GL2Qp}}]
    The element $\alpha_p$ is in the normalization of $R_\rbar^{\Box,\psi}(k)$ in $R_\rbar^{\Box,\psi}(k)[1/p].$
\end{lem}
Hence, the ring $R_\rbar^{\Box,\psi}(k)[\alpha_p]$ is finite over $R_\rbar^{\Box,\psi}(k).$ However, it need not be a local ring. 
\begin{prop} \label{prop_R_ap_decomposition}
Let $\mathfrak{m}$ denote the maximal ideal of $R^{\Box,\psi}_\rbar(k)$. Define
\[
U_{\rbar} \subset \F \coloneqq \{0\} \;\;\sqcup\;\; 
\bigl\{ \eta(\Frob_p) : \eta: G_{\Q_p} \to \F^\times \text{ unramified character}, \ \Hom_{G_{\Q_p}}(\rbar,\eta) \neq 0 \bigr\}.
\]
Then the ring $R_\rbar^{\Box,\psi}(k)[\alpha_p]$ decomposes as
\[
R_\rbar^{\Box,\psi}(k)[\alpha_p] \simeq \bigoplus_{u \in U_\rbar} R_u,
\]
where each $R_u$ is an object in $\mathcal{C}_\Oh$ with maximal ideal generated by $\mathfrak{m}$ and $\alpha_p - \widetilde{u}$, for any choice of lift $\widetilde{u}$ of $u$ in $R_\rbar^{\Box,\psi}(k)$.
\end{prop}

 \begin{proof}
    The ring $R_\rbar^{\Box,\psi}(k)[\alpha_p] \otimes_{R_\rbar^{\Box,\psi}(k)} \F$ is an $\F$-algebra generated by a single element $\alpha_p$. Hence, there exists a monic polynomial $f(x) \in \F[x]$ such that
\[
R_\rbar^{\Box,\psi}(k)[\alpha_p] \otimes_{R_\rbar^{\Box,\psi}(k)} \F \xrightarrow{\sim} \F[x]/(f(x)),
\]
where $\alpha_p$ is mapped to $x$. By enlarging $\F$ if necessary, we may assume that all roots $a_1, \dots, a_s$ of $f(x)$ lie in $\F$. Let $n_i$ denote the multiplicity of the root $a_i$. By the Chinese remainder theorem, we then have an isomorphism
\[
R_\rbar^{\Box,\psi}(k)[\alpha_p] \otimes_{R_\rbar^{\Box,\psi}(k)} \F \xrightarrow{\sim} \bigoplus_{i=1}^{s} \F[x]/(x-a_i)^{n_i}.
\]

Let $\overline{e}_i$ denote the idempotent in $R_\rbar^{\Box,\psi}(k)[\alpha_p] \otimes_{R_\rbar^{\Box,\psi}(k)} \F$ corresponding to the direct summand $\F[x]/(x-a_i)^{n_i}$. Since $R_\rbar^{\Box,\psi}(k)[\alpha_p]$ is finite over $R_\rbar^{\Box,\psi}(k)$, which is $\mathfrak{m}$-adically complete, the same holds for $R_\rbar^{\Box,\psi}(k)[\alpha_p]$. By Hensel's lemma, each $\overline{e}_i$ lifts uniquely to an idempotent $e_i \in R_\rbar^{\Box,\psi}(k)[\alpha_p]$, giving an isomorphism
\[
R_\rbar^{\Box,\psi}(k)[\alpha_p] \xrightarrow{\sim} \bigoplus_{i=1}^s e_i R_\rbar^{\Box,\psi}(k)[\alpha_p].
\]
Write $R_{a_i} := e_i R_\rbar^{\Box,\psi}(k)[\alpha_p]$.  

Since $R_{a_i}$ is finite over $R_\rbar^{\Box,\psi}(k)$, all of its maximal ideals contain $\mathfrak{m}R_{a_i}$. 
Moreover,
\[
R_{a_i} \otimes_{R_\rbar^{\Box,\psi}(k)} \F \xrightarrow{\sim} \F[x]/(x-a_i)^{n_i}
\] 
is a local ring with maximal ideal $(\alpha_p - a_i)$ and residue field $\F$. It follows that each $R_{a_i}$ is itself a local ring with maximal ideal
\[
\mathfrak{m}_i := \mathfrak{m} + (\alpha_p - \widetilde{a}_i),
\]
where $\widetilde{a}_i$ is any lift of $a_i$ in $R_{a_i}$, and residue field $\F$. Since $\mathfrak{m}_i^{n_i} R_{a_i} \subseteq \mathfrak{m} R_{a_i}$, the $\mathfrak{m}$-adic topology coincides with the $\mathfrak{m}_i$-adic topology. Thus the finite $R_\rbar^{\Box,\psi}(k)$-algebra $R_{a_i}$ is complete with respect to its $\mathfrak{m}$-adic (and hence $\mathfrak{m}_i$-adic) topology, and therefore $R_{a_i}$ is an object of $\mathcal{C}_\Oh$.

Furthermore, since $R_\rbar^{\Box,\psi}(k)[\alpha_p]$ is $p$-torsion free, so is each $R_{a_i}$. Hence, $R_{a_i}[1/p]$ is nontrivial. By Lemma~\ref{lem_Taylor}, $\Spec(R_{a_i}[1/p])$ contains some $E'$-point 
\[
x_i: R_\rbar^{\Box,\psi}(k)[\alpha_p][1/p] \to E'
\] 
for a finite extension $E'/E$ of $p$-adic fields. By definition, the crystalline Frobenius of the corresponding representation $\rho_{x_i}$ is congruent to $a_i$. If $a_i$ is a unit, then $\rho_{x_i}$ is an extension of an unramified character $\eta$ by a crystalline character, and $\eta(\Frob_p)$ is congruent to $a_i$ by Proposition~\ref{prop_crystalline_characteristic_polynomial}. Since $\rho_{x_i}$ reduces to $\rbar$, the reduction $\overline{\eta}$ is a quotient of $\rbar$. As a result, $a_i \in U_\rbar$. 
 \end{proof}

\begin{remark} \label{rmk_R_ap_completed_tensor_product_direct_sum}
    Let $R$ be an object in $\mathcal{C}_{\Oh}$. Consider the completed tensor product 
\[
R_\rbar^{\Box,\psi}(k)[\alpha_p] \,\widehat{\otimes}_{\Oh}\, R,
\] 
where we view $R_\rbar^{\Box,\psi}(k)[\alpha_p]$ as a finitely generated module over $R_\rbar^{\Box,\psi}(k)$ and $R$ as a module over itself. By the proposition above, we have 
\[
R_\rbar^{\Box,\psi}(k)[\alpha_p] \xrightarrow{\sim} \bigoplus_{u \in U_\rbar} R_u,
\] 
so that 
\[
R_\rbar^{\Box,\psi}(k)[\alpha_p] \,\widehat{\otimes}_{\Oh}\, R \;\xrightarrow{\sim}\; \bigoplus_{u \in U_\rbar} R_u \,\widehat{\otimes}_{\Oh}\, R.
\] 
Here, each $R_u$, despite being an object in $\mathcal{C}_\Oh$, is still viewed as an $R_\rbar^{\Box,\psi}(k)$-module equipped with the $\mathfrak{m}$-adic topology. Since, as noted in the proof above, this topology coincides with the adic topology induced by the maximal ideal of $R_u$, we can also view the completed tensor product $R_u \,\widehat{\otimes}_\Oh\, R$ as being taken between two objects in $\mathcal{C}_\Oh.$
\end{remark}

Following the notation in the proposition above, we let $R_\rbar^{\Box,\psi}(k)[\alpha_p]^{\ord}$ be $\oplus_{u\in U_{\rbar}\setminus\{0\}}R_u$ and call it the ordinary part of $R_\rbar^{\Box,\psi}(k)[\alpha_p]$ and let $R_\rbar^{\Box,\psi}(k)[\alpha_p]^{\rm no}$ be $R_0$ and call it the non-ordinary part. We write $R^{\Box,\psi}_\rbar(k)^{\ord}$ for the image of $R^{\Box,\psi}_\rbar(k)$ in $R^{\Box,\psi}_\rbar(k)[\alpha_p]^{\ord}$ and  $R^{\Box,\psi}_\rbar(k)^{\no}$ for the image of $R^{\Box,\psi}_\rbar(k)$ in $R^{\Box,\psi}_\rbar(k)[\alpha_p]^{\no}.$ For $*\in \{\ord,\no\},$ we have $$R_\rbar^{\Box,\psi}(k)^*[\alpha_p] = R_\rbar^{\Box,\psi}(k)[\alpha_p]^*.$$

Following the notation in Proposition~\ref{prop_Vp_preserves_ordinary_nonordinary_forms}, we define the ordinary and non-ordinary parts of $R_\rbar^{\Box,\psi}(k)[\alpha_p]$ as follows.  
Let
\[
R_\rbar^{\Box,\psi}(k)[\alpha_p]^{\ord} := \bigoplus_{u \in U_{\rbar} \setminus \{0\}} R_u,
\]
which we call the ordinary part, and let
\[
R_\rbar^{\Box,\psi}(k)[\alpha_p]^{\rm no} := R_0,
\]
which we call the non-ordinary part.  We then define
\[
R^{\Box,\psi}_\rbar(k)^{\ord} \subset R_\rbar^{\Box,\psi}(k)[\alpha_p]^{\ord} \quad\text{and}\quad
R^{\Box,\psi}_\rbar(k)^{\no} \subset R_\rbar^{\Box,\psi}(k)[\alpha_p]^{\rm no}
\]
as the images of $R^{\Box,\psi}_\rbar(k)$ under the natural projection. With this notation, for $* \in \{\ord, \no\}$, we have
\[
R_\rbar^{\Box,\psi}(k)^*[\alpha_p] = R_\rbar^{\Box,\psi}(k)[\alpha_p]^*.
\]

\subsubsection{Deformations of \texorpdfstring{$\rbar$}{rbar} and its Semisimplification} \label{subsubsec_semisimplification}
Let $\rbar$ be a local representation that is not semisimple and satisfies Assumption \ref{assumption_semisimplification}. Fix a character $\psi:G_{\Q_p}\to \Oh^\times.$ In this part, we compare the deformation theory of $\rbar$ with that of its semisimplification $\rbarss$ in order to replace Assumption \ref{assumption_global_rhobar} in all the theorems concerning a local representation $\rbar$ stated in \S\ref{subsec_main_theorems}. All results recalled here are from \cite{paskunas_2017_2_adic_deformations}, adapted to the fixed-determinant setting with determinant $\psi$. As noted in \cite[Remark 4.5, Remark 6.1]{paskunas_2017_2_adic_deformations}, Pa\v{s}k\=unas explained that all statements remain valid in this context.

The representations $\rbar$ and $\rbar^{\rm ss}$ have the same trace and determinant. By definition, they therefore have the same fixed-determinant pseudo-deformation ring $R^{{\rm ps},\psi}$ and a universal pseudo-deformation $(t^{\univ,\psi},\psi)$ representing the fixed-determinant pseudo-deformation functor $D^{{\rm ps},\psi}.$ We will compare the deformation theory of $\rbar,\rbarss$ by relating them to their pseudo-deformations.

Since $\rbar$ is Schur, the unrestricted deformation functor of $\rbar$ is represented by a power series ring $R_\rbar^{\univ,\psi}$ over $\Oh$ in three variables by a standard Galois cohomology calculation. Let $$r^{\univ,\psi}:G_{\Q_p}\to \GL_2(R_\rbar^{\univ,\psi})$$ be the universal deformation. Taking the trace and determinant of $r^{\univ,\psi}$ gives a natural ring homomorphism $$R^{{\rm ps},\psi}\to R^{\univ,\psi}_\rbar,$$ which is an isomorphism by \cite[Theorem B.17]{paskunas_2013_image_montreal}. In particular, we then have $$t^{\univ,\psi}=\tr r^{\univ,\psi}.$$

Let $D_{\rbarss}^{\psi}$ denote the deformation functor of $\rbarss$ with fixed determinant $\psi$. 
Since $\rbarss$ is not Schur, $D_{\rbarss}^{\psi}$ is not pro-representable. 
We usually work with the universal lifting ring $R^{\Box,\psi}_{\rbarss}$ together with its universal lifting
\[
s^{\Box,\psi}:G_{\Q_p}\to \GL_2(R^{\Box,\psi}_{\rbarss})
\] 
of $\rbarss$. 
On the other hand, $D_{\rbarss}^{\psi}$ always admits a versal ring $R^{\ver,\psi}_{\rbarss}$, together with
\[
s^{\ver,\psi}:G_{\Q_p}\to \GL_2(R^{\ver,\psi}_{\rbarss}).
\] 
By the definition of a versal hull (see \cite[Definition 2.7]{schlessinger_1968_functors}), the functor
\[
D^{\ver,\psi}_{\rbarss} := \Hom_{\mathcal{C}_\Oh}(R^{\ver,\psi}_{\rbarss},\cdot)
\] 
is formally smooth over $D_{\rbarss}^{\psi}$, and the induced map on tangent spaces is an isomorphism.

The following proposition from \cite[Proposition 2.1]{KhareWintenberger2009serreII} shows that there is essentially no difference between using $R_\rbar^{\Box,\psi}$ and $R_\rbar^{\univ,\psi}$, or between $R_\rbarss^{\Box,\psi}$ and $R_\rbarss^{\ver,\psi}$.

\begin{prop}[Khare--Wintenberger] \label{prop_lifting_ring_formally_smooth_over_versal_ring}
    Let $\rhobar$ be a representation. The universal lifting ring $R^{\Box,\psi}_{\rhobar}$ is a formal power series ring over the versal ring $R^{\ver,\psi}_{\rhobar}$.
\end{prop}

As a result, it suffices to understand the relation between the versal ring $R_\rbarss^{\ver,\psi}$ and the pseudo-deformation ring $R^{{\rm ps},\psi}$ of $\rbarss.$

In \cite[\S 5]{paskunas_2017_2_adic_deformations}, Pa\v{s}k\=unas explicitly calculated a presentation of the versal hull:
\[
    R^{\ver,\psi}_\rbarss = R^{{\rm ps},\psi}\llbracket x,y\rrbracket/(xy-c),
\]
where $c$ is a generator of the reducibility ideal in $R^{{\rm ps},\psi}.$ He also constructed the versal lifting
\[
    s^{\ver,\psi}:G_{\Q_p}\to \GL_2(R^\ver),
\]
and described the maps
\[
    D^{\ver,\psi}_\rbarss(A) \to D^{\psi}_\rbarss(A) \to D^{{\rm ps},\psi}(A),
\]
given by
\[
    s_A := s^{\ver,\psi}\otimes_{R^{\ver,\psi}_\rbarss}A \mapsto [s_A] \mapsto (\tr s_A,\det s_A).
\]

We are interested in how the deformation theories are related when we impose the crystalline condition.  
We write $R^{\ver,\psi}_\rbarss(k)$ for the crystalline quotient of $R^{\ver,\psi}_\rbarss$ whose generic fiber parametrizes crystalline lifts factoring through $R^{\ver,\psi}_\rbarss$, with Hodge--Tate weights $(0,k-1)$ and determinant $\psi$. Likewise, we define $R^{\univ,\psi}_\rbar(k)$ for $\rbar$.  Let 
\[
    f_x:R^{\ver,\psi}_\rbarss[1/p] \to \overline{\Q}_p
\] 
be an $\Oh$-algebra homomorphism with kernel $\mathfrak{m}_x$, and let 
\[
    s_x:G_{\Q_p}\to \GL_2(R^{\ver,\psi}_\rbarss)\to \GL_2(\overline{\Q}_p)
\] 
be the corresponding $p$-adic representation.  Suppose that $s_x$ is crystalline with distinct Hodge--Tate weights $(0,k-1)$ for $k\ge 2$.  
If $s_x$ is reducible, it is an extension of $\delta_2$ by $\delta_1$, where $\delta_i$ are crystalline characters of $G_{\Q_p}$ such that $\delta_1$ has Hodge--Tate weight $k-1\ge 1$, $\delta_2$ is unramified, and $\delta_1\delta_2 = \psi$.  
We say that $s_x$ is of reducibility type $1$ if
\[
    \delta_1\equiv \chi_1 \pmod\varpi \quad \text{and} \quad \delta_2 \equiv \chi_2 \pmod\varpi,
\] 
and of reducibility type $2$ if
\[
    \delta_1\equiv \chi_2 \pmod\varpi \quad \text{and} \quad \delta_2 \equiv \chi_1 \pmod\varpi.
\] 
We say that $s_x$ is of reducibility $\irr$ if $s_x$ is irreducible.  Let $* \in \{1,2,\irr\}$. We define 
\[
    I_*^{\ver} := R^{\ver,\psi}_\rbarss \cap \bigcap_{s_x \text{ of type } *} \mathfrak{m}_x.
\] 
(Note that $I_1^{\ver}$ is nontrivial only if $\chi_2$ is unramified, and $I_2^{\ver}$ is nontrivial only if $\chi_1$ is unramified.)  
By Kisin's construction of crystalline deformation rings, we then have
\[
    R^{\ver,\psi}_\rbarss(k) = R^{\ver,\psi}_\rbarss / (I_1^{\ver} \cap I_2^{\ver} \cap I_\irr^{\ver}).
\] By our previous discussion, we may identify $R_\rbar^{\univ,\psi}$ with the pseudo-deformation ring $R^{{\rm ps},\psi}$, which is then a subring of 
\[
    R^{\ver,\psi}_\rbarss = R^{{\rm ps},\psi}\llbracket x,y\rrbracket /(xy-c).
\] 
We then define
\[
    I_*^{\univ} := R^{\univ,\psi}_\rbar \cap \bigcap_{s_x \text{ of type } *} \mathfrak{m}_x.
\]

\begin{lem}[Pa\v{s}k\=unas]\label{lem_paskunas_versal_universal_crystalline}
    Suppose that $\rbar$ is non-semisimple and satisfies Assumption \ref{assumption_semisimplification}. For every $k\ge 2,$ we have $$R_\rbar^{\univ,\psi}(k) = R_\rbar^{\univ,\psi}/(I_1^\univ\cap I_\irr^{\univ}).$$ For each reducibility type, we have $$(R_\rbar^{\univ,\psi}/I_1^\univ)\llbracket y\rrbracket\xrightarrow{\sim}R_\rbarss^{\ver,\psi}/I_1^{\ver}$$ and $$(R_\rbar^{\univ,\psi}/I_\irr^\univ)\llbracket x,y\rrbracket/(xy-c)\xrightarrow{\sim}R_\rbarss^{\ver,\psi}/I_\irr^{\ver}.$$ Furthermore, $c$ is a regular element in $R_\rbar^{\univ,\psi}/I_\irr^{\univ}.$
\end{lem}
\begin{proof}
    All but the last statement are proved in \cite[Lemma 6.4, Lemma 6.5]{paskunas_2017_2_adic_deformations}. 
    By construction, $R_\rbar^{\univ,\psi}/I_\irr^{\univ}$ is $p$-torsion free. 
    To show that $c$ is not a zero-divisor in $R_\rbar^{\univ,\psi}/I_\irr^{\univ}$, it suffices to prove that $c$ is a unit in 
    \(
        (R_\rbar^{\univ,\psi}/I_\irr^{\univ})[1/p].
    \)
    Since all the $p$-adic representations $s_x$ parametrized by $(R_\rbar^{\univ,\psi}/I_\irr^{\univ})[1/p]$ are irreducible, and $c$ is the generator of the reducibility ideal, we have $f_x(c)\neq 0$ for every maximal ideal $\mathfrak{m}_x$ of $(R_\rbar^{\univ,\psi}/I_\irr^{\univ})[1/p]$. 
    Thus $c$ is not contained in any maximal ideal and hence is a unit.
\end{proof}

As before, let $\alpha_p \in R_\rbarss^{\ver,\psi}/I_1^{\ver}[1/p]$ denote the function that sends a crystalline representation $s_x$ of reducibility type $1$ to the trace of its crystalline Frobenius element. 
By $p$-adic Hodge theory, we have
\[
    \alpha_p(s_x) = \delta_2(\Frob_p) + p^{k-1}(\delta_1/\chi_\cyc^{k-1})(\Frob_p),
\]
so that $\alpha_p$ depends only on the semisimplification of $s_x$. 
Hence, $\alpha_p$ descends to
\[
    \Spf\Bigl(R^{{\rm ps},\psi} \cap \bigcap_{s_x \text{ of type }1} \mathfrak{m}_x \Bigr)^{\eta} 
    \xrightarrow{\sim} \Spf(R_\rbar^{\univ,\psi}/I_1^{\univ})^{\eta},
\]
and we obtain
\[
    (R_\rbar^{\univ,\psi}/I_1^\univ)[\alpha_p]\llbracket y\rrbracket \xrightarrow{\sim} (R_\rbarss^{\ver,\psi}/I_1^{\ver})[\alpha_p].
\]
The same reasoning applies in the irreducible case, where the argument is simpler because $s_x$ is already semisimple.

\begin{cor}\label{cor_versal_universal_ap_presentation}
    Suppose that $\rbar$ is non-semisimple and satisfies Assumption \ref{assumption_semisimplification}. For every $k\ge 2,$ the isomorphisms in Lemma \ref{lem_paskunas_versal_universal_crystalline} extend to $$(R_\rbar^{\univ,\psi}/I_1^\univ)[\alpha_p]\llbracket y\rrbracket\xrightarrow{\sim}(R_\rbarss^{\ver,\psi}/I_1^{\ver})[\alpha_p]$$ and $$(R_\rbar^{\univ,\psi}/I_\irr^\univ)[\alpha_p]\llbracket x,y\rrbracket/(xy-c)\xrightarrow{\sim}(R_\rbarss^{\ver,\psi}/I_\irr^{\ver})[\alpha_p].$$  
\end{cor}

Furthermore, applying the proof of Proposition \ref{prop_R_ap_decomposition} to $\rbar$ and $\rbarss$, we conclude the following. 

\begin{prop} \label{prop_versal_univ_crystalline_ap}
    Suppose that $\rbar$ is non-semisimple and satisfies Assumption \ref{assumption_semisimplification}. For every $k\ge 2,$ we have $$R_\rbarss^{\ver,\psi}(k)[\alpha_p] \xrightarrow{\sim} (R_\rbarss^{\ver,\psi}/I_1^{\ver})[\alpha_p]\oplus (R_\rbarss^{\ver,\psi}/I_2^{\ver})[\alpha_p]\oplus (R_\rbarss^{\ver,\psi}/I_\irr^{\ver})[\alpha_p]$$ and $$R_\rbar^{\univ,\psi}(k)[\alpha_p]\xrightarrow{\sim} (R_\rbar^{\univ,\psi}/I_1^{\univ})[\alpha_p]\oplus (R_\rbar^{\univ,\psi}/I_\irr^{\univ})[\alpha_p].$$ Following the notation at the end of \S \ref{subsubsection_crystalline_representations}, we have $$R_\rbarss^{\ver,\psi}(k)^{\ord}\xrightarrow{\sim} R_\rbarss^{\ver,\psi}/I_1^\ver\cap I_2^\ver,\quad R_\rbarss^{\ver,\psi}(k)^{\no}\xrightarrow{\sim}R_\rbarss^{\ver,\psi}/I_\irr^\ver,$$ $$R_\rbar^{\univ,\psi}(k)^{\ord} \xrightarrow{\sim}R_\rbar^{\univ,\psi}/I_1^{\univ}\quad{\rm and\quad}R_\rbar^{\univ,\psi}(k)^{\no}\xrightarrow{\sim}R_\rbar^{\univ,\psi}/I_\irr^{\univ}.$$ 
\end{prop}

\begin{prop} \label{prop_semisimplification_CM}
    Suppose that Assumption \ref{assumption_semisimplification} holds and $k \ge 2$. 
    If $R_\rbarss^{\Box,\psi}(k)[\alpha_p]$ is Cohen--Macaulay, then so is $R_\rbar^{\Box,\psi}(k)[\alpha_p]$.
\end{prop}

\begin{proof}
    We may assume $\rbar$ is not semisimple. By Proposition \ref{prop_lifting_ring_formally_smooth_over_versal_ring}, it suffices to show that 
    $R_\rbar^{\univ,\psi}(k)[\alpha_p]$ is Cohen--Macaulay if $R_\rbarss^{\ver,\psi}(k)[\alpha_p]$ is Cohen--Macaulay. 
    By the decomposition in Proposition \ref{prop_versal_univ_crystalline_ap}, it further suffices to show that 
    $(R^{\univ,\psi}_\rbar/I_*^{\univ})[\alpha_p]$ is Cohen--Macaulay if $(R^{\ver,\psi}_\rbarss/I_*^{\ver})[\alpha_p]$ is Cohen--Macaulay 
    for $* \in \{1,\irr\}.$
    The conclusion is immediate for $* = 1$ by Corollary \ref{cor_versal_universal_ap_presentation}. 
    When $* = \irr$, we need to show that $xy - c$ is regular in $R_\rbar^{\univ,\psi}/I_\irr^{\univ}\llbracket x,y\rrbracket$. 
    Suppose that 
    $$
        (xy-c)\sum_{i,j=0}^{\infty} a_{i,j} x^i y^j = 0.
    $$
    Then we have 
    $$
        a_{i-1,j-1} = c a_{i,j} \quad \text{for all } i,j \ge 1, \qquad 
        a_{i,0} = a_{0,j} = 0 \quad \text{for all } i,j \ge 0.
    $$
    As a result, 
    $$
        c^{\min\{i,j\}} a_{i,j} = a_{i-\min\{i,j\}, j-\min\{i,j\}} = 0 \quad \text{for all } i,j \ge 0.
    $$
    Since $c$ is regular in $R_\rbar^{\univ,\psi}/I_\irr^{\univ}$ by Lemma \ref{lem_paskunas_versal_universal_crystalline}, 
    we conclude that $a_{i,j} = 0$ for all $i,j \ge 0$, and thus $xy-c$ is regular.  
\end{proof}

\begin{prop} \label{prop_versal_universal_ap_in_R_equivalent}
    Suppose that $\rbar$ is non-semisimple and satisfies Assumption \ref{assumption_semisimplification}. Let $k \ge 2$. 
    \begin{enumerate}
        \item If $\chi_2$ is unramified, then 
        $$
            R_\rbar^{\Box,\psi}(k) \xrightarrow{\sim} R_\rbar^{\Box,\psi}(k)[\alpha_p]
        $$ 
        if and only if the non-ordinary part satisfies 
        $$
            R_\rbar^{\Box,\psi}(k)^{\no} = 0.
        $$
        
        \item If both $\chi_1$ and $\chi_2$ are ramified, then 
        $$
            R^{\Box,\psi}_{\rbarss}(k) \xrightarrow{\sim} R^{\Box,\psi}_\rbarss(k)[\alpha_p]
        $$ 
        if and only if 
        $$
            R^{\Box,\psi}_{\rbar}(k) \xrightarrow{\sim} R^{\Box,\psi}_\rbar(k)[\alpha_p].
        $$
        
        \item If $\chi_1$ is unramified and $\chi_2$ is ramified, then 
        $$
            R^{\Box,\psi}_{\rbarss}(k)^{\no} \xrightarrow{\sim} R^{\Box,\psi}_\rbarss(k)^{\no}[\alpha_p]
        $$ 
        if and only if 
        $$
            R^{\Box,\psi}_{\rbar}(k) \xrightarrow{\sim} R^{\Box,\psi}_\rbar(k)[\alpha_p].
        $$
    \end{enumerate}
\end{prop}

\begin{proof}
    By Proposition \ref{prop_lifting_ring_formally_smooth_over_versal_ring}, it suffices to prove these statements for $R_\rbar^{\univ,\psi}(k)$ and $R_\rbarss^{\ver,\psi}(k).$ 
    \begin{enumerate}
        \item If $R_\rbar^{\univ,\psi}(k)\xrightarrow{\sim}R_\rbar^{\univ,\psi}(k)[\alpha_p]$, then $R_\rbar^{\univ,\psi}(k)[\alpha_p]$ must be a local ring. Since $I_1^{\univ}$ is always nontrivial, it follows that $I_\irr^{\univ}=(1)$. Thus 
$$
R_\rbar^{\univ,\psi}(k)^{\no}=0.
$$
Conversely, if $I_\irr^{\univ}=(1)$, then 
$$
R_\rbar^{\univ,\psi}(k)\xrightarrow{\sim}R_\rbar^{\univ,\psi}(k)^{\ord}.
$$
Therefore, it suffices to prove that
$$
R_\rbar^{\univ,\psi}(k)^{\ord}\xrightarrow{\sim} R_\rbar^{\univ,\psi}(k)^{\ord}[\alpha_p].
$$
This follows because the above ring homomorphism is finite and the ordinary deformation ring $R_\rbar^{\univ,\psi}(k)^{\ord}$ is regular in this case by \cite[Lemma 7.3.7(1)]{BCGP_2021_abelian_surfaces_potentially_modular}. (Although the lemma is stated for a particular shape of $\rbar$, one checks that the same proof applies in our situation.)
        \item In this case, both $I_1^{\ver}$ and $I_1^{\univ}$ are $(1)$. As a result, we have
$$
R^{\univ,\psi}_\rbar(k) \xrightarrow{\sim} R_\rbar^{\univ,\psi}/I_\irr^{\univ}
\quad{\rm and\quad}
R^{\ver,\psi}_\rbarss(k) \xrightarrow{\sim} R_\rbarss^{\ver,\psi}/I_\irr^{\ver}.
$$
By Lemma \ref{lem_paskunas_versal_universal_crystalline}, these two rings are related by the isomorphism
$$
R_\rbarss^{\ver,\psi}(k)\xrightarrow{\sim}R_\rbar^{\univ,\psi}(k)\llbracket x,y\rrbracket /(xy-c).
$$
By Proposition \ref{prop_versal_univ_crystalline_ap}, we also have
$$
R_\rbarss^{\ver,\psi}(k)[\alpha_p]\xrightarrow{\sim}R_\rbar^{\univ,\psi}(k)[\alpha_p]\llbracket x,y\rrbracket /(xy-c).
$$ If $R_\rbar^{\univ,\psi}(k)\xrightarrow{\sim}R_\rbar^{\univ,\psi}(k)[\alpha_p]$, then it is clear that
$$
R^{\ver,\psi}_{\rbarss}(k) \xrightarrow{\sim}R^{\ver,\psi}_\rbarss(k)[\alpha_p].
$$
Conversely, since $R_\rbar^{\univ,\psi}(k)$ is a subring of $R_\rbar^{\univ,\psi}(k)[\alpha_p]$ by definition, it suffices to prove that
$$
R^{\univ,\psi}_{\rbar}(k) \surj R^{\univ,\psi}_\rbar(k)[\alpha_p]
$$
assuming that
$$
R^{\ver,\psi}_{\rbarss}(k) \xrightarrow{\sim}R^{\ver,\psi}_\rbarss(k)[\alpha_p].
$$ Since $c$ lies in the maximal ideal of $R_\rbar^{\univ,\psi}(k)$, it suffices to prove that
$$
R^{\univ,\psi}_{\rbar}(k)/(c) \to R^{\univ,\psi}_\rbar(k)[\alpha_p]/(c)
$$
is surjective by Nakayama's lemma for complete $R^{\univ,\psi}_\rbar$-modules. This follows because
\begin{align*}
R_\rbar^{\univ,\psi}(k)/(c)
&\xrightarrow{\sim}(R_\rbar^{\univ,\psi}(k)/(c))\llbracket x,y\rrbracket/(xy,x,y)
\xrightarrow{\sim}R^{\ver,\psi}_{\rbarss}(k)/(c,x,y) \\
&\xrightarrow{\sim}R^{\ver,\psi}_\rbarss(k)[\alpha_p]/(c,x,y)
\xrightarrow{\sim} (R_\rbar^{\univ,\psi}(k)[\alpha_p]/(c))\llbracket x,y\rrbracket/(xy,x,y)\\
&\xrightarrow{\sim} R_\rbar^{\univ,\psi}(k)[\alpha_p]/(c).
\end{align*}
        \item In this case, we have $I_1^{\univ}=(1)$ and hence
$$
R_\rbar^{\univ,\psi}(k)\xrightarrow{\sim}R_\rbar^{\univ,\psi}(k)^\no\xrightarrow{\sim}R_\rbar^{\univ,\psi}/I_\irr^{\univ}.
$$
Moreover,
$$
R_\rbarss^{\ver,\psi}(k)^{\no} \xrightarrow{\sim} R_\rbarss^{\ver,\psi}/I_\irr^{\ver}.
$$
It follows from Lemma \ref{lem_paskunas_versal_universal_crystalline} that
$$
R_\rbarss^{\ver,\psi}(k)^{\no} \xrightarrow{\sim} 
R_\rbar^{\univ,\psi}(k)\llbracket x,y\rrbracket /(xy-c).
$$
By Corollary \ref{cor_versal_universal_ap_presentation}, we have
$$
R_\rbarss^{\ver,\psi}(k)^{\no}[\alpha_p] 
\xrightarrow{\sim}
R_\rbar^{\univ,\psi}(k)[\alpha_p]\llbracket x,y\rrbracket /(xy-c).
$$
The argument now proceeds exactly as in the previous case.
    \end{enumerate}  
\end{proof}

After base change to $\F$, we obtain an isomorphism
$$
R_\rbarss^{\ver,\psi}\otimes_\Oh\F \xrightarrow{\sim} 
(R^{\univ,\psi}_\rbar\otimes_\Oh\F)\llbracket x,y\rrbracket/(xy-c).
$$
By Theorem \ref{thm_unrestricted_deformation_ring_at_p} and Proposition 
\ref{prop_lifting_ring_formally_smooth_over_versal_ring}, the element 
$c\in R_\rbar^{\univ,\psi}\otimes_\Oh \F$ is nonzero and thus regular.
Hence, we have an inclusion
$$
R_\rbar^{\univ,\psi}\otimes_\Oh \F 
\hookrightarrow 
(R^{\univ,\psi}_\rbar\otimes_\Oh\F)\llbracket x,y\rrbracket/(xy-c).
$$
Let $\mathfrak{m}_{1}$ denote the maximal ideal of 
$R_\rbar^{\univ,\psi}\otimes_\Oh \F$, and let $\mathfrak{m}_2$ denote the maximal ideal of 
$R_\rbarss^{\ver,\psi}\otimes_\Oh \F$.

\begin{prop} \label{prop_adic_topology_equivalent}
On $R_\rbar^{\univ,\psi}\otimes_\Oh \F$, the $\mathfrak{m}_1$-adic topology coincides with the topology defined by the ideals $\{\mathfrak{m}_2^{n}\cap R_\rbar^{\univ,\psi}\otimes_\Oh\F\}_{n\ge0}$.
\end{prop}

\begin{proof}
    Since $\mathfrak{m}_1 = \mathfrak{m}_2\cap (R_\rbar^{\univ,\psi}\otimes_\Oh\F)$, we have
\[
\mathfrak{m}_1^n\subseteq \mathfrak{m}_2^n\cap (R_\rbar^{\univ,\psi}\otimes_\Oh\F).
\]
Hence it suffices to show that for every $n\ge 1$ there exists $k\gg0$ such that
\[
\mathfrak{m}_2^k\cap (R_\rbar^{\univ,\psi}\otimes_\Oh\F)\subseteq \mathfrak{m}_1^n.
\]
Let $z=x+y$. Then
\[
(R^{\univ,\psi}_\rbar\otimes_\Oh\F)\llbracket x,y\rrbracket/(xy-c)
\xrightarrow{\sim}
(R^{\univ,\psi}_\rbar\otimes_\Oh\F)\llbracket z\rrbracket[x]/(x^2-zx+c).
\]
Thus we obtain a composite of maps
\[
R^{\univ,\psi}_\rbar\otimes_\Oh\F
\hookrightarrow
(R^{\univ,\psi}_\rbar\otimes_\Oh\F)\llbracket z\rrbracket
\hookrightarrow
(R^{\univ,\psi}_\rbar\otimes_\Oh\F)\llbracket z\rrbracket[x]/(x^2-zx+c).
\]
Suppose that for every $n>0$ there exists a sufficiently large $k\gg0$ such that
\[
\mathfrak{m}_2^k \cap (R^{\univ,\psi}_\rbar\otimes_\Oh\F)\llbracket z\rrbracket
\subseteq
(\mathfrak{m}_1+(z))^n.
\]
Then
\[
\mathfrak{m}_2^k\cap (R^{\univ,\psi}_\rbar\otimes_\Oh\F)
\subseteq
(\mathfrak{m}_1+(z))^n\cap (R^{\univ,\psi}_\rbar\otimes_\Oh\F)
=
\mathfrak{m}_1^n.
\]
Therefore it suffices to prove that the $(\mathfrak{m}_1+(z))$-adic topology and the topology defined by the ideals $\{\mathfrak{m}_2^n\cap ((R^{\univ,\psi}_\rbar\otimes_\Oh\F)\llbracket z\rrbracket)\}_{n\ge 1}$ coincide.
This follows from the following commutative algebra proposition.
\end{proof}

\begin{prop}
Let $(A,\mathfrak{m}_A)$ be an object of $\mathcal{C}_\Oh$ which is also an $\F$-algebra, and let
\[
B = A[x]/(x^2 - zx + c)
\]
where $c,z \in \mathfrak{m}_A$. Then $B$ is finite free over $A$ of rank $2$. Moreover, $B$ is an object of $\mathcal{C}_\Oh$ with maximal ideal
\[
\mathfrak{m}_B = \mathfrak{m}_A B + (x).
\]
Furthermore, the topology on $A$ defined by the ideals $\{\mathfrak{m}_B^n\cap A\}_{n\ge 1}$ coincides with the $\mathfrak{m}_A$-adic topology.
\end{prop}

\begin{proof}
    The freeness statement follows from \cite[Lemma 6.2]{paskunas_2017_2_adic_deformations}. Thus we can write $B = A \oplus Ax$. Since $\mathfrak{m}_A \subseteq \mathfrak{m}_B \cap A$, we have
    $$\mathfrak{m}_A^n \subseteq \mathfrak{m}_B^n \cap A.$$
    Therefore it suffices to show that, for every $n > 0$, there exists $k \gg 0$ such that $\mathfrak{m}_B^{k} \cap A \subseteq \mathfrak{m}_A^n$. We claim that
    $$\mathfrak{m}_B^{k} \subseteq \mathfrak{m}_A^{\lceil k/2 \rceil}A \oplus \mathfrak{m}_A^{\lfloor k/2 \rfloor}Ax.$$
    It then follows that
    $$\mathfrak{m}_B^{k} \cap A \subseteq \mathfrak{m}_A^{\lceil k/2 \rceil},$$
    and the proof will be complete.
    When $k = 1$, this is clear because
    $$\mathfrak{m}_B = \mathfrak{m}_A A \oplus A x.$$
    Suppose that we have established this for $\mathfrak{m}_B^{k-1}$. Then
    \begin{align*}
        \mathfrak{m}_B^k &= \mathfrak{m}_B\cdot  \mathfrak{m}_B^{k-1} \\
        &\subseteq (\mathfrak{m}_A A \oplus A x)(\mathfrak{m}_A^{\lceil (k-1)/2 \rceil}A \oplus \mathfrak{m}_A^{\lfloor (k-1)/2 \rfloor}A x) \\
        &\subseteq \mathfrak{m}_A^{\lceil (k+1)/2 \rceil}A + \mathfrak{m}_A^{\lceil (k-1)/2 \rceil}Ax + \mathfrak{m}_A^{\lfloor (k+1)/2 \rfloor}Ax + \mathfrak{m}_A^{\lfloor (k-1)/2 \rfloor}Ax^2.
    \end{align*}
    Since
    $$\lceil (k-1)/2 \rceil \le \lfloor (k+1)/2 \rfloor$$
    and
    $$x^2 = z x - c,$$
    we deduce that
    \begin{align*}
        \mathfrak{m}_B^k \subseteq \mathfrak{m}_A^{\lceil (k+1)/2 \rceil}A + \mathfrak{m}_A^{\lceil (k-1)/2 \rceil}Ax + \mathfrak{m}_A^{\lfloor (k-1)/2 \rfloor}z Ax + \mathfrak{m}_A^{\lfloor (k-1)/2 \rfloor}cA.
    \end{align*}
    Since both $c$ and $z$ lie in $\mathfrak{m}_A$, we have
    \begin{align*}
        \mathfrak{m}_B^k &\subseteq \mathfrak{m}_A^{\lceil (k+1)/2 \rceil}A + \mathfrak{m}_A^{\lceil (k-1)/2 \rceil}Ax + \mathfrak{m}_A^{\lfloor (k+1)/2 \rfloor}Ax + \mathfrak{m}_A^{\lfloor (k+1)/2 \rfloor}A \\
        &\subseteq \mathfrak{m}_A^{\lfloor (k+1)/2 \rfloor}A \oplus \mathfrak{m}_A^{\lceil (k-1)/2 \rceil}xA.
    \end{align*}
    Since $\lfloor (k+1)/2 \rfloor = \lceil k/2 \rceil$ for every integer $k$, we obtain
    $$\mathfrak{m}_B^k \subseteq \mathfrak{m}_A^{\lceil k/2 \rceil}A \oplus \mathfrak{m}_A^{\lfloor k/2 \rfloor}xA.$$
\end{proof} 

\begin{prop} \label{prop_thm_crystalline_stronger_topology_reduction_to_semisimple_case}
    Suppose that Assumption \ref{assumption_semisimplification} holds. If Theorem \ref{thm_crystalline_stronger_topology} holds for $\rbarss$ with $$J_k\supseteq \ker(R_\rbarss^{\Box,\psi}\otimes_\Oh \F\to R_\rbarss^{\Box,\psi}(k)^{\no}[\alpha_p]\otimes_\Oh \F),$$ then it holds for $\rbar.$ 
\end{prop}

\begin{proof}
    We may assume that $\rbar$ is not semisimple. By Proposition \ref{prop_lifting_ring_formally_smooth_over_versal_ring}, we may replace $R_\rbar^{\Box,\psi}$ with $R_\rbar^{\univ,\psi}$ and $R_\rbarss^{\Box,\psi}$ with $R_\rbarss^{\ver,\psi}.$ By assumption, we have $\{J_k\}$ that is a descending sequence of ideals in $R^{\ver,\psi}_\rbarss\otimes_\Oh \F$ such that $$\ker(R^{\ver,\psi}_\rbarss\otimes_\Oh \F \to R^{\ver,\psi}_\rbarss(k)^\no[\alpha_p]\otimes_\Oh \F)\subseteq J_k$$ for every and $\{J_k\}$ defines a stronger topology than the $\mathfrak{m}_2$-adic topology. Consider the sequence $\{J_k\cap R_\rbar^{\univ,\psi}\otimes_\Oh \F\}.$ It is clearly a descending sequence. Furthermore, by Lemma \ref{lem_paskunas_versal_universal_crystalline}, we have the commutative diagram $$\begin{tikzcd}
        R^{\univ,\psi}_\rbar\otimes_\Oh\F \arrow[r, two heads]\arrow[d,hook] & R^{\univ,\psi}_\rbar(k)\otimes_\Oh\F \arrow[r, two heads] & R^{\univ,\psi}_\rbar(k)^{\no}[\alpha_p]\otimes_\Oh\F\arrow[d]\\
        R_\rbarss^{\ver,\psi}\otimes_\Oh\F \arrow[r,two heads]& R^{\univ,\psi}_\rbar(k)\otimes_\Oh\F \arrow[r,two heads]& R_\rbarss^{\ver,\psi}(k)^{\no}[\alpha_p]\otimes_\Oh\F. 
    \end{tikzcd}$$ A little diagram chasing shows that \begin{align*}
    \ker(R_\rbar^{\univ,\psi}\otimes_\Oh\F\to R^{\univ,\psi}_\rbar(k)\otimes_\Oh\F)&\subseteq \ker(R_\rbar^{\univ,\psi}\otimes_\Oh\F\to R^{\univ,\psi}_\rbar(k)^\no[\alpha_p]\otimes_\Oh\F)\\&\subseteq \ker(R_\rbarss^{\ver,\psi}\otimes_\Oh\F\to R^{\ver,\psi}_\rbarss(k)^\no[\alpha_p]\otimes_\Oh\F) \subseteq J_k.
    \end{align*} Thus $$\ker(R_\rbar^{\univ,\psi}\otimes_\Oh\F\to R^{\univ,\psi}_\rbar(k)\otimes_\Oh\F)\subseteq J_k\cap (R_\rbar^{\univ,\psi}\otimes_\Oh \F).$$
    Since $\{J_k\}$ defines a stronger topology than the $\mathfrak{m}_2$-adic topology, $\{J_k\cap R_\rbar^{\univ,\psi}\otimes_\Oh,\F\}$ defines a stronger topology than the topology defined by $\{\mathfrak{m}_2^{n}\cap (R_\rbar^{\univ,\psi}\otimes_\Oh\F)\}_{n\ge 1},$ which is equivalent to the $\mathfrak{m}_1$-adic topology on $R_\rbar^{\univ,\psi}\otimes_\Oh\F$ by Proposition \ref{prop_adic_topology_equivalent}.
\end{proof}

\begin{proof}
    By Proposition \ref{prop_lifting_ring_formally_smooth_over_versal_ring}, we may replace $R_\rbar^{\Box,\psi}$ with $R_\rbar^{\univ,\psi}$ and $R_\rbarss^{\Box,\psi}$ with $R_\rbarss^{\ver,\psi}$. By assumption, we have a descending sequence of ideals $\{J_k\}$ in $R^{\ver,\psi}_\rbarss \otimes_\Oh \F$ such that
    $$\ker(R^{\ver,\psi}_\rbarss \otimes_\Oh \F \surj R^{\ver,\psi}_\rbarss(k)^\no[\alpha_p] \otimes_\Oh \F) \subseteq J_k,$$
    and $\{J_k\}$ defines a stronger topology than the $\mathfrak{m}_2$-adic topology on $R_\rbarss^{\ver,\psi}.$

    Consider the sequence $\{J_k \cap (R_\rbar^{\univ,\psi} \otimes_\Oh \F)\}$. It is clearly a descending sequence. Furthermore, by Lemma \ref{lem_paskunas_versal_universal_crystalline}, we have the commutative diagram
    $$
    \begin{tikzcd}
        R^{\univ,\psi}_\rbar \otimes_\Oh \F \arrow[r, two heads]\arrow[d,hook] & R^{\univ,\psi}_\rbar(k) \otimes_\Oh \F \arrow[r, two heads] & R^{\univ,\psi}_\rbar(k)^{\no}[\alpha_p] \otimes_\Oh \F \arrow[d] \\
        R_\rbarss^{\ver,\psi} \otimes_\Oh \F \arrow[r,two heads] & R^{\univ,\psi}_\rbar(k) \otimes_\Oh \F \arrow[r,two heads] & R_\rbarss^{\ver,\psi}(k)^{\no}[\alpha_p] \otimes_\Oh \F. 
    \end{tikzcd}
    $$
    A brief diagram chase shows that
    \begin{align*}
        \ker(R_\rbar^{\univ,\psi} \otimes_\Oh \F \to R^{\univ,\psi}_\rbar(k) \otimes_\Oh \F)
        &\subseteq \ker(R_\rbar^{\univ,\psi} \otimes_\Oh \F \to R^{\univ,\psi}_\rbar(k)^\no[\alpha_p] \otimes_\Oh \F) \\
        &\subseteq \ker(R_\rbarss^{\ver,\psi} \otimes_\Oh \F \to R^{\ver,\psi}_\rbarss(k)^\no[\alpha_p] \otimes_\Oh \F) \subseteq J_k.
    \end{align*}
    Thus
    $$
    \ker(R_\rbar^{\univ,\psi} \otimes_\Oh \F \to R^{\univ,\psi}_\rbar(k) \otimes_\Oh \F)
    \subseteq J_k \cap (R_\rbar^{\univ,\psi} \otimes_\Oh \F).
    $$ Since $\{J_k\}$ defines a stronger topology than the $\mathfrak{m}_2$-adic topology, the sequence of ideals $\{J_k \cap (R_\rbar^{\univ,\psi} \otimes_\Oh \F)\}$ defines a stronger topology than the topology defined by $\{\mathfrak{m}_2^{n} \cap (R_\rbar^{\univ,\psi} \otimes_\Oh \F)\}_{n \ge 1}$, which is equivalent to the $\mathfrak{m}_1$-adic topology on $R_\rbar^{\univ,\psi} \otimes_\Oh \F$ by Proposition \ref{prop_adic_topology_equivalent}.
\end{proof}

\subsection{Hecke Modules and Hecke Algebras}\label{subsection_hecke_modules_and_algebras}
In this subsection, we introduce the Hecke modules and Hecke algebras that will be used for ultrapatching. 

Let $\rbar:G_{\Q_p}\to \GL_2(\F)$ be a Galois representation satisfying Assumption \ref{assumption_global_rhobar}. 
By the strong form of Serre's conjecture, which is now a theorem by \cite{KhareWintenberger2009serreI,KhareWintenberger2009serreII}, the first two conditions in Assumption \ref{assumption_global_rhobar} imply that $\rhobar$ comes from a Katz mod-$p$ modular form $f_0$ of level $U_1(N(\rhobar))$ and weight $k(\rbar).$ Here $k(\rbar)$ is the minimal weight among all the weights of Katz modular forms that are congruent to $f_0$ modulo $\varpi,$ which is specified and proved in \cite[\S 4]{edixhoven1992weight}. We will recall the definition of $k(\rbar)$ in Definition \ref{def_Serre_weight}.

We let $$\chi:=\det(\rhobar)/\epsilon^{k(\rbar)-1} : G_\Q\surj (\Z/N(\rhobar)\Z)^{\times}\to \F^\times$$ be the character given by  and denote its Teichm\"uller lift by $\chi$ as well by abuse of notation. This $f_0$ has Nebentypus character $\chi.$ 

Let $N_{\varnothing}$ be $N(\rhobar)$ if $N(\rhobar)\ge 5$ and $N(\rhobar)q_0$ where $q_0\ge 5$ is a prime number that satisfies the conditions in Lemma \ref{lem_unramified_prime}. Thus $N_\varnothing$ is always big enough so that $U_1(N_\varnothing)$ has no torsion elements. 
If $Q$ is a finite set of primes $q$ such that $q\nmid N_\varnothing p$, $q\equiv1\pmod p$ and $\rhobar(\Frob_q)$ has distinct eigenvalues $\alpha_q\neq \beta_q$, we let $N_Q$ be the product of $N_\varnothing$ and all the primes in $Q.$ Extend $\Oh$ if necessary so that we fix an eigenvalue $\alpha_{q}\in \F^\times$ of $\rhobar(\Frob_{q})$ and a lift $\tilde{\alpha}_{q}\in \Oh^{\times}$ for every $q\in \{q_0\}\sqcup Q.$ 

Let $\Delta_Q$ be the maximal $p$-quotient of $\left(\Z/\left(\prod_{q\in Q}q\right)\Z\right)^{\times}$ and let $H_Q$ be the kernel of $$H_Q\ := \ker\left(\left(\Z/N_Q\Z\right)^{\times}\surj \left(\Z/\left(\prod_{q\in Q}q\right)\Z\right)^{\times}\surj \Delta_Q\right).$$ 
By abuse of notation, we also denote by $\chi$ the character
$$\chi: H_Q\surj \left(\Z/N(\rhobar)\Z\right)^{\times}\xrightarrow{\chi} \Oh^{\times}.$$ Since the $p$-part $P$ of $H_Q$ can only come from $\left(\Z/N(\rhobar)\Z\right)^{\times},$ we have $$U_{P}(N_Q) \subseteq U_{(\Z/N(\rhobar)\Z)^{\times}}(N_Q)\subseteq U_1(q_0)$$ and so $U_P(N_Q)$ is torsion free. In particular, Proposition \ref{prop_Nebentypus_base_change} applies. 

Let $\mathfrak{m}_Q^{\{N(\rhobar) p\}}$ be the maximal ideal of $\T^{N(\rhobar) p}$ that is given by $$\mathfrak{m}_Q^{\{N(\rhobar) p\}} := (\varpi, T_\ell-a_\ell(f_0),\langle\ell\rangle-\chi(\ell),T_q - \widetilde{\alpha}_q )_{\ell\nmid pN_Q,q\in \{q_0\}\sqcup Q}.$$  For every integer $0\le a\le p-2,$ we write $$\mathfrak{m}_Q^{\{N(\rhobar) p\}}(a) := (\varpi, T_\ell-\ell^a a_\ell(f_0),\langle\ell\rangle-\chi(\ell),T_q - q^a\widetilde{\alpha}_q )_{\ell\nmid pN_Q,q\in \{q_0\}\sqcup Q}.$$ This notation will be used in \S \ref{subsection_Serre_mult_one}. Set 
$$M(k,Q,R) := \left(\left(H^0(X_{U_1(N_Q),R},\omega^{\otimes k}_{U_1(N_Q),R})\otimes \chi\right)^{H_Q}\right)_{\mathfrak{m}_{Q}^{\{N(\rhobar)p\}}}$$ 
for $k\ge 2$ and $R=\Oh$ or $\F.$ By the usual idempotent argument, one sees that
$
M(k,Q,R)
$
is a direct summand of
$
\left(H^0(X_{U_1(N_Q),R},\omega^{\otimes k}_{U_1(N_Q),R})\otimes \chi\right)^{H_Q}
$
as a $\T^{N(\rhobar)p}$-module. Since the Hecke operators at primes dividing $N(\rhobar)p$ commute with the $\T^{N(\rhobar)p}$-action, it follows that $M(k,Q,R)$ is in fact a $\T$-direct summand. Thus we can decompose $M(k,Q,R)$ with respect to the Hecke operator at $p.$ We write $M(k,Q,R)^{\ord}$ for the ordinary part and $M(k,Q,R)^{\no}$ for the non-ordinary part and they are both $\T$-direct summands of $\left(H^0(X_{U_1(N_Q),R},\omega^{\otimes k}_{U_1(N_Q),R})\otimes \chi\right)^{H_Q}.$ In particular, $M(k,Q,R),M(k,Q,R)^{\ord}$ and $M(k,Q,R)^{\no}$ all satisfy the base-change statement of Proposition~\ref{prop_Nebentypus_base_change}. 
Furthermore, since $\mathfrak{m}_Q^{\{N(\rhobar)p\}}$ comes from an absolutely irreducible representation $\rhobar,$ the module $M(k,Q,R)$ is a submodule of the cusp forms $H^0(X_{U_1(N_Q),R},\omega^{\otimes k}_{U_1(N_1),R}(-\infty)).$ We set $$S(k,Q) := \left(\left(H^0(X_{U_1(N_Q),\F},\mathcal{S}^k_{U_1(N_Q),\F})\otimes \chi\right)^{H_Q}\right)_{\mathfrak{m}_{Q}^{\{N(\rhobar)p\}}}.$$ 

\begin{lem}\label{lem_finite_free_over_group_alg}
    Suppose that $k\ge 2,$ $R=\Oh$ or $\F$. Let $M(k,Q,R)^*$ denote any one of the three modules
    $M(k,Q,R)$, $M(k,Q,R)^{\ord}$, or $M(k,Q,R)^{\no}$. Then the Hecke algebra $\T(M(k,Q,R)^*)$ is a finite free module over $R[\Delta_Q]$ such that
    $$
      \T(M(k,Q,R)^*) \otimes_{R[\Delta_Q]} R \xrightarrow{\sim} \T(M(k,\varnothing,R)^*).
    $$
    In particular, the $R[\Delta_Q]$-rank of $\T(M(k,Q,R)^*)$ is independent of $Q$.
\end{lem}
\begin{proof}
    This proof is the same as \cite[Theorem 3.31, Corollary 3.32]{DDT1995fermat}.
\end{proof}

The Hecke algebra $\T^{pN_Q}(M(k,Q,\Oh))$ is a complete local Noetherian $\Oh$-algebra with maximal ideal induced by
\[
\mathfrak{m}_Q^{\{pN_Q\}} := (\varpi, T_\ell - a_\ell(f_0), \langle \ell \rangle - \chi(\ell))_{\ell \nmid pN_Q} \subseteq \T^{pN_Q}.
\]
Similarly, for each integer $0 \le a \le p-2$, we define the \emph{twist version} of the maximal ideal by
\[
\mathfrak{m}_Q^{\{pN_Q\}}(a) := (\varpi, T_\ell - \ell^a a_\ell(f_0), \langle \ell \rangle - \chi(\ell))_{\ell \nmid pN_Q} \subseteq \T^{pN_Q}.
\]
This twist notation will be used in \S \ref{subsection_Serre_mult_one}.

The rest of this subsection is to prove the following:
\begin{lem}[Wiles]\label{lem_U_ell_in_anemic}
Let $M(k,Q,\Oh)^*$ denote any one of the three modules
    $M(k,Q,\Oh)$, $M(k,Q,\Oh)^{\ord}$, or $M(k,Q,\Oh)^{\no}$. The Hecke algebra $\T^{pN_Q}(M(k,Q,\Oh)^*)$ is equal to $\T^{p}(M(k,Q,\Oh)^*)$.
\end{lem}
\begin{remark}
    The lemma was first proved by Wiles in \cite[Proposition 2.15]{wiles1995modular}, with further analysis below \cite[(2.37)]{wiles1995modular}. See also \cite[Propositions 4.7 and 4.10]{DDT1995fermat}. For self-consistency, we include a proof adapted to our notation and setup; the argument is the same as in \cite{wiles1995modular} and \cite{DDT1995fermat}.
\end{remark}

The main idea is to use modular Galois deformations of $\rhobar$ as in Assumption \ref{assumption_global_rhobar}.
To construct such deformations, let us first recall the Galois representations attached to modular forms.  
Let $f$ be a cuspidal eigenform over $\Oh$ of weight $k$, level $U_1(N)$, and Nebentypus character $\chi_f: (\Z/N\Z)^\times \to \Oh^\times$, normalized so that $a_1(f) = 1$.

\begin{thm}[Deligne--Serre, Eichler--Shimura, Deligne]\label{thm_galois_reps_attached_to_modular_forms}
For every prime $p$, there exists a unique absolutely irreducible Galois representation
\[
\rho_f: G_\Q \to \GL_2(\Oh)
\] 
that is odd and unramified away from $pN$.  
For a prime $\ell \nmid pN$, the characteristic polynomial of $\rho_f(\Frob_\ell)$ is
\[
X^2 - a_\ell(f) X + \chi_f(\ell) \ell^{k-1}.
\]
\end{thm}

\begin{remark}\label{rmk_conductor_equals_new_level}
Let $g$ be the newform associated to $f$ of level $N_g$.  
Then $\rho_f$ is isomorphic to $\rho_g$ by Chebotarev's density theorem.  
For newforms, Carayol proved that the Artin conductor of $\rho_g$ equals the level $N_g$ of $g$ (see \cite{carayol_1986_hilbert}).
\end{remark}
This theorem describes the local behavior of $\rho_f$ at primes $\ell \nmid pN$ and at infinity.  
When $\ell = p$ and $p \nmid N$, the local behavior is described in Theorem \ref{thm_Scholl_rhof_crystalline_at_p}.  
We are also interested in the restriction $\rho_f|_{G_{\Q_\ell}}$ for primes $\ell \mid N$.  
All of these cases are summarized by the local-global compatibility theorem for modular forms, which we now recall.

For an irreducible smooth automorphic representation $\pi_\ell$ of $\GL_2(\Q_\ell)$, we write 
$
\rec(\pi_\ell)
$
for the semisimple Weil--Deligne representation of the Weil group $W_\ell$ of $\Q_\ell$ associated via the local Langlands correspondence (following the convention in \cite[\S4.1]{CDT1999Barsotti}).  For a Galois representation 
\[
\rho_\ell: G_{\Q_\ell} \to \GL_2(L),
\] 
if $\ell \neq p$, we denote by $\WD(\rho_\ell)$ the Weil--Deligne representation associated to $\rho_\ell$ via Grothendieck's monodromy theorem.  
If $\ell = p$ and $\rho_p$ is crystalline, the Weil--Deligne representation $\WD(\rho_p)$ is defined as in \S \ref{subsubsection_crystalline_representations}.  
We write $(\cdot)^\semisimple$ for semisimplification.

Given a modular form $f$ as above, there exists a unique automorphic representation
\[
\pi_f: \GL_2(\A_\Q^\infty) \to \GL(V),
\]
where $\A_\Q^\infty$ is the ring of finite ad\`{e}les of $\Q$ and $V$ is an infinite-dimensional $\C$-vector space.  
Fixing a non-canonical isomorphism $\C \xrightarrow{\sim} \overline{\Q}_p$, we may regard $V$ as a $\overline{\Q}_p$-vector space, and we write $\pi_{f,\ell} := \pi_f|_{\GL_2(\Q_\ell)}$ for its local component at $\ell$.

\begin{thm}[Local-Global Compatibility]
Suppose that $p \nmid N$. Then, for all primes $\ell$, we have
\[
\WD(\rho_f|_{G_{\Q_\ell}})^\semisimple \xrightarrow{\sim} \rec(\pi_{f,\ell}).
\]
\end{thm}

\begin{prop}\label{prop_basis_of_M}
The Hecke module $M(k,Q,\Oh)$ contains a basis 
$
B(k,Q)$ of $$M(k,Q,L) := M(k,Q,\Oh) \otimes_\Oh L$$ 
consisting of eigenvectors $f$ for the full Hecke algebra $\T$ with leading coefficient $a_1(f) = 1$.  
In particular, for each $f \in B(k,Q)$ and every positive integer $n$, we have
\[
T_n(f) = a_n(f).
\]
\end{prop}

\begin{proof}
    Rationally, $M(k,Q,L)$ is a $\T$-submodule of 
\[
\left(H^0(X_{U_1(N_Q),L},\omega^{\otimes k}_{U_1(N_Q),L})\otimes \chi\right)^{H_Q} \subset H^0(X_{U_1(N_Q),L},\omega^{\otimes k}_{U_1(N_Q),L}).
\] 
Since we are localizing at a maximal ideal corresponding to an absolutely irreducible $\rhobar$, it follows from \cite[Theorem 3.30(2)]{calegari_geraghty_2018_beyond} that 
\[
M(k,Q,L) \subset H^0(X_{U_1(N_Q),L},\omega^{\otimes k}_{U_1(N_Q),L}(-\infty)).
\] By Theorem \ref{thm_oldforms_as_newform_span}, $M(k,Q,L)$ has a basis of eigenforms $f$ for $\T^{N_Q}$, which we may scale to be in $M(k,Q,\Oh)$.  

If $f$ is a newform of level $N_Q$, it is automatically an eigenvector for $\T$ by Theorem \ref{thm_strong_multiplicity_one}. Otherwise, $f$ comes from a newform $g$ of level $N_g \mid N_Q$. By Remark \ref{rmk_conductor_equals_new_level}, $\rho_f$ has Artin conductor $N_g$. By the definition of $M(k,Q,\Oh)$, we have $\rhobar_f \xrightarrow{\sim} \rhobar$, and hence
$
N(\rhobar) \mid N_g \mid N_Q.
$
It follows that $f$ can only be old at primes $q \in \{q_0\} \sqcup Q$, each of which satisfies $q \parallel N_Q$ and $\rhobar(\Frob_q)$ has distinct eigenvalues.  

By Proposition \ref{prop_eigenform_new_old_at_q} and the definition of $N_Q$, $f$ is an eigenvector for the Hecke operators away from these primes, and at each $q$ it is annihilated by
\[
U_q^2 - a_q(g) U_q + \chi_g(q) q^{k-1}.
\]
Since $q \nmid N_g$, $\rho_g$ is unramified at $q$, so $a_q(g) = \tr \rho_g(\Frob_q)$ reduces to $\tr(\rhobar(\Frob_q))$ modulo $\varpi$. In characteristic $p$, the polynomial $X^2 - a_q(g) X + \chi(q) q^{k-1}$ has two distinct roots $\alpha_q \neq \beta_q$. By Hensel's lemma, we can factor
\[
U_q^2 - a_q(g) U_q + \chi_g(q) q^{k-1} = (U_q - \widetilde{\alpha}_q)(U_q - \widetilde{\beta}_q)
\]
with $\widetilde{\alpha}_q, \widetilde{\beta}_q$ lifting $\alpha_q, \beta_q$, respectively. Since $U_q - \widetilde{\beta}_q$ is invertible on $M(k,Q,\Oh)$, we deduce
\[
(U_q - \widetilde{\alpha}_q) f = 0.
\]

Finally, we check $a_1(f) \neq 0$. Suppose not. Let $\lambda_f(n)$ be the eigenvalue of $T_n$ for $f$. By the $q$-expansion formula,
\[
a_1(T_n f) = a_n(f) \quad \text{and} \quad a_1(T_n f) = a_1(\lambda_f(n) f) = \lambda_f(n) a_1(f) = 0.
\] 
Hence $a_n(f) = 0$ for all $n$, forcing $f=0$ by the $q$-expansion principle \ref{thm_q_expansion_principle}. Scaling $f$ so that $a_1(f) = 1$, we obtain 
\[
a_n(f) = a_1(T_n f) = a_1(\lambda_f(n) f) = \lambda_f(n) a_1(f) = \lambda_f(n)\in \Oh,
\] 
as desired.
    
\end{proof}

Using the proposition above, we may map $\T(M(k,Q,\Oh))$ to a sub-algebra of $\Oh^{\oplus \# B(k,Q)}$ by sending $T_n$ to $$(\lambda_f(n))_{f\in B(k,Q)} = (a_n(f))_{f\in B(k,Q)}.$$ This map is an embedding because the Hecke action on $B(k,Q)$ determines the action on all of $M(k,Q,\Oh).$

We are now ready to describe modular deformation $\rho_Q$ of $\rhobar$ for $Q$ as before. 
Let $S_Q$ be the union of $\{p\}$, the prime divisors of $N_\varnothing$ and $Q.$ 
We have a Galois representation lifting $\rhobar$ \begin{align*}
\rho_{Q}^{\rm mod}: G_{\Q,S_Q}&\to\prod_{f \in B(k,Q)} \GL_2(\Oh) \xrightarrow{\sim} \GL_2(\Oh^{\oplus \#B(k,Q)})\\
\sigma &\mapsto (\rho_f(\sigma))_{f\in B(k,Q)}.
\end{align*} As a result of Theorem \ref{thm_galois_reps_attached_to_modular_forms}, it follows that $$\tr(\rho^{\rm mod}_Q(\Frob_\ell))=T_\ell{\rm \quad and\quad}\det(\rho^{\rm mod}_Q(\Frob_\ell)) = \langle\ell\rangle \chi_\cyc(\ell)^{k-1}$$ are in $\T^{pN_Q}(M(k,Q,\Oh)).$ By Chebotarev's density theorem, we have $$\{\tr(\rho^{\rm mod}_Q(\sigma)),\det(\rho^{\rm mod}_Q(\sigma)):\sigma\in G_{\Q,S_Q}\}= \T^{pN_Q}(M(k,Q_n,\Oh)).$$ By Carayol's lemma \cite[Theorem 2]{Carayol1994FormesME} that we may and do assume $\rho^{\rm mod}_Q$ has image in $\GL_2(\T^{pN_Q}(M(k,Q,\Oh)))$. In order to prove Lemma \ref{lem_U_ell_in_anemic}, it suffices to show that the Hecke operators at primes dividing $pN_Q$ can be generated by the trace and the determinant of $\rho_Q^{\rm mod}.$

\begin{prop}
    The Hecke operator $U_\ell$ is in $\T^{pN_Q}(M(k,Q,\Oh))$ for every $\ell|N(\rhobar).$
\end{prop}

\begin{proof}
    For $\ell|N(\rhobar)$, we have $\dim_\F \rhobar^{I_\ell} = \dim_L\rho_f^{I_\ell} = 1$ or $0$. If $\dim_\F \rhobar^{I_\ell} =1$, then $\rhobar|_{I_\ell}$ is reducible. 
    So $\rhobar|_{I_\ell}$ is either a non-split extension $\begin{pmatrix}
        1 & b\\
        0 &1
    \end{pmatrix}$ where $b$ defines a nontrivial class in $H^1(I_\ell,\F)$ or it is a direct sum of 1 and a nontrivial character by Proposition \ref{prop_dim_rhobar_I_ell_is_one}.
    
    In the first case, the residual representation $\rhobar|_{G_{\Q_\ell}}$ is tamely ramified.  
From the structure of the Galois group $G_{\Q_\ell}/P_\ell$ of the maximal tamely ramified extension, we then have
\[
\rhobar|_{G_{\Q_\ell}} \simeq \psi \otimes 
\begin{pmatrix}
\epsilon & b\\
0 & 1
\end{pmatrix},
\] 
where $\psi: G_{\Q_\ell} \to \F^\times$ is unramified and $b$ defines a class in $H^1(G_{\Q_\ell}, \epsilon)$ extending the aforementioned $b.$ 
In particular, $\det \rhobar$ and therefore $\chi$ are unramified at $\ell$.
For $f\in B(k,Q)$, we then have 
\[
n_\ell(\rho_f) = n_\ell(\rhobar) = 2 - 1 + sw(\rhobar) = 1,
\] 
and $\det(\rho_f)$ is unramified at $\ell$.  
By the classification of two-dimensional Frobenius-semisimple Weil--Deligne representations, 
\[
{\rm WD}(\rho_f|_{G_{\Q_\ell}})^{\rm ss} \simeq 
\left(
\eta_f \chi_\cyc \oplus \eta_f, 
\begin{pmatrix} 0 & 1 \\ 0 & 0 \end{pmatrix} 
\right),
\] 
where $\eta_f$ is an unramified character of $G_{\Q_\ell}$. Since $\Frob_\ell$ has distinct eigenvalues, the semisimplification is unnecessary. By Grothendieck's monodromy theorem, we have
\[
\rho_f|_{G_{\Q_\ell}} \simeq \eta_f \otimes 
\begin{pmatrix} \chi_\cyc & B \\ 0 & 1 \end{pmatrix},
\] 
where $B$ defines a non-trivial class in $H^1(G_\ell, \chi_\cyc) \setminus H^1(G_\ell/I_\ell, \chi_\cyc)$.  
It follows that 
\[
\eta_f^2(\Frob_\ell) = \det(\rho_f(\Frob_\ell))/\ell,
\] 
and that $\eta_f \chi_\cyc$ is the character through which $G_{\Q_\ell}$ acts on $\rho_f^{I_\ell}$.  
Moreover, the reduction of $\eta_f$ modulo $\varpi$ coincides with $\psi \epsilon$ on $\rhobar|_{I_\ell}$. 
    On the other hand, local--global compatibility implies that $\pi_{f,\ell}$ is the twisted Steinberg representation $\St(\eta_f)$, so that $\eta_f(\Frob_\ell) = a_\ell(g)$, where $g$ is the newform associated to $f$.  
Since $f$ is new at $\ell$, we have $a_\ell(f) = a_\ell(g)$, and hence the image of $U_\ell$ in $\Oh^{\oplus \#B(k,Q)}$ is
\[
(a_\ell(f))_{f \in B(k,Q)} = (\eta_f(\Frob_\ell))_{f \in B(k,Q)}.
\]  
Combining this with the previous analysis of $\eta_f$, we see that $U_\ell$ satisfies
\[
X^2 = \det(\rho_Q^{\rm mod}(\Frob_\ell))/\ell
\]
and is congruent to $\psi(\Frob_\ell)$ modulo the maximal ideal of $\T^{pN_Q}(M(k,Q,\Oh))$.  
By Hensel's lemma, in both $\T^{pN_Q}(M(k,Q,\Oh))$ and $\T^{pN_Q/\ell}(M(k,Q,\Oh))$, there is a unique solution to this equation congruent to $\psi(\Frob_\ell)$, which must therefore be $U_\ell$.  
In particular, $U_\ell \in \T^{pN_Q}(M(k,Q,\Oh))$.

    In the second case, fix a lift $\sigma \in G_{\Q_\ell}$ of $\Frob_\ell$. Since $I_\ell$ is normal in $G_{\Q_\ell}$, a straightforward calculation shows that $\rhobar(\sigma)$ is diagonal and its determinant $\chi \epsilon^{k-1}$ is ramified. We then have
\[
\rhobar|_{G_{\Q_\ell}} \xrightarrow{\sim} \psi \oplus \chi \epsilon^{k-1}/\psi
\]
for some unramified character $\psi$.  
If $\rhobar(\sigma)$ has repeated eigenvalues, we may replace $\sigma$ by $\sigma \tau$ for some $\tau \in I_\ell$; in this way, we may and do assume that $\rhobar(\sigma)$ has distinct eigenvalues. As a result, the Weil--Deligne representation of $\rho_f|_{G_{\Q_\ell}}$ is automatically Frobenius-semisimple.  
By Grothendieck's monodromy theorem and the classification of such Weil--Deligne representations, for $\rho_f$ to satisfy $\dim_L \rho_f^{I_\ell} = 1$ and $\rhobar_f|_{G_{\Q_\ell}}\xrightarrow{\sim} \rhobar|_{G_{\Q_\ell}}$, it must take the form
\[
\rho_f|_{G_{\Q_\ell}} \xrightarrow{\sim} \eta_f \oplus \chi_f \chi_\cyc^{k-1}/\eta_f
\]
for some unramified character $\eta_f$. By the local Langlands correspondence and \cite[Lemma 4.2.4]{CDT1999Barsotti}, $\pi_{f,\ell}$ is a principal series $I(\chi_f \chi_\cyc^{k-1}/\eta_f, \eta_f)$, and $\eta_f(\Frob_\ell) = a_\ell(f) \in \Oh^\times$.
For every lift $\sigma' \in G_{\Q_\ell}$ of $\Frob_\ell$, we have
\[
a_\ell(f) + \det(\rho_f(\sigma')) a_\ell(f)^{-1} = \tr(\rho_f(\sigma')),
\]
which gives the equation for $U_\ell$ in $\T^{pN_Q}(M(k,Q,\Oh))$:
\[
U_\ell + \det(\rho_Q^{\rm mod}(\sigma')) U_\ell^{-1} = \tr(\rho_Q^{\rm mod}(\sigma')).
\]
Since $\chi$ is ramified at $\ell$, we can choose $\tau \in I_\ell$ with $\chi(\tau) \neq 1$. Then
\[
\det(\rho_Q^{\rm mod}(\sigma)) - \det(\rho_Q^{\rm mod}(\sigma \tau)) \equiv \chi(\sigma) \ell^{k-1} (1 - \chi(\tau)) \not\equiv 0 \pmod{\mathfrak{m}_Q^{pN_Q}},
\]
so the system
\[
\begin{cases}
X + \det(\rho_Q^{\rm mod}(\sigma)) Y = \tr(\rho_Q^{\rm mod}(\sigma))\\[1ex]
X + \det(\rho_Q^{\rm mod}(\sigma \tau)) Y = \tr(\rho_Q^{\rm mod}(\sigma \tau))
\end{cases}
\]
has a unique solution in $\T^{pN_Q}(M(k,Q,\Oh))$, which proves that $U_\ell \in \T^{pN_Q}(M(k,Q,\Oh))$.
    
    We are left with the situation where $\dim_\F \rhobar^{I_\ell} = 0$.  
In this case, the Artin conductor at $\ell$ is at least $2$.  
By \cite[Lemma 4.2.2]{CDT1999Barsotti}, we then have $U_{\ell}(f) = 0$ for all $f \in M(k,Q,\Oh)$, and hence 
\[
U_\ell = 0 \in \T^{pN_Q}(M(k,Q,\Oh)).
\]
\end{proof}

\begin{prop}
    The Hecke operator $U_q$ is in $\T^{pN_Q}(M(k,Q,\Oh))$ for every $q\in Q\sqcup\{q_0\}.$
\end{prop}

\begin{proof}
    Let $f\in B(k,\Oh).$
    If $f$ is new at $q$, then the local automorphic representation $\pi_{f,q}$ attached to $f$ has conductor $q$.
    It is either a Steinberg representation $\St(\eta_f)$ or a principal series 
\[
I(\chi_f\chi_\cyc^{k-1}/\eta_f,\eta_f),
\]
where $\eta_f$ is an unramified character of $G_{\Q_q}$ sending $\Frob_q$ to $a_q(f)$. In the first case, the local Galois representation $\rho_f|_{G_{\Q_q}}$ is an unramified twist of an extension of $1$ by $\chi_\cyc$ by local--global compatibility.  
Hence, modulo $\varpi$, the ratio of the eigenvalues of $\rhobar(\Frob_q)$ lies in $\{\epsilon(q),\epsilon^{-1}(q)\}$, contradicting our assumption on primes $q \in Q \sqcup\{q_0\}$.  
Thus $\pi_{f,q}$ must be a principal series. By an analysis identical to that in the proof of the previous proposition, if $\sigma \in G_{\Q_q}$ is a lift of $\Frob_q$, then
\[
a_q(f)^2 - \tr(\rho_f(\sigma))\, a_q(f) + \det(\rho_f(\sigma)) = 0,
\]
where $\rho_f$ is the Galois representation attached to $f$.

If $f$ is old at $q,$ then it comes from a newform $g$ with level $N_g$ such that $q\nmid N_g$. In this case, the eigenvalue of $U_q$ for $f$ is $a_q(f)$ and it is the unique root of 
$$
X^2-\tr(\rho_f(\Frob_q))X+\det(\rho_f(\Frob_q))=0
$$
that is congruent to $\alpha_q$ by the proof of Proposition \ref{prop_eigenform_new_old_at_q} and Remark \ref{rmk_conductor_equals_new_level}.

    The two cases combined implies that $U_q$ is a root of 
$$
X^2 - \tr(\rho_Q^{\rm mod}(\sigma))X + \det(\rho_Q^{\rm mod}(\sigma))=0
$$
for every $\sigma\in G_{\Q_q}$ that is a lift of $\Frob_q.$ Modulo the maximal ideal $\mathfrak{m}_Q^{\{pN_Q\}}\T^{pN_Q}(M(k,Q,\Oh)),$ the equation reduces to 
$$
X^2 - \tr(\rhobar(\Frob_q))X + \chi(q)q^{k-1}
$$
which has distinct roots in $\F$ by construction. Again by Hensel's lemma, we conclude that 
$U_q\in \T^{pN_Q}(M(k,Q,\Oh)).$
\end{proof}

\begin{proof}[Proof of Lemma \ref{lem_U_ell_in_anemic}]
    The statement for $\T^{p}(M(k,Q,\Oh))$ follows from the previous two propositions. For $* = \ord$ or $\no$, since
    $$M(k,Q,\Oh)^* \hookrightarrow M(k,Q,\Oh),$$
    we have
    $$\T^{pN_Q}(M(k,Q,\Oh)) \twoheadrightarrow \T^{pN_Q}(M(k,Q,\Oh)^*).$$ Thus, if $\T^{pN_Q}(M(k,Q,\Oh))$ contains all the Hecke operators dividing the level, then so does $\T^{pN_Q}(M(k,Q,\Oh)^*)$.
\end{proof}

\subsection{Serre Weights}\label{subsection_serre_weights}
In this subsection, we recall the definitions of Serre weights and record some facts about the weights of mod-$p$ modular forms that will be used later. Throughout, we write $M(k,Q,\Oh)^*$ for either $M(k,Q,\Oh)$ or $M(k,Q,\Oh)^{\no}$.

Let \( j(k) \) denote the largest integer such that \( j(k) \equiv k \pmod{p-1} \) and \( j(k)p \le k \). Recall from Corollary \ref{cor_anemic_equal_full_criterion} that \( C(M(k, Q, \Oh)^*) \) denotes the cokernel of the natural map  
\[
\T^{p}(M(k, Q, \Oh)^*) \hookrightarrow \T(M(k, Q, \Oh)^*).
\]

\begin{prop} \label{prop_image_and_cokernel_of Hecke_algebras_finite_free}
    Suppose that \( k \ge 2 \), and if $k(\rbar)=1,$ assume that $k\ge p^2.$
    Then both the image $\T^p(M(k,Q,\F)^*)$ and the cokernel $C(M(k, Q, \Oh)^*) \otimes_{\Oh} \F$ of the map
    \[
        \T^p(M(k, Q, \Oh)^*) \otimes_{\Oh} \F \to \T(M(k, Q, \Oh)^*) \otimes_{\Oh} \F
    \]
    are finite free \( \F[\Delta_Q] \)-modules. Moreover, the \( \F[\Delta_Q] \)-rank of \( C(M(k, Q, \Oh)^*) \otimes_{\Oh} \F \) is equal to \( \dim_\F M(j(k), \varnothing, \F)^* \).
\end{prop}

\begin{proof}
    By Corollary~\ref{cor_anemic_equal_full_criterion}, the cokernel 
\( C(M(k, Q, \Oh)^*) \otimes_{\Oh} \F \) is dual to 
\[
(M(k, Q, \Oh)^*\otimes_{\Oh}\F)\cap \Image(V_p).
\]
When \(k\ge 2\), we may identify this intersection with 
\(M(k, Q, \F)^*\cap \Image(V_p)\). It follows that the dual of 
\( C(M(k, Q, \Oh)^*)\otimes_{\Oh}\F \) is isomorphic to 
\( M(j(k), Q, \F)^* \). Since \(V_p\) is injective by Proposition~\ref{prop_Vp_filtration} and respects the ordinary/non-ordinary decomposition by \ref{prop_Vp_preserves_ordinary_nonordinary forms}, 
it induces an isomorphism
\[
V_p:\ M(j(k), Q, \F)^*\xrightarrow{\sim} M(pj(k), Q, \F)^*\cap \Image(V_p).
\] This is $\T^{p}$-equivariant because $V_p$ commutes with Hecke action away from $p.$ 

Since \( \F[\Delta_Q] \) is self-dual as an \( \F[\Delta_Q]\)-module, the 
\( \F \)-linear dual of any finite free \( \F[\Delta_Q] \)-module is again 
finite free of the same rank. By Theorem~\ref{thm_ribet_duality_full} and 
the preceding lemma, the module \( M(j(k), Q, \F)^* \) is finite free over
\( \F[\Delta_Q] \) of rank \( \dim_\F M(j(k), \varnothing, \F)^* \) whenever 
\( j(k)\ge 2 \). This always holds when \( k(\rbar)\ge 2 \). If 
\( k(\rbar)=1 \), then \( j(k)\ge 2 \) holds exactly when \( j(k)\ge p \),
i.e.\ when \( k\ge p^2 \). Hence in all such cases, the cokernel
\( C(M(k, Q, \Oh)^*)\otimes_{\Oh}\F \), being dual to 
\(M(j(k), Q, \F)^*\), is finite free over \( \F[\Delta_Q] \) of the same rank.

Finally, the exact sequence
\[
0\to \T^p(M(k, Q, \F)^*)\to 
\T(M(k, Q, \Oh)^*)\otimes_{\Oh}\F \to 
C(M(k, Q, \Oh)^*)\otimes_{\Oh}\F \to 0
\]
splits, because \( C(M(k, Q, \Oh)^*)\otimes_{\Oh}\F \) is projective over the
local ring \( \F[\Delta_Q] \). Thus 
\( \T^p(M(k, Q, \F))^* \) is a direct summand of the finite free 
\( \F[\Delta_Q] \)-module 
\( \T(M(k, Q, \Oh)^*)\otimes_{\Oh}\F \), and is therefore also projective,
hence finite free. 
\end{proof}

To determine when $M(j(k),\varnothing,\F)^* = 0$, we recall the Serre weights $k(\rbar)$ and introduce the non-ordinary Serre weights $k(\rbar)^{\no}$.

\begin{definition}[Serre, Edixhoven]\label{def_Serre_weight}
    The Serre weight $k(\rbar)$ is defined as follows:
    \begin{enumerate}
    \item If $\rbar$ is absolutely irreducible, then $\rbar$ is tamely ramified and the restriction $\rbar|_{I_p}$ is a direct sum of $\epsilon_2^{a+pb}$ and $\epsilon_2^{pa+b}$ where $0\le a<b\le p-1$ are integers. Then $$k(\rbar) = 1 + pa + b.$$
    \item If $\rbar$ is reducible, there are two sub-cases:
    \begin{enumerate}
        \item If $\rbar$ is tamely ramified, then $\rbar|_{I_p}$ is a direct sum of $\epsilon^{a}$ and $\epsilon^b$ where $0\le a\le b\le p-2.$ In this case, we set $$k(\rbar) = 1+pa+b.$$ 
        \item If $\rbar$ is wildly ramified, then $$\rbar|_{I_p}\sim \begin{pmatrix}
        \epsilon^{\beta} & * \\
        0 & \epsilon^{\alpha}
        \end{pmatrix}$$ where $*$ is a nontrivial extension and  $0\le \alpha\le p-2$ and $1\le \beta\le p-1$ are integers. We set $a = \min\{\alpha,\beta\}$ and $b = \max\{\alpha,\beta\}.$ If $\epsilon^{\beta-\alpha} = \epsilon$ and $*\chi^{-\alpha}$ is tres ramifi\'ee, then $$k(\rbar) = 1 + pa + b + p-1 = p + pa + b.$$ Otherwise, we set $$k(\rbar) = 1+ pa + b.$$ 
    \end{enumerate}
\end{enumerate}
\end{definition}

\begin{definition} \label{def_nonordinary_serre_weight}
    Given $\rbar:G_{\Q_p}\to \GL_2(\F),$ we define the non-ordinary Serre weight $k(\rbar)^{\no}$ of $\rbar$ as follows.
    \begin{enumerate}
        \item If $\rbar|_{I_{p}}\not\sim \begin{pmatrix}
        \epsilon^{b} & *\\
        0 & 1
        \end{pmatrix}$ for any integer $0\le b\le p-2,$ we set $k(\rbar)^{\no} = k(\rbar).$
        \item If $\rbar|_{I_{p}}\sim \begin{pmatrix}
        \epsilon^{b} & *\\
        0 & 1
        \end{pmatrix}$ for some integer $0\le b\le p-2$ and $\rbar$ is ramified, we set $k(\rbar)^{\no} = pk(\rbar).$
        \item If $\rbar$ is unramified, we set $k(\rbar)^{\no} = p^2.$
    \end{enumerate}
\end{definition}

We now explain why we call these weights non-ordinary Serre weights, from both global and local perspectives, in our setup.

\begin{prop} \label{prop_nonordinary_serre_weight_global_meaning}
    Under Assumption \ref{assumption_global_rhobar}, the non-ordinary Serre weight $k(\rbar)^{\no}$ is the smallest integer greater than one such that 
    $$M(k(\rbar)^{\no},\varnothing,\F)^\no \neq 0.$$
\end{prop}
  
\begin{proof}
    Let $k'$ be the smallest integer greater than one such that $M(k',\varnothing,\F)^{\no}\neq 0.$ 
    We need to show that such $k'$ exists and that $k' = k(\rbar)^\no$ as defined above. 

    First, assume $k(\rbar)\ge 2$, i.e., $\rbar$ is ramified. By definition, $M(k,\varnothing,\F)$ is a generalized eigenspace for $\T^{pN_\varnothing}$ with eigenvalues determined by $\rhobar.$ For $k\ge 2$, Lemma \ref{lem_U_ell_in_anemic} and the base-change theorem imply that it is actually a generalized eigenspace for $\T^p$ with eigenvalues determined by $\rhobar.$ Since all the Hecke operators commute, the ordinary and non-ordinary parts of $M(k,\varnothing,\F)$ are also generalized eigenspaces with the same eigenvalues. Therefore, there exists a nonzero eigenform
    $$f_0 = \sum_{n\ge 1} a_n(f_0)q^n \in M(k(\rbar),\varnothing,\F)$$ 
    for the full Hecke algebra $\T$ with eigenvalues $\lambda_{f_0}:\T\to \F.$
    Because $a_n(f_0) = a_1(T_n f_0) = \lambda(T_n) a_1(f_0)$, we have $a_1(f_0)\neq 0.$ Rescaling, we may assume $a_1(f_0)=1$, so $a_n(f_0)=\lambda(T_n)$ for all $n.$ If $k'$ exists, applying the same reasoning to $M(k',\varnothing,\F)^{\no}\neq 0$ gives a $\T$-eigenform
    $$g_0 = \sum_{n\ge 1} a_n(g_0)q^n$$
    with $a_n(g_0)$ the Hecke eigenvalues of $T_n$. Since the eigenvalues for $\T^p$ are determined by $\rhobar$, we have
    $$a_n(f_0)=a_n(g_0)\quad \text{for } \gcd(n,p)=1,$$
    and $a_p(g_0)=0$ because $g_0$ is non-ordinary.

    When $\rbar|_{I_p}\not\sim \begin{pmatrix} \epsilon^b & *\\ 0 & 1 \end{pmatrix}$ for any $0\le b\le p-2,$ we claim that $a_p(f_0) = 0.$
 Otherwise, by Deligne--Serre's lifting lemma, $f_0$ lifts to a $\T$-eigenform $\widetilde{f}_0$ with $a_p(\widetilde{f}_0)\in \Oh^\times.$ Then Theorem \ref{thm_Scholl_rhof_crystalline_at_p} and Proposition \ref{prop_crystalline_characteristic_polynomial} imply that the associated Galois representation is an extension of an unramified character by an unramified twist of $\chi_\cyc^{k(\rbar)-1}$. Its reduction modulo $p$ would have inertia restriction of the form $\rbar|_{I_p}\sim \begin{pmatrix} \epsilon^b & *\\ 0 & 1 \end{pmatrix}$, a contradiction.  
    Hence $f_0$ is non-ordinary of weight $k(\rbar)$, so $k'=k(\rbar)=k(\rbar)^\no.$

    When $\rbar|_{I_p} \sim \begin{pmatrix} \epsilon^b & *\\ 0 & 1 \end{pmatrix}$ for some integer $0\le b\le p-2,$ one deduce from Definition \ref{def_Serre_weight} that $2\le k(\rbar)\le p+1.$ The classification of residual representations of crystalline representations in the Fontaine--Laffaille range (e.g., \cite[Theorem 2.5, 2.6]{edixhoven1992weight}) implies that $a_p(f_0)$ is a unit in $\F.$ Consider
    $$A^{k(\rbar)} f_0 - a_p(f_0) (f_0 V_p).$$
    By Proposition \ref{prop_Vp_filtration}, this has weight and filtration $pk(\rbar)=k(\rbar)^\no$. Since $A$ and $V_p$ commute with all Hecke operators away from $p$, it is a $\T^p$-eigenform with the same Hecke eigenvalues as $f_0.$ At $p$, we compute
    $$(A^{k(\rbar)} f_0 - a_p(f_0) f_0 V_p) U_p = 0,$$
    so it is non-ordinary of weight $k(\rbar)^\no.$ Thus $k'$ exists and $k'\le pk(\rbar).$ On the other hand, we have
    $$A^{(k(\rbar)^\no - k(\rbar))/(p-1)} f_0 - g_0 = \sum_{p|n} (a_n(f_0) - a_n(g_0)) q^n \neq 0.$$
    This lies in $\ker\theta = \Image V_p$, so by Proposition \ref{prop_Vp_filtration}, $g_0$ has filtration at least $pk(\rbar)$. Hence $k'\ge pk(\rbar)$, giving $k'=pk(\rbar)=k(\rbar)^\no$ as desired.

    At last, we treat the case where $\rbar$ is unramified, i.e., $k(\rbar)=1.$ In particular, $\rbar$ is reducible. Then $U_p$ is invertible on $M(p,\varnothing,\F)$ by the classification  of residual Galois representations in the Fontaine--Laffaille range. A similar argument produces a non-ordinary $\T$-eigenform of weight $p^2$, showing $k'$ exists and $k'\le p^2.$ On the other hand, the standard exactness argument in \cite[\S 6.1]{BergdallPollack2019slopes} and the fact that $\Sym^{p-2} \F^2$ is the only Serre weight (in the sense of \cite{buzzard2010serre}) of $\rbar$ imply
    $$\dim_\F M(p,\varnothing,\F) = \dim_\F M(2p-1,\varnothing,\F) = \dots = \dim_\F M(p^2 - p + 1, \varnothing, \F),$$
    so all eigenforms up to weight $p^2-p+1$ come from weight $p$ and are therefore ordinary. Hence $k'\ge p^2$, giving $k'=p^2 = k(\rbar)^\no.$
\end{proof}

\begin{prop} \label{prop_nonordinary_Serre_weight_local_meaning}
    Suppose that a local representation 
    $\rbar:G_{\Q_p}\to \GL_2(\F)$ satisfies Assumption \ref{assumption_semisimplification}. 
    Then the non-ordinary Serre weight $k(\rbar)^\no$ is the smallest integer greater than one 
    such that $\rbar$ has an irreducible crystalline lift of Hodge--Tate weights 
    $\{0,k(\rbar)^{\no}-1\}.$
\end{prop}

\begin{proof}
    Let $k''$ be the smallest integer bigger than one such that $\rbar$ has an irreducible crystalline lift of Hodge--Tate weights $\{0,k''-1\}.$ We need to establish the existence of such a $k''$ and prove that $k''=k(\rbar)^\no$ defined as above. 

First suppose that $\rbar$ is semisimple. Then by Remark \ref{rmk_semisimple_reps_satisfy_assumption_global}, Assumption \ref{assumption_global_rhobar} is satisfied. It follows from the proposition above that $k''$ exists and $k''\le k(\rbar)^{\no}.$ On the other hand, as explained by the proof of \cite[Proposition 3.7]{calegari_2011_even_fontaine_mazur}, modular forms are supported on every component of $R_\infty^{\psi_k}(k)$ (defined in \S \ref{section_patching}), if $k''\ge 2$ is strictly smaller than $k(\rbar)^{\no},$ we will find a non-ordinary modular form of weight $k''$ and level coprime to $p$ such that its Galois representation lifts $\rhobar.$ In particular, the local representation associated to this form is an irreducible crystalline lift of $\rbar$ of Hodge--Tate weights $\{0,k''-1\},$ contradicting Proposition \ref{prop_nonordinary_serre_weight_global_meaning}. Thus $k''=k(\rbar)^{\no}.$

Now suppose that $\rbar$ satisfies Assumption \ref{assumption_semisimplification}. When $k\ge 2,$ by Lemma \ref{lem_paskunas_versal_universal_crystalline}, $\rbar$ has an irreducible crystalline lift of Hodge--Tate weights $\{0,k-1\}$ if and only if $\rbarss$ does. So it suffices to prove $k(\rbar)^{\no} = k(\rbarss)^{\no}.$ When $\rbar$ is unramified, then $\rbarss$ is unramified as well. By definition, we have $k(\rbar)^{\no} = k(\rbarss)^\no=p^2.$ So we may assume $\rbar$ is ramified. There are then four cases:

\begin{enumerate}
    \item If $\chi_1$ and $\chi_2$ are both ramified, then 
    $$k(\rbar)^{\no} = k(\rbar),\quad k(\rbarss)^{\no} = k(\rbarss)\quad{\rm and}\quad k(\rbar) = k(\rbarss).$$

    \item If $\chi_1$ is ramified and $\chi_2$ is unramified, then 
    $$k(\rbar)^{\no}=pk(\rbar),\quad k(\rbarss)^{\no}=pk(\rbarss)\quad{\rm and\quad} k(\rbar) = k(\rbarss).$$

    \item If $\chi_1$ is unramified and $\chi_2$ is ramified, then 
    $$k(\rbar)^{\no} = k(\rbar),\quad k(\rbarss)^{\no} = pk(\rbarss)\quad {\rm and}\quad k(\rbar) = pk(\rbarss).$$

    \item If $\chi_1$ and $\chi_2$ are both unramified, then 
    $$k(\rbar)^{\no} = pk(\rbar)=p^2\quad{\rm and\quad} k(\rbarss)^{\no}=p^2.$$
\end{enumerate}
In all cases, we have $k(\rbar)^{\no}=k(\rbarss)^{\no}.$

\end{proof}

\begin{remark} \label{rmk_anemic_full_iso_weight_range}
    Suppose that $k(\rbar)\ge 2.$ By Serre's conjecture, we have $M(j(k),\varnothing,\F) = 0$ if and only if $j(k)<k(\rbar),$ which is equivalent to $k<pk(\rbar).$ By Proposition \ref{prop_nonordinary_serre_weight_global_meaning}, we have $M(j(k),\varnothing,\F)^\no = 0$ if and only if $j(k)<k(\rbar)^\no,$ which is equivalent to $k<pk(\rbar)^\no.$ 
\end{remark}

\section{Patching} \label{section_patching}
In this section, we carry out the main construction outlined in \S\ref{subsubsection_main_construction}. We first establish our notation, then recall the ultrapatching functor, and finally apply it to the coherent cohomology of modular curves.

\subsection{Setup} \label{subsection_setup_for_patching}
Throughout the section, we assume that $\rbar$ satisfies Assumption \ref{assumption_global_rhobar} and $k\ge 2.$ For a prime $\ell\neq p$ of $\Q,$ we let $\rhobar_\ell$ be the restriction  of $\rhobar$ to $G_{\Q_\ell}.$ Let $T(\rhobar)$ be the set of vexing primes satisfying conditions in \ref{rmk_vexing_primes}. 
For a finite set $Q$ of primes as in \S \ref{subsection_serre_weights}, let $R^{\Box}_{Q}$ represent the deformation functor of 
$\rhobar:G_{\Q}\to \GL_2(\F)$ that to each 
$R\in \mathcal{C}_{\Oh}$ assigns
\[
\left\{
\left(
\rho_R,\{\gamma_v\}_{v\in T(\rhobar)\sqcup\{p\}}
\right):
\begin{aligned}
&\rho_R|_{G_{\Q_v}} \text{ is minimally ramified for all } 
   v\notin T(\rhobar)\sqcup\{p\}\sqcup Q,\\
&\gamma_v\in \ker\!\left(\GL_2(R)\to \GL_2(\F)\right)
\end{aligned}
\right\}/\!\sim
\]
where $\sim$ is given by
\[
(\rho_R,\{\gamma_v\}_{v\in T(\rhobar)\sqcup\{p\}})
\sim
\left(
\delta\rho_R\delta^{-1},
\{\delta\,\gamma_v\}_{v\in T(\rhobar)\sqcup\{p\}}
\right)
\]
for each 
$\delta\in \ker\!\left(\GL_2(R)\surj \GL_2(\F)\right).$
Let $R^{\univ}_Q$ denote the deformation ring representing the same local deformation conditions as $R^{\Box}_Q$, except with the local framings 
$\{\gamma_v\}_{v\in T(\rhobar)\sqcup\{p\}}$ omitted.  It follows that
$R^{\univ}_Q$ is the subring of $R^{\Box}_Q$ corresponding to the deformation functor obtained by forgetting these framings.
Set
\[
\mathcal{J}
:= 
\Oh\llbracket X_{v,i,j}\rrbracket_{\,v\in T(\rhobar)\sqcup\{p\},\, i,j=1,2}/(X_{p,1,1}).
\]
We fix once and for all an identification
\[
R^{\univ}_Q \widehat{\otimes}_{\Oh} \mathcal{J}
\xrightarrow{\ \sim\ }
R^{\Box}_Q.
\]

For each integer $k$, let $\psi_k := \chi\chi_{\cyc}^{\,k-1}$.  
Write $R^{\Box,\psi_k}_Q$ and $R^{\univ,\psi_k}_Q$ for the quotients of
$R^{\Box}_Q$ and $R^{\univ}_Q$ respectively that parametrize liftings with
fixed determinant $\psi_k$.  
The above identification then induces an isomorphism
\[
R^{\univ,\psi_k}_Q \widehat{\otimes}_{\Oh} \mathcal{J}
\xrightarrow{\ \sim\ }
R^{\Box,\psi_k}_Q.
\]

\begin{lem}[{\cite[Proposition~3.2.5]{kisin2009moduli}}]\label{lem_Taylor_Wiles_primes}
    Under the Taylor--Wiles condition (in Assumption \ref{assumption_global_rhobar}), for every integer $n\ge 1,$ there exists a finite set $Q_n$ of rational primes such that all of the following hold \begin{enumerate}
        \item $Q_n$ is disjoint from the divisors of $N_\varnothing$
        \item $\#Q_n=\dim H^1(G_{\Q, T(\rhobar)\sqcup \{p\}},\ad^0\rhobar(1));$ 
        \item if $q\in Q_n$, then $q\equiv 1\pmod {p^n};$
        \item $\rhobar(\Frob_q)$ has distinct eigenvalues $\alpha_q$ and $\beta_q$; and
        \item $\ker\left(H^1(G_{\Q,T(\rhobar)\sqcup \{p\}},\ad^0\rhobar(1))\to \prod_{q\in Q_n}H^1(G_{\Q_q}/I_q,\ad^0\rhobar(1))\right)=0.$ 
    \end{enumerate}
\end{lem}

Set $r = \#Q_n$ and $
g = \# T(\rhobar) - 1+r.$
Write $R_\rbar^{\Box,\psi_k}(k)^*$ for either one of the three rings: $R_\rbar^{\Box,\psi_k}(k),R_\rbar^{\Box,\psi_k}(k)^{\ord}$ or $R_\rbar^{\Box,\psi_k}(k)^\no$ defined in \S \ref{subsubsection_crystalline_representations}. Define $$R_{T(\rhobar)\sqcup\{p\}}^{\Box,\psi_k} := \widehat{\otimes}_{v\in T(\rhobar)\sqcup\{p\},\Oh}R_{\rhobar_v}^{\Box,\psi_k}\quad{\rm and}\quad R_{T(\rhobar)\sqcup\{p\}}^{\Box,\psi_k}(k)^* := R_\rbar^{\Box,\psi}(k)^*\widehat{\otimes}_{\Oh}\left(\widehat{\otimes}_{v\in T(\rhobar),\Oh}R_{\rhobar_v}^{\Box,\psi}\right).$$ 
The last condition in Lemma \ref{lem_Taylor_Wiles_primes} implies that $R^{\Box,\psi_k}_{Q_n}$ can be topologically generated by $g$ elements over $R^{\Box,\psi_k}_{T(\rhobar)\sqcup\{p\}}$. Thus for each $n,$ we fix a surjection $$R_\infty^{\psi_k} := R_{T(\rhobar)\sqcup\{p\}}^{\Box,\psi_k}\llbracket x_1,\ldots,x_g\rrbracket \surj R_{Q_n}^{\Box,\psi_k}.$$ Let $R_\infty^{\psi_k}(k)^* := R_{T(\rhobar)\sqcup\{p\}}^{\Box,\psi_k}(k)^*\llbracket x_1,\ldots,x_g\rrbracket$ be the corresponding quotient of $R_\infty^{\psi_k}.$

\begin{prop}\label{prop_R_infty_special_fiber}
    The ring $R_\infty^{\psi_k}\otimes_\Oh\F$ is a geometrically irreducible local complete intersection of Krull dimension $$4\#T(\rhobar) + r + 5.$$ In particular, all zero-divisors in $R_\infty^{\psi_k}\otimes_\Oh\F$ are contained in its nilradical. 
\end{prop}
\begin{proof}
By definition, $R_\infty^{\psi_k}\otimes_\Oh \F$ is the completed tensor product over $\F$ of the special fibers of the local deformation rings 
$R_{\rhobar_v}^{\Box,\psi_k}\otimes_\Oh \F$ for $v \in T(\rhobar)\sqcup\{p\}$, together with a power series ring in $g$ variables. By Lemma~\ref{lem_Taylor} and Proposition~\ref{prop_completed_tensor_product_in_char_p}, 
$R_\infty^{\psi_k}\otimes_\Oh \F$ is a geometrically irreducible local complete intersection
provided each local deformation ring $R_{\rhobar_v}^{\Box,\psi_k}\otimes_\Oh \F$ has the same property. 
This follows from Remark~\ref{rmk_vexing_prime_local_deformation_ring}, Theorem~\ref{thm_unrestricted_deformation_ring_at_p}, and its remark.

The zero-divisors of a Noetherian ring are contained in the union of its associated primes.  
Since $R_\infty^{\psi_k}\otimes_\Oh \F$ is a local complete intersection, it is Cohen--Macaulay, 
so all its associated primes are minimal by~\cite[Corollary~18.10]{eisenbud2013commutative}.  
As $R_\infty^{\psi_k}\otimes_\Oh \F$ is irreducible, it has a unique minimal prime ideal, namely its nilradical.

Finally, the Krull dimension is computed by summing the dimensions of the local deformation rings and adding the $g$ auxiliary variables:
\[
\dim R_\infty^{\psi_k}\otimes_\Oh\F 
= \sum_{v\in T(\rhobar)\sqcup\{p\}} \dim (R_{\rhobar_v}^{\Box,\psi_k}\otimes_\Oh \F) + g
= 3\#T(\rhobar) + 6 + (\#T(\rhobar) - 1 + r)
= 4\#T(\rhobar) + r + 5.
\]
\end{proof}

On the automorphic side, recall the modular Galois representation from \S \ref{subsection_hecke_modules_and_algebras}:
\[
\rho^{\rm mod}_Q : G_{\Q,S_Q} \to \GL_2\big(\T^{pN_Q}(k,Q,\Oh)\big),
\]
which is minimally ramified at primes $v \notin T(\rhobar)\sqcup\{p\}\sqcup Q$.  
By construction, the character 
$
\psi_k \det(\rho_Q^{\rm mod})^{-1}
$
takes values in $1 + \varpi\,\T^{pN_Q}(M(k,Q,\Oh))$ and is unramified outside $Q$.  
There is a unique character of $G_{\Q,Q}$, valued in $1 + \varpi\,\T^{pN_Q}(M(k,Q,\Oh))$, whose square is $\psi_k \det(\rho_Q^{\rm mod})^{-1}$.  
We denote this square root by 
\[
\big(\psi_k \det(\rho_Q^{\rm mod})^{-1}\big)^{1/2}.
\]  
Twisting $\rho_Q^{\rm mod}$ by this character, we obtain
$
\rho_Q^{\rm mod} \otimes \big(\psi_k \det(\rho_Q^{\rm mod})^{-1}\big)^{1/2},
$
a Galois representation of $G_{\Q,S_Q}$ with determinant $\psi_k$, which is minimally ramified at primes $v \notin T(\rhobar)\sqcup\{p\}\sqcup Q$.  
Hence, this gives a natural surjection
\begin{equation}\label{eqn_R_global_to_T_global}
R_Q^{\univ,\psi_k} \twoheadrightarrow \T^{pN_Q}(M(k,Q,\Oh)).
\end{equation}

\begin{prop}[Darmon--Diamond--Taylor]
    The action of $\Delta_Q$ on $\T^{pN_Q}(M(k,Q,\Oh))$ via the diamond operators $\langle q\rangle$ for $q\in Q$ 
    coincides with the action induced by
    \[
        \Delta_Q \longrightarrow \big(R_Q^{\Box,\psi_k}\big)^\times
        \twoheadrightarrow \big(R_Q^{\univ,\psi_k}\big)^\times
        \twoheadrightarrow \T^{pN_Q}(M(k,Q,\Oh))^\times
    \]
    via Galois deformation theory.
\end{prop}

\begin{proof}
    The action of $\Delta_Q$ on $\big(R_Q^{\Box,\psi_k}\big)^\times$ factors through the local universal lifting rings at the Taylor--Wiles primes $q \in Q$.  
    This is explained in \cite[Lemma~2.44]{DDT1995fermat}, and the coincidence of these structures is proved in \cite[Proposition~4.10]{DDT1995fermat}.
\end{proof}

Let 
\[
S_\infty := \mathcal{J} \llbracket y_1,\ldots,y_r \rrbracket.
\] 
We denote by $\mathfrak{m}_{S_\infty}$ its maximal ideal. 
Since $S_\infty$ is formally smooth, we can choose a ring homomorphism 
$
S_\infty \longrightarrow R_\infty^{\psi_k}
$
lifting the map 
\[
S_\infty \twoheadrightarrow \mathcal{J}[\Delta_{Q_n}] \longrightarrow R_{Q_n}^{\Box,\psi_k}.
\] 

For $R=\Oh$ or $\F,$ we write $M(k,Q,R)^{*}$ for one of the three modules $M(k,Q,R),M(k,Q,R)^\ord$ or $M(k,Q,R)^\no$ defined in \S \ref{subsection_hecke_modules_and_algebras}. The choice of $*$ is understood to be fixed throughout each statement. We set $$\T^{pN_Q}(M(k,Q,\Oh)^*)^{\Box} := \T^{pN_Q}(M(k,Q,\Oh)^*)\otimes_{R_Q^{\univ,\psi_k}}R^{\Box,\psi_k}_Q.$$
The full Hecke algebra $\T(M(k,Q,\Oh)^*)$ is an $R_Q^{\univ,\psi_k}$-algebra whose action factors through $\T^{pN_Q}(M(k,Q,\Oh)^*)$. We define $$\T(M(k,Q,\Oh)^*)^{\Box}:=\T(M(k,Q,\Oh)^*)\otimes_{R^{\univ,\psi_k}_Q}R^{\Box,\psi_k}_Q$$ and $$C(M(k,Q,\Oh)^*)^{\Box} := C(M(k,Q,\Oh)^*)\otimes_{R_Q^{\univ,\psi_k}}R^{\Box,\psi_k}_Q.$$ By Lemma \ref{lem_U_ell_in_anemic}, we have a short exact sequence of $S_\infty$-modules \begin{equation} \label{eqn_anemic_full_cokernel_short_exact_sequence}
    0 \to \T^{pN_Q}(M(k,Q,\Oh)^*)^{\Box} \to \T(M(k,Q,\Oh)^*)^{\Box} \to C(M(k,Q,\Oh)^*)^{\Box} \to 0
\end{equation} because $R_Q^{\Box,\psi_k}$, as a formal power series ring over $R^{\univ,\psi_k}_Q,$ is flat. Now we have an $S_\infty$-equivariant diagram
\begin{equation}\label{eqn_R_surj_T}
R^{\psi_k}_\infty\surj \T^{pN_{Q_n}}(M(k,Q_n,\Oh)^*)^{\Box} \curvearrowright \T(M(k,Q_n,\Oh)^*)^{\Box},
\end{equation} which factors through $R_\infty^{\psi_k}(k)$ by local-global compatibility. Moreover, the action of $R_\infty^{\psi_k}(k)$ can be extended to $R_\infty^{\psi_k}(k)[\alpha_p]$ which we now explain. 

By \cite[Theorem~4.1]{Caraiani_Emerton_Gee_Geraghty_Paskunas_Shin_2016_p_adic_local_Langlands} 
and the proof of \cite[Proposition~2.9]{Caraiani_Emerton_Gee_Geraghty_PaskunasShin_2018_GL2Qp}, 
the element $\alpha_p$ acts on $\T(M(k,Q_n,\Oh)^*)^{\Box}$ via $T_p.$  
(Note that our crystalline functor is the dual of the $D_{\cris}$ used in \cite{Caraiani_Emerton_Gee_Geraghty_PaskunasShin_2018_GL2Qp}.)
Set
\[
\widetilde{R_\infty^{\psi_k}(k)^*} 
:= \Big(R_\rbar^{\Box,\psi_k}(k)^*[\alpha_p] \,\widehat{\otimes}_{\Oh} 
\big(\widehat{\otimes}_{v\in T(\rhobar),\,\Oh} R_{\rhobar_v}^{\Box,\psi_k}\big)\Big) \llbracket x_1,\ldots,x_g \rrbracket.
\]
We then have a surjection
\[
\widetilde{R_\infty^{\psi_k}(k)} \twoheadrightarrow \T^{N_{Q_n}}(M(k,Q_n,\Oh)^*)^{\Box},
\]
which in fact is
\begin{equation}\label{eqn_R_ap_surj_T_full}
\widetilde{R_\infty^{\psi_k}(k)} \twoheadrightarrow \T(M(k,Q_n,\Oh)^*)^{\Box},
\end{equation}
by Lemma~\ref{lem_U_ell_in_anemic}. As a result, the surjection \eqref{eqn_R_surj_T} factors through $R^{\psi_k}_\infty(k)^*$. Combining the maps, we then have \begin{equation}\label{eqn_R_surj_T_star_version}\begin{tikzcd}
    R_\infty^{\psi_k}(k)\arrow[r]\arrow[d,two heads] & \widetilde{R_\infty^{\psi_k}(k)}\arrow[d,two heads] & \\
    R_\infty^{\psi_k}(k)^* \arrow[r] & \widetilde{R_\infty^{\psi_k}(k)^*} \arrow[r, two heads] & \T(M(k,Q_n,\Oh)^*)^\Box
\end{tikzcd},\end{equation} where the horizontal maps are surjective by Proposition \ref{prop_completed_tensor_product_ideal}.

\begin{prop}\label{prop_R_inf_tilde_property}
    The ring $\widetilde{R_\infty^{\psi_k}(k)}$ is flat over $\Oh$. Its generic fiber $\widetilde{R_\infty^{\psi_k}(k)}[1/p]$ is isomorphic to $R_\infty^{\psi_k}(k)[1/p].$ 
\end{prop}

\begin{proof}
    By Remark~\ref{rmk_R_ap_completed_tensor_product_direct_sum}, 
it suffices to show that 
\[
\Big(R_u \widehat{\otimes}_\Oh \big(\widehat{\otimes}_{v \in T(\rhobar),\,\Oh} R_{\rhobar_v}^{\Box,\psi_k}\big)\Big) \llbracket x_1,\ldots,x_g \rrbracket
\] 
is flat over $\Oh$ for every $u \in U(\rbar)$. By Lemma~\ref{lem_kisin_completed_tensor_product}, 
it is enough to verify that $R_u$ and $R_{\rhobar_v}^{\Box,\psi_k}$ are flat over $\Oh$ for all $v \in T(\rhobar)$.  
Since $R_u$ is a subring of $R_\rbar^{\Box,\psi_k}[1/p]$, it is $\varpi$-torsion free and hence flat over $\Oh$.  
The flatness of $R_{\rhobar_v}^{\Box,\psi_k}$ follows from Remark~\ref{rmk_vexing_prime_local_deformation_ring}.  
For the generic fiber, we note that \begin{align*}
        \widetilde{R_\infty^{\psi_k}(k)} 
        &\xrightarrow{\sim}R_\rbar^{\Box,\psi_k}(k)[\alpha_p]\widehat{\otimes}_{R_\rbar^{\Box,\psi_k}(k)} R_\infty^{\psi_k}(k)\xrightarrow{\sim} R^{\Box,\psi_k}_\rbar(k)[\alpha_p]\otimes_{R_\rbar^{\Box,\psi_k}(k)}R_\infty^{\psi_k}(k),
    \end{align*} where the last isomorphism follows from \cite[Proposition 7.7.8]{EGAI}. 
    The second assertion then follows because $R^{\Box,\psi_k}_\rbar(k)$ and $R^{\Box,\psi_k}_\rbar(k)[\alpha_p]$ have the same generic fibers. 

    For simpler notation, we set \begin{equation}\label{eqn_def_B}
    B = \left(\widehat{\otimes}_{v\in T(\rhobar),\F}R_{\rhobar_v}^{\Box,\psi}/\varpi\right)\widehat{\otimes}_\F \F\llbracket x_1,\ldots,x_g \rrbracket\end{equation} 
    so that $$R^{\psi_k}_\infty(k)\otimes_\Oh\F = R^{\Box,\psi_k}_\rbar(k) \widehat{\otimes}_\Oh B \quad{\rm and}\quad \widetilde{R^{\psi_k}_\infty(k)}\otimes_\Oh\F = R^{\Box,\psi_k}_\rbar(k)[\alpha_p] \widehat{\otimes}_\Oh B.$$ 
\end{proof}

\subsection{Ultrapatching}
As mentioned in \S\ref{subsubsection_organization_of_the_paper}, we use Scholze's ultrapatching functor \cite{Scholze_2018_Lubin_Tate} to describe the Taylor--Wiles--Kisin patching process. We briefly recall the construction here, following \cite{ACampo_2023_Rigidity} and \cite{Manning_2021_Multiplicity}. 

A filter on the natural numbers $\N = \{1,2,\ldots\}$ is a set $\mathfrak{F}$ of subsets of $\N$ satisfying:
\begin{itemize}
    \item $\N\in \mathfrak{F}$;
    \item if $I,J\in \mathfrak{F},$ then $I\cap J\in \mathfrak{F};$
    \item if $I\in \mathfrak{F}$ and $I\subseteq J\subset \N,$ then $J\in \mathfrak{F}$; 
\end{itemize}
For example, the collection of cofinite subsets $$\{I\subseteq \N: \N\setminus I{\rm\ is\ finite}\}$$ is a filter, which we call the cofinite filter. A filter $\mathfrak{F}$ is called an \emph{ultrafilter} if it additionally satisfies:
\begin{itemize}
    \item for every $I\subseteq \N$, exactly one of $I$ or $\N\setminus I$ is in $\mathfrak{F}.$
\end{itemize}
An example of an ultrafilter is $$\left\{I\subseteq \N: n_0\in I\right\}$$ for some integer $n_0\in \N.$ These are called \emph{principal ultrafilters}. Ultrafilters that are not principal are called \emph{non-principal ultrafilters}. One can check that an ultrafilter $\mathfrak{F}$ is non-principal if and only if it contains no finite sets. The existence of non-principal ultrafilters follows from the theorem below applied to the cofinite filter. 
\begin{thm}
    Every filter $\mathfrak{F}$ such that $\emptyset\notin \mathfrak{F}$ is contained in an ultrafilter.  
\end{thm}
For a proof of the theorem, see \cite[Proposition 4.1.3]{Change_Keisler_1990_Model_Theorey}. The idea is to use Zorn's lemma to construct an ultrafilter as the maximal element of certain family of filters. 

From now on, we fix a non-principal ultrafilter $\mathfrak{F}$ on $\N.$ 
\begin{definition}
    Let $\{M_n\}_{n\ge 1}$ be a sequence of sets. The ultraproduct of $\{M_n\}$ is defined to be $$\mathcal{U}(\{M_n\}):=\left(\prod_{n\ge 1}M_n\right)/\sim$$ where $(x_n)\sim(x'_n)$ if and only if $\{n:x_n = x'_n\}$ is in $\mathfrak{F}.$
\end{definition}
Let $\mathcal{C}$ be the category of $R$-modules or $R$-algebras where $R$ is a ring. One can check that $\mathcal{U}(\{M_n\})$ is an object in $\mathcal{C}$ if $\{M_n\}$ is a sequence of objects in $\mathcal{C}$ and if $M_n\xrightarrow{f_n} N_n$ is a $\mathcal{C}$-morphism, there is a $\mathcal{C}$-morphism $$\mathcal{U}(\{f_n\}):\mathcal{U}(\{M_n\})\to \mathcal{U}(\{N_n\})$$ that is defined component-wise. 

We will mainly be interested in the case where the set $\{M_n\}$ has uniformly bounded finite cardinality. In this situation, the ultraproduct provides a functorial formulation of the pigeonhole principle, which is a key step in the Taylor--Wiles--Kisin patching method.

\begin{prop}\label{prop_bounded_ultraproduct}
    Let $\{M_n\}$ be a sequence of objects in $\mathcal{C}$ such that there exists a constant $C>0$ with $\# M_n < C$ for all $n\ge 1$. Then the ultraproduct $\mathcal{U}(\{M_n\})$ is isomorphic to $M_{n_0}$ for some $n_0$.
\end{prop}

\begin{proof}
    For each $k$, set 
    \[
        I_k := \{n\in \N : M_n \xrightarrow{\sim} M_k\}.
    \]
    Since $\#M_n<C$ for all $n$, there are finitely many $I_k$'s and they form a partition of $\N.$ As $\mathfrak{F}$ is an ultrafilter, exactly one of them belongs to $\mathfrak{F}$. Without loss of generality, assume $I_1\in\mathfrak{F}$.
    For $n\in I_1$, fix a chosen isomorphism $\phi_n:M_1\to M_n$, and for $n\notin I_1$ let $\phi_n:M_1\to M_n$ be the zero morphism. Define
    \[
        M_1 \longrightarrow \mathcal{U}(\{M_n\}),\qquad
        a \longmapsto [(\phi_n(a))].
    \]
    To see that this map is an isomorphism, let $(x_n)$ represent an element of $\mathcal{U}(\{M_n\})$.  
    For each $x\in M_1$, define
    \[
         I_x := \{ n\in I_1 : \phi_n^{-1}(x_n)=x\}.
    \]
    The sets $\{I_x : x\in M_1\}$ together with $\N\setminus I_1$ form a finite partition of~$\N$, so exactly one of them lies in $\mathfrak{F}$. Let $x_0\in M_1$ be such that $I_{x_0}\in\mathfrak{F}$. Then
    \[
        \{n : \phi_n(x_0)=x_n\} \supset I_{x_0} \in\mathfrak{F},
    \]
    so $[(\phi_n(x_0))]=[(x_n)]$. Thus every element of the ultraproduct is represented by the image of some $x_0\in M_1$.
    For the injectivity, suppose that $[(\phi_n(a))] = 0$ for some $a\in M_1.$ Then $$I := \{n:\phi_n(a)=0\}$$ is in $\mathfrak{F}.$ Since also $I_1\in\mathfrak{F}$, the intersection $I\cap I_1$ lies in $\mathfrak{F}$ and is therefore nonempty. For any $n\in I\cap I_1$, the map $\phi_n$ is an isomorphism, so $\phi_n(a)=0$ forces $a=0$. 
\end{proof}

\begin{prop}
    Let $\{A_n\},\{B_n\}$ and $\{C_n\}$ be sequences of objects in $\mathcal{C}$ such that each sequence has bounded cardinality. If $$0\to A_n \xrightarrow{f_n}B_n \xrightarrow{g_n} C_n\to 0$$ is exact for every $n$, then $$0\to \mathcal{U}(\{A_n\})\xrightarrow{\mathcal{U}(\{f_n\})} \mathcal{U}(\{B_n\})\xrightarrow{\mathcal{U}(\{g_n\})} \mathcal{U}(\{C_n\})\to 0$$ is also exact. 
\end{prop}

\begin{proof}
    By the proof of the proposition above, without loss of generality we may assume that 
$A_1 \xrightarrow{\sim} \mathcal{U}(\{A_n\})$, 
$B_1 \xrightarrow{\sim} \mathcal{U}(\{B_n\})$, and 
$C_1 \xrightarrow{\sim} \mathcal{U}(\{C_n\})$. 
Furthermore, we may fix $\mathcal{C}$-isomorphisms 
\[
    \alpha_n : A_1 \to A_n,\qquad 
    \beta_n : B_1 \to B_n,\qquad 
    \gamma_n : C_1 \to C_n
\]
for all $n$ in some $I \in \mathfrak{F}$ such that the diagram
\[
\begin{tikzcd}
    0 \arrow[r] & 
    A_1 \arrow[r] \arrow[d, "\alpha_n"] & 
    B_1 \arrow[r] \arrow[d, "\beta_n"] & 
    C_1 \arrow[r] \arrow[d, "\gamma_n"] & 0 \\
    0 \arrow[r] & 
    A_n \arrow[r] & 
    B_n \arrow[r] & 
    C_n \arrow[r] & 0
\end{tikzcd}
\]
commutes. 
It follows that the following diagram
\[
\begin{tikzcd}
    0 \arrow[r] & 
    A_1 \arrow[r, "f_1"] \arrow[d, "(\alpha_n)"] & 
    B_1 \arrow[r, "g_1"] \arrow[d, "(\beta_n)"] & 
    C_1 \arrow[r] \arrow[d, "(\gamma_n)"] & 0 \\
    0 \arrow[r] & 
    \mathcal{U}(\{A_n\}) \arrow[r, "\mathcal{U}(\{f_n\})"] & 
    \mathcal{U}(\{B_n\}) \arrow[r, "\mathcal{U}(\{g_n\})"] & 
    \mathcal{U}(\{C_n\}) \arrow[r] & 0
\end{tikzcd}
\]
also commutes. 
Since the top row is exact and the vertical arrows are isomorphisms, the bottom row is exact as well.
\end{proof}

\begin{remark}
    The proposition above in fact remains valid without the bounded cardinality hypothesis, provided that the ring $R$ is a finite local ring. One way to see this is to interpret the ultraproduct as a localization; see, for example, \cite[Lemma~2.2.3]{ACampo_2023_Rigidity}. 
    We choose to state the result in the bounded-cardinality form and give a concrete proof, since this formulation is sufficient for our purposes and avoids introducing localization.
\end{remark}

Let $I_n \subseteq S_\infty$ be the ideal generated by 
\[
\bigl((1+y_1)^{p^{n(q_1)}}-1, \ldots, (1+y_r)^{p^{n(q_r)}}-1\bigr),
\]
where 
\[
n(q_i) = v_p(q_i-1) \ge n
\]
for each Taylor--Wiles prime $q_i \in Q_n$. 

From now on, we take $R = S_\infty$ and work under the following assumption:

\begin{assumption} \label{assumption_ultrapatching}
    Let $\{M_n\}$ be a sequence of $\mathcal{C}$-objects such that the following hold:
    \begin{enumerate}
        \item Each $M_n$ is annihilated by $I_n.$
        \item There is a constant $r$, independent of $n$, such that for every $n\ge 1,$ there is a surjection $S_\infty^{\oplus r}\surj M_n$  
    \end{enumerate}
\end{assumption}

\begin{definition}
    Let $\{M_n\}$ be a sequence of $\mathcal{C}$-objects satisfying Assumption \ref{assumption_ultrapatching}. The ultrapatching functor is defined by $$\mathcal{P}(\{M_n\}) := \varprojlim_{\mathfrak{a}}\mathcal{U}(\{M_n\otimes_{S_\infty}S_\infty/\mathfrak{a}\})$$ where the limit is taken over open ideals $\mathfrak{a}$ of $S_\infty,$ which form a countable index set. If $\{N_n\}$ is another sequence satisfying Assumption \ref{assumption_ultrapatching} and, for each $n$, we have a $\mathcal{C}$-morphism $f_n:M_n\to N_n$, we define $$\mathcal{P}(\{f_n\}):\mathcal{P}(\{M_n\})\to \mathcal{P}(\{N_n\})$$ to be the morphism induced by the maps $f_n$.
\end{definition}

\begin{remark}
    By condition~(2) in Assumption~\ref{assumption_ultrapatching}, for every open ideal $\mathfrak{a}$ the cardinalities of the modules $M_n \otimes_{S_\infty} S_\infty / \mathfrak{a}$ are uniformly bounded. 
    Hence, the ultraproduct 
    $
        \mathcal{U}(\{M_n \otimes_{S_\infty} S_\infty / \mathfrak{a}\})
    $
    is isomorphic to $M_{n_0} \otimes_{S_\infty} S_\infty / \mathfrak{a}$ for some $n_0$. 
    It follows that $\mathcal{P}(\{M_n\})$ coincides with the patched module or ring in the Taylor--Wiles--Kisin patching method.
\end{remark}

\begin{example} \label{example_constant_patching}
    Let $M$ be a finitely generated $S_\infty$-module and let $M_n := M \otimes_{S_\infty} S_\infty/I_n.$ Then $\{M_n\}$ satisfies all the conditions in Assumption \ref{assumption_ultrapatching}. Fix an open ideal $\mathfrak{a},$ since $I_n\subseteq \mathfrak{a}$ for all but finitely many $n,$ by the proof of Proposition \ref{prop_bounded_ultraproduct}, we have $$\mathcal{U}(\{M_n\otimes_{S_\infty}S_\infty/\mathfrak{a}\})\xrightarrow{\sim} M\otimes_{S_\infty} S_\infty/\mathfrak{a}.$$ Hence, we have $$\mathcal{P}(\{M_n\}) \xrightarrow{\sim} \varprojlim_{\mathfrak{a}} \left(M\otimes_{S_\infty}S_\infty/\mathfrak{a}\right) \xrightarrow{\sim}M$$ because finitely generated $S_\infty$-modules are complete by \cite[Proposition 10.13]{Atiyah_Macdonald_2016_Commutative_Algebra}. If $M$ is further equipped with an $S_\infty$-algebra structure, then $\mathcal{P}(\{M_n\})$ is isomorphic to $M$ as an $S_\infty$-algebra.
\end{example}

\begin{prop} \label{prop_ultrapatching_right_exact}
    Let $\{A_n\},\{B_n\}$ and $\{C_n\}$ be sequences of objects in $\mathcal{C}$ that satisfy Assumption \ref{assumption_ultrapatching}. 
    \begin{enumerate}
        \item If $$A_n \to B_n \to C_n\to 0$$ is exact, then $$\mathcal{P}(\{A_n\})\to \mathcal{P}(\{B_n\})\to \mathcal{P}(\{C_n\})\to 0$$ is exact.
        \item If $$0\to A_n\to B_n\to C_n\to 0$$ is split exact, then $$0\to \mathcal{P}(\{A_n\})\to \mathcal{P}(\{B_n\})\to \mathcal{P}(\{C_n\})\to 0$$ is split exact.
        \item Let $\{M_n\}$ be a sequence of objects in $\mathcal{C}$, and let $\mathfrak{b}$ be an ideal of $S_\infty.$ Then we have $$\mathcal{P}(\{M_n\})\otimes_{S_\infty}S_\infty/\mathfrak{b}\xrightarrow{\sim}\mathcal{P}(\{M_n\otimes_{S_\infty}S_\infty/\mathfrak{b}\}).$$ 
    \end{enumerate}
\end{prop}

\begin{proof}
    \leavevmode
    \begin{enumerate}
        \item Tensor the first short exact sequence with $S_\infty/\mathfrak{a}$ for $\mathfrak{a}$ an open ideal of $S_\infty.$ We then obtain a right exact sequence $$A_n/\mathfrak{a}\to B_n/\mathfrak{a}\to C_n/\mathfrak{a}\to 0.$$  By (2) in Assumption \ref{assumption_ultrapatching}, $\{A_n/\mathfrak{a}\}_n$ has bounded cardinality. In this case, $\mathcal{U}$ is exact and we have a right exact sequence of finite modules $$\mathcal{U}(\{A_n/\mathfrak{a}\})\to \mathcal{U}(\{B_n/\mathfrak{a}\})\to \mathcal{U}(\{C_n/\mathfrak{a}\})\to 0.$$ The conclusion then follows from Corollary \ref{cor_mittag_leffler_right_exact}. 
        \item The operation $\cdot \otimes_{S_\infty} S_\infty/\mathfrak{a}$ preserves split exactness and $\mathcal{U}$ is exact. The conclusion follows from Lemma \ref{lem_inverse_limit_exactness} and its remark. 
        \item Unraveling the definition of the patching functor $\mathcal{P},$ we want to show $$\left(\varprojlim_{\mathfrak{a}}\mathcal{U}(\{M_n/\mathfrak{a}\})\right)\otimes_{S_\infty}S_\infty/\mathfrak{b}\xrightarrow{\sim} \varprojlim_{\mathfrak{a}}\mathcal{U}(\{M_n/\mathfrak{a}\otimes_{S_\infty}S_\infty/\mathfrak{b}\}).$$ By Lemma \ref{lem_inverse_limit_tensor_product}, it suffices to show that for each open ideal $\mathfrak{a}$ of $S_\infty,$ we have $$\mathcal{U}(\{M_n/\mathfrak{a}\})\otimes_{S_\infty}S_\infty/\mathfrak{b}\xrightarrow{\sim} \mathcal{U}(\{M_n/\mathfrak{a}\otimes_{S_\infty}S_\infty/\mathfrak{b}\}).$$ This is clear from the proof of Proposition \ref{prop_bounded_ultraproduct}.
    \end{enumerate}
\end{proof}

\begin{prop} \label{prop_patched_module_property}
    Let $\{M_n\}$ be a sequence satisfying Assumption~\ref{assumption_ultrapatching}. Then the patched module $\mathcal{P}(\{M_n\})$ satisfies the following properties: 
    \begin{enumerate}
        \item For an ideal $\mathfrak{c}$ of $S_\infty$, if every $M_n$ is finite free of the same rank $r$ over $S_\infty/(I_n+\mathfrak{c})$, then $\mathcal{P}(\{M_n\})$ is finite free of rank $r$ over $S_\infty/\mathfrak{c}$.
        \item As an $S_\infty$-module, $\mathcal{P}(\{M_n\})$ is finitely generated.
    \end{enumerate}
\end{prop}

\begin{proof}
    For (1), we notice that for a fixed open ideal $\mathfrak{a}$ of $S_\infty,$ we have $I_n\subseteq \mathfrak{a}$ for all but finitely many $n$. Hence, for these $n,$ the module $M_n/(\mathfrak{a}+\mathfrak{c})$ is finite free of rank $r$ over $S_\infty/(\mathfrak{a}+\mathfrak{c})$. It follows from the proof of Proposition \ref{prop_bounded_ultraproduct} that $\mathcal{U}(\{M_n/\mathfrak{a}\})$ is finite free over $S_\infty/(\mathfrak{a}+\mathfrak{c})$ of rank $r.$ Now $\mathcal{P}(\{M_n\})=\varprojlim_{\mathfrak{a}}\mathcal{U}(\{M_n/\mathfrak{a}\})$ is then finite free of rank $r$ over $$\varprojlim_{\mathfrak{a}}S_\infty/(\mathfrak{a}+\mathfrak{c}) \xrightarrow{\sim} \left(\varprojlim_\mathfrak{a} S_\infty/\mathfrak{a}\right)/\mathfrak{c} \xrightarrow{\sim}S_\infty/\mathfrak{c}$$ by Lemma \ref{lem_inverse_limit_tensor_product}.

    For (2), we will prove there is a finite free $S_\infty$-module that surjects onto $\mathcal{P}(\{M_n\}).$ Let $r$ be the constant in Assumption \ref{assumption_ultrapatching} and consider the finite free $S_\infty$-module $S_\infty^{\oplus r}.$ As in Example \ref{example_constant_patching}, form the sequence $\{(S_\infty/I_n)^{\oplus r}\}$. By Assumption~\ref{assumption_ultrapatching} (2), there are surjections $(S_\infty/I_n)^{\oplus r} \twoheadrightarrow M_n$ for every $n$. Applying the ultrapatching functor $\mathcal{P},$ Proposition \ref{prop_ultrapatching_right_exact} implies that $$\mathcal{P}(\{(S_\infty/I_n)^{\oplus r}\surj M_n\}): \mathcal{P}(\{(S_\infty/I_n)^{\oplus r}\})\to \mathcal{P}(\{M_n\})$$ is still surjective. But by part (1), $\mathcal{P}(\{(S_\infty/I_n)^{\oplus r}\})$ is free of rank $r$ over $S_\infty$ and the proof is complete. 
\end{proof}

\subsection{Patching the Coherent Cohomology}
We will apply the ultrapatching functor to certain modules coming from the coherent cohomology of modular curves. By Lemma~\ref{lem_U_ell_in_anemic}, we may write 
$\T^{p}(M(k,Q,\Oh)^*)$ in place of $\T^{pN_{Q}}(M(k,Q,\Oh)^*)$, and we adopt 
this simpler notation from now on.

The short exact sequence \eqref{eqn_anemic_full_cokernel_short_exact_sequence} gives rise to the following exact sequence of sequences of modules:
\[
    0 \to \left\{\T^{p}(M(k,Q_n,\Oh)^*)^{\Box}\right\}_{n\ge 1} 
    \to \left\{\T(M(k,Q_n,\Oh)^*)^{\Box}\right\}_{n\ge 1} 
    \to \left\{C(M(k,Q_n,\Oh)^*)^{\Box}\right\}_{n\ge 1} 
    \to 0
\] and \begin{equation}\label{eqn_mod_p_image_cokernel_hecke_sequence}
    0\to \left\{\T^{p}(M(k,Q_n,\F)^*)^{\Box} \right\}_{n\ge 1}\to \left\{\T(M(k,Q_n,\F)^*)^{\Box} \right\}_{n\ge 1}\to \left\{C(M(k,Q_n,\F)^*)^{\Box}\right\}_{n\ge 1}\to 0,
\end{equation} where the second sequence splits for $k\ge 2$ unless $k(\rbar)=1$, in which case it splits for $k\ge p^{2}$, by Proposition~\ref{prop_image_and_cokernel_of Hecke_algebras_finite_free}. 
Dualizing the short exact sequence
\[
    0 \to M(k,Q_{n},\F)
      \to M(k+p-1,Q_{n},\F)
      \to S(k+p-1,Q_{n}) 
      \to 0,
\]
we obtain an exact sequence
\[
    0 \to S(k+p-1,Q_{n})^{\vee}
      \to \T(M(k+p-1,Q_{n},\F))
      \to \T(M(k,Q_{n},\F))
      \to 0
\]
for every $n$. 
Since both $\T(M(k+p-1,Q_n,\F))$ and $\T(M(k,Q_n,\F))$ are free over $\F[\Delta_Q]$ when $k\ge 2,$ the sequence is split exact, and all its terms are free over $\F[\Delta_Q]$. 
Set
\[
    S(k+p-1,Q_{n})^{\vee,\Box}
    := S(k+p-1,Q_{n})^{\vee}
       \otimes_{R_{Q}^{\univ,\psi_{k+p-1}}}
             R_{Q}^{\Box,\psi_{k+p-1}}.
\]
Then we obtain another split exact sequence
\begin{equation}\label{eqn_Sk_dual_ses}
    0 \to \{S(k+p-1,Q_{n})^{\vee,\Box}\}_{n\ge 1}
      \to \{\T(M(k+p-1,Q_n,\F))^{\Box}\}_{n\ge 1}
      \to \{\T(M(k,Q_n,\F))^{\Box}\}_{n\ge 1}
      \to 0.
\end{equation}

\begin{prop} \label{prop_patching_the_coherent_cohomology}
    Assume \( k \ge 2 \). Then each term of the sequences above satisfies Assumption~\ref{assumption_ultrapatching}, so we may apply the patching functor \( \mathcal{P} \). The resulting patched modules satisfy the following properties:
    \begin{enumerate}
        \item The module \( \mathcal{P}(\{\T(M(k,Q_n,\Oh)^*)^{\Box}\}_{n\ge 1}) \) is finite free over \( S_\infty \).
        \item If we further assume that \( k \ge p^2 \) in the case where \( k(\rbar) = 1 \), then the patched module
        $\mathcal{P}(\{\T^{p}(M(k,Q_n,\F)^*)^\Box\}_{n\ge 1})$ and $
            \mathcal{P}\left(\left\{C(M(k,Q,\F)^*)^{\Box} \right\}_{n\ge 1}\right)
        $
        are finite free over \( S_\infty \otimes_\Oh \F \). The rank of $
            \mathcal{P}\left(\left\{C(M(k,Q,\F)^*)^{\Box} \right\}_{n\ge 1}\right)
        $ is \( \dim_\F M(j(k), \varnothing, \F)^* \).
        \item The patched module $\mathcal{P}\left(\left\{S(k+p-1,Q_n)^{\vee,\Box}\right\}_{n\ge 1}\right)$ is finite free over $S_\infty\otimes_\Oh\F.$
    \end{enumerate}
\end{prop}

\begin{proof}
    By construction, the $n$-th module in every sequence above is 
annihilated by $I_n$. Thus we only need to verify condition~(2) in 
Assumption~\ref{assumption_ultrapatching} and prove the claimed 
properties of the patched modules.

Since $k\ge 2$, Lemma~\ref{lem_finite_free_over_group_alg} implies that
\[
    \T(M(k,Q_n,\Oh)^*)^{\Box} 
        \xrightarrow{\sim} 
    \T(M(k,Q_n,\Oh)^*) \otimes_{\Oh} \mathcal{J}
\]
is finite free over 
\(
    \Oh[\Delta_{Q_n}] \otimes_{\Oh} \mathcal{J} 
        \xrightarrow{\sim} 
    S_\infty/I_n,
\)
with rank independent of $n$.  
Hence the sequence 
$\{\T(M(k,Q_n,\Oh)^*)^{\Box}\}_{n\ge 1}$ satisfies 
Assumption~\ref{assumption_ultrapatching}.  
It follows that its reduction mod~$\varpi$ 
$\{\T(M(k,Q_n,\F)^*)^{\Box}\}_{n\ge 1}$,  
the quotient sequence 
$\{C(M(k,Q_n,\Oh)^*)^{\Box}\}_{n\ge 1}$,  
and its mod-$\varpi$ reduction  
$\{C(M(k,Q_n,\F)^*)^{\Box}\}_{n\ge 1}$  
all satisfy Assumption~\ref{assumption_ultrapatching}~(2). Since \eqref{eqn_Sk_dual_ses} is split exact,  
$S(k+p-1,Q_n)^{\vee,\Box}$ is a quotient of 
$\T(M(k+p-1,Q_n,\F))^{\Box}$.  
Thus $\{S(k+p-1,Q_n)^{\vee,\Box}\}_{n\ge 1}$ 
also satisfies Assumption~\ref{assumption_ultrapatching}~(2).

By Proposition~\ref{prop_patched_module_property}, the patched module
$
    \mathcal{P}\Bigl(\{\T(M(k,Q_n,\Oh))^{\Box}\}_{n\ge 1}\Bigr)
$
is finite free over $S_\infty$.  
Hence its endomorphism ring
\[
    \End_{S_\infty}\!\Bigl(
        \mathcal{P}\bigl(\{\T(M(k,Q_n,\Oh))^{\Box}\}_{n\ge 1}\bigr)
    \Bigr)
\]
is finite free over $S_\infty$, and therefore every $S_\infty$-submodule 
of this ring is finitely generated.  
Kisin's theorem below then gives an embedding
\[
    R_\infty^{\psi_k}(k)
        \hookrightarrow
    \End_{S_\infty}\!\Bigl(
        \mathcal{P}\bigl(\{\T(M(k,Q_n,\Oh))^{\Box}\}_{n\ge 1}\bigr)
    \Bigr),
\]
so $R_\infty^{\psi_k}(k)$ is finite over $S_\infty$. The surjection
\[
    R_\infty^{\psi_k}(k)
        \twoheadrightarrow
    \T^p(M(k,Q_n,\Oh)^*)^{\Box}
\]
from \eqref{eqn_R_surj_T} shows that 
$\{\T^p(M(k,Q_n,\Oh)^*)^{\Box}\}_{n\ge 1}$ and its quotient 
$\{\T^p(M(k,Q_n,\F)^*)^{\Box}\}_{n\ge 1}$ also satisfy 
Assumption~\ref{assumption_ultrapatching}~(2).
    
    The freeness of $\mathcal{P}\left(\left\{S(k+p-1,Q_n)^{\vee,\Box}\right\}_{n\ge 1}\right)$ follows from \eqref{eqn_Sk_dual_ses} and the freeness of the other two terms, by Proposition~\ref{prop_patched_module_property}. 
    Under the weight assumption, the same reasoning applies to the sequence \( \{C(M(k,Q,\F)^*)^{\Box} \}_{n \ge 1} \) and $\left\{\T^{p}(M(k,Q,\F)^*)^{\Box}\right\}_{n\ge 1}.$ The rank follows from Proposition~\ref{prop_image_and_cokernel_of Hecke_algebras_finite_free}.
\end{proof}

\begin{thm}[Kisin] \label{thm_kisin_faithful}
    Assume that $k\ge 2.$ Then $\mathcal{P}(\{\T(M(k,Q_n,\Oh))^{\Box}\}_{n\ge 1})$ is a faithful $R_\infty^{\psi_k}(k)$-module. 
\end{thm}

\begin{remark}
    Kisin first proved this (see \cite[Corollary~2.2.17]{Kisin09FontaineMazur}) under the additional hypothesis that 
    $\rbar$ is not a twist of the extension of $1$ by~$\epsilon$.  
    When $p\ge 3$, this restriction was later removed by 
    Pa\v{s}k\=unas~\cite{Paskunas_2015_BM_conj}, 
    Hu--Tan~\cite{HuTan2013TheBC} and Tung \cite{tung2021automorphy}.
\end{remark}
\begin{cor} \label{cor_R_inf_tilde_faithful}
    Assume that $k\ge 2.$ Then $\mathcal{P}(\{\T(M(k,Q_n,\Oh))^{\Box}\}_{n\ge 1})$ is also a faithful $\widetilde{R_\infty^{\psi_k}(k)}$-module. 
\end{cor}

\begin{proof}
    Since $\widetilde{R_\infty^{\psi_k}(k)}$ is $\Oh$-flat by 
    Proposition~\ref{prop_R_inf_tilde_property}, it suffices to show that 
    $\widetilde{R_\infty^{\psi_k}(k)}[1/p]$ acts faithfully on 
    \(
        \mathcal{P}\bigl(\{\T(M(k,Q_n,\Oh))^{\Box}\}_{n\ge 1}\bigr)[1/p].
    \)
    By the same proposition, this is equivalent to showing that 
    $R_\infty^{\psi_k}(k)[1/p]$ acts faithfully on 
    \(
        \mathcal{P}\bigl(\{\T(M(k,Q_n,\Oh))^{\Box}\}_{n\ge 1}\bigr)[1/p].
    \)
    Since the patched module 
    $\mathcal{P}\bigl(\{\T(M(k,Q_n,\Oh))^{\Box}\}_{n\ge 1}\bigr)$ 
    is finite free over $S_\infty$, hence $\Oh$-torsion free, the desired 
    conclusion follows from Theorem~\ref{thm_kisin_faithful}.
\end{proof}

The surjection \eqref{eqn_R_ap_surj_T_full} induces, for each $n$, a surjective map
\[
\widetilde{R_\infty^{\psi_k}(k)} \otimes_{S_\infty} S_\infty/I_n \twoheadrightarrow \T(M(k, Q_n, \Oh))^{\Box}.
\]
Applying Proposition~\ref{prop_ultrapatching_right_exact} and Example~\ref{example_constant_patching}, we obtain a surjective homomorphism
\[
\widetilde{R^{\psi_k}_\infty(k)} \twoheadrightarrow \mathcal{P}\left(\left\{\T(M(k,Q_n,\Oh))^{\Box}\right\}_{n \ge 1}\right).
\]
It is straightforward to check that this map is compatible with the ring structures so it is a ring homomorphism. The same argument applies to the sequence $\{\T^{p}(M(k,Q_n,\Oh))^{\Box}\}_{n \ge 1}$ and $R_\infty^{\psi_k}(k)$, giving rise to a surjective ring homomorphism
\[
R^{\psi_k}_\infty(k) \twoheadrightarrow \mathcal{P}\left(\left\{\T^{p}(M(k,Q_n,\Oh))^{\Box}\right\}_{n \ge 1}\right).
\]

We are now ready to summarize how the patched modules and rings are identified with the local deformation rings, and to deduce the resulting maps and properties of local deformation rings arising from the ultrapatching construction.

\begin{lem} \label{lem_Hecke_algebras_patch_to_deformation_rings}
Assume that $k \ge 2$. Then:
\begin{enumerate}

    \item  
    In the commutative diagram
    \[
    \begin{tikzcd}
        R^{\psi_k}_\infty(k)^* \arrow[r, two heads]\arrow[d] & 
        \mathcal{P}\!\left(\left\{\T^{p}(M(k,Q_n,\Oh)^*)^{\Box}\right\}_{n \ge 1}\right)
        \arrow[d] \\
        \widetilde{R^{\psi_k}_\infty(k)^*} \arrow[r, two heads] & 
        \mathcal{P}\!\left(\left\{\T(M(k,Q_n,\Oh)^*)^{\Box}\right\}_{n \ge 1}\right)
    \end{tikzcd}
    \]
    both horizontal maps are isomorphisms.  
    Moreover, $\widetilde{R^{\psi_k}_\infty(k)^*}$ is finite free over $S_\infty$.

    \item  
    There are natural isomorphisms
    \[
        \bigl(\widetilde{R^{\psi_k}_\infty(k)^*} / R^{\psi_k}_\infty(k)^*\bigr)
        \otimes_{\Oh} \F 
        \xrightarrow{\;\sim\;}
        \mathcal{P}\!\left(\{ C(M(k,Q_n,\F)^*)^{\Box} \}\right)
    \] and $$\Image\!\left(R^{\psi_k}_\infty(k)^*\otimes_{\Oh}\F 
        \to \widetilde{R^{\psi_k}_\infty(k)^*}\otimes_{\Oh}\F\right)
        \xrightarrow{\sim} \Image \!\left(\mathcal{P}(\T^p(M(k,\Q_n,\F)^*)^{\Box})\to \mathcal{P}(\T(M(k,\Q_n,\F)^*)^{\Box})\right).$$
    If in addition $k(\rbar)=1$ and $k \ge p^{2}$, then
    \[
        \Image\!\left(R^{\psi_k}_\infty(k)^*\otimes_{\Oh}\F 
        \to \widetilde{R^{\psi_k}_\infty(k)^*}\otimes_{\Oh}\F\right)
        \xrightarrow{\;\sim\;}
        \mathcal{P}\!\left(\{\T^{p}(M(k,Q_n,\F)^*)^{\Box}\}\right).
    \]
    In this case, both 
    \[
        \bigl(\widetilde{R^{\psi_k}_\infty(k)^*} / R^{\psi_k}_\infty(k)^*\bigr)\otimes_{\Oh}\F 
        \quad\text{and}\quad
        \Image\!\left(R^{\psi_k}_\infty(k)^*\otimes_{\Oh}\F 
        \to \widetilde{R^{\psi_k}_\infty(k)^*}\otimes_{\Oh}\F\right)
    \]
    are finite free over $S_\infty \otimes_{\Oh}\F$.  
    The former has rank equal to $\dim_{\F} M(j(k),\varnothing,\F)^*$.

    \item  
    For every $k \ge 2$, there is a natural surjection
    \[
        \widetilde{R^{\psi_{k+p-1}}_\infty(k+p-1)^*} \otimes_{\Oh}\F
        \twoheadrightarrow
        \widetilde{R^{\psi_{k}}_\infty(k)^*} \otimes_{\Oh}\F.
    \]
    It restricts to a surjection
    \[
        \Image\!\left(R^{\psi_{k+p-1}}_\infty(k+p-1)^*\otimes_{\Oh}\F 
        \to \widetilde{R^{\psi_{k+p-1}}_\infty(k+p-1)^*}\otimes_{\Oh}\F\right)
        \twoheadrightarrow
        \Image\!\left(R^{\psi_k}_\infty(k)^*\otimes_{\Oh}\F 
        \to \widetilde{R^{\psi_k}_\infty(k)^*}\otimes_{\Oh}\F\right).
    \]
    When $* = \ ,$ the kernel is given by
    \[
        \ker\!\left(
            \widetilde{R^{\psi_{k+p-1}}_\infty(k+p-1)}\otimes_{\Oh}\F
            \longrightarrow
            \widetilde{R^{\psi_{k}}_\infty(k)}\otimes_{\Oh}\F
        \right)
        \;=\;
        \mathcal{P}\!\left(
            \{ S(k+p-1, Q_n)^{\vee,\Box} \}_{n\ge 1}
        \right).
    \]

\end{enumerate}
\end{lem}

\begin{proof}\leavevmode
    \begin{enumerate}
        \item To prove that the horizontal surjections are isomorphisms, it suffices to show 
that they are injective. For $R_\infty^{\psi_k}(k),$ this follows immediately from 
Theorem~\ref{thm_kisin_faithful} and Corollary~\ref{cor_R_inf_tilde_faithful}. When $* \in \{\ord,\no\},$ we note that $$\widetilde{R_\infty^{\psi_k}(k)} \xrightarrow{\sim}\widetilde{R_\infty^{\psi_k}(k)^{\ord}}\oplus \widetilde{R_\infty^{\psi_k}(k)^{\no}}$$ and $$\mathcal{P}\left(\{\T(M(k,Q_n,\Oh))^{\Box}\}_{n\ge 1}\right)\xrightarrow{\sim}\mathcal{P}\left(\{\T(M(k,Q_n,\Oh)^\ord)^{\Box}\}_{n\ge 1}\right)\oplus \mathcal{P}\left(\{\T(M(k,Q_n,\Oh)^\no)^{\Box}\}_{n\ge 1}\right)$$ and the maps respect the splittings. 
        \item Reduce the short exact sequence \eqref{eqn_anemic_full_cokernel_short_exact_sequence} 
modulo $\varpi$, and we obtain a right exact sequence
\[
    \T^p(M(k,Q_n,\Oh)^*)\otimes_\Oh\F
        \to \T(M(k,Q_n,\Oh)^*)\otimes_\Oh\F
        \to C(M(k,Q_n,\Oh)^*)\otimes_\Oh\F
        \to 0.
\]
By Corollary~\ref{cor_finite_level_mod_p_commute}, this may be rewritten as
\[
    \T^p(M(k,Q_n,\Oh)^*)\otimes_\Oh\F
        \to \T(M(k,Q_n,\F)^*)
        \to C(M(k,Q_n,\F)^*)
        \to 0.
\]
which induces a right exact sequence
\[
    \mathcal{P}(\{\T^p(M(k,Q_n,\Oh)^*)^{\Box}\otimes_\Oh\F\}_{n\ge 1})
        \to \mathcal{P}(\{\T(M(k,Q_n,\F)^*)^{\Box}\}_{n\ge 1})
        \to \mathcal{P}(\{C(M(k,Q_n,\F)^*)^{\Box}\}_{n\ge 1})
        \to 0.
\]

By Proposition~\ref{prop_ultrapatching_right_exact} and the commutative diagram in~(1), we obtain
\[
\begin{tikzcd}
    R^{\psi_k}_\infty(k)^*\otimes_\Oh\F
        \arrow[r,"\sim"] \arrow[d]
    &
    \mathcal{P}\!\left(\{\T^{p}(M(k,Q_n,\Oh)^*)^{\Box}\otimes_\Oh\F\}_{n\ge 1}\right)
        \arrow[d]
    \\
    \widetilde{R^{\psi_k}_\infty(k)^*}\otimes_\Oh\F
        \arrow[r,"\sim"] \arrow[d, two heads]
    &
    \mathcal{P}\!\left(\{\T(M(k,Q_n,\F)^*)^{\Box}\}_{n\ge 1}\right)
        \arrow[d, two heads]
    \\
    (\widetilde{R^{\psi_k}_\infty(k)^*}/R^{\psi_k}_\infty(k)^*)\otimes_\Oh\F
        \arrow[r,"\sim"]
    &
    \mathcal{P}(\{C(M(k,Q_n,\F)^*)^{\Box}\})
\end{tikzcd}
\] and $$\Image\!\left(R^{\psi_k}_\infty(k)^*\otimes_{\Oh}\F 
        \to \widetilde{R^{\psi_k}_\infty(k)^*}\otimes_{\Oh}\F\right)
        \xrightarrow{\sim} \Image \left(\mathcal{P}(\T^p(M(k,\Q_n,\F)^*)^{\Box})\to \mathcal{P}(\T(M(k,\Q_n,\F)^*)^{\Box})\right).$$

If we further assume $k \ge p^2$ when $k(\rbar)=1$, the sequence
\eqref{eqn_mod_p_image_cokernel_hecke_sequence}
\[
    0 \to \{\T^{p}(M(k,Q_n,\F)^*)^{\Box}\}
      \to \{\T(M(k,Q_n,\F)^*)^{\Box}\}
      \to \{C(M(k,Q_n,\F)^*)^{\Box}\}
      \to 0
\]
is split exact, and therefore induces a split short exact sequence of patched modules
\[
    0 \to
    \mathcal{P}(\{\T^{p}(M(k,Q_n,\F)^*)^{\Box}\})
      \to \mathcal{P}(\{\T(M(k,Q_n,\F)^*)^{\Box}\})
      \to \mathcal{P}(\{C(M(k,Q_n,\F)^*)^{\Box}\})
      \to 0.
\]

Hence
\[
    \Image\!\left(
        R^{\psi_k}_\infty(k)^*\otimes_\Oh\F
        \longrightarrow
        \widetilde{R^{\psi_k}_\infty(k)^*}\otimes_\Oh\F
    \right)
    \xrightarrow{\;\sim\;}
    \mathcal{P}\!\left(\{\T^{p}(M(k,Q_n,\F)^*)^{\Box}\}_{n\ge 1}\right).
\]

The assertions on finite freeness and on the rank follow from
Proposition~\ref{prop_patching_the_coherent_cohomology}.

\item By Corollary~\ref{cor_finite_level_surjection}, for each $n\ge 1$ and $k\ge 2,$ we have the commutative diagram
\[
\begin{tikzcd}
    \T^{p}(M(k+p-1,Q_n,\F)^*) \arrow[r,hook] \arrow[d,two heads] &
    \T(M(k+p-1,Q_n,\F)^*) \arrow[d,two heads] \\
    \T^{p}(M(k,Q_n,\F)^*) \arrow[r,hook] &
    \T(M(k,Q_n,\F)^*),
\end{tikzcd}
\]
which induces a surjection
\[
C(M(k+p-1,Q_n,\F)^*) \twoheadrightarrow C(M(k,Q_n,\F)^*).
\]
By right exactness of the patching functor (Proposition~\ref{prop_ultrapatching_right_exact}) and part~(2), we obtain
\[
\begin{tikzcd}
    \widetilde{R_\infty^{\psi_{k+p-1}}(k+p-1)^*}\otimes_\Oh\F
        \arrow[r,two heads] \arrow[d,two heads] &
    \bigl(\widetilde{R_\infty^{\psi_{k+p-1}}(k+p-1)^*}/
          R_\infty^{\psi_{k+p-1}}(k+p-1)^*\bigr)\otimes_\Oh\F
        \arrow[d,two heads] \\
    \widetilde{R_\infty^{\psi_{k}}(k)^*}\otimes_\Oh\F
        \arrow[r,two heads] &
    \bigl(\widetilde{R_\infty^{\psi_{k}}(k)^*}/
          R_\infty^{\psi_{k}}(k)^*\bigr)\otimes_\Oh\F,
\end{tikzcd}
\]
and hence a natural map between the kernels of the two horizontal maps
\[
\Image\!\left(
    R^{\psi_{k+p-1}}_\infty(k+p-1)^*\otimes_\Oh\F
    \to
    \widetilde{R^{\psi_{k+p-1}}_\infty(k+p-1)^*}\otimes_\Oh\F
\right)
\longrightarrow
\Image\!\left(
    R^{\psi_{k}}_\infty(k)^*\otimes_\Oh\F
    \to
    \widetilde{R^{\psi_{k}}_\infty(k)^*}\otimes_\Oh\F
\right).
\]
On the other hand, again by part~(2) and Proposition~\ref{prop_ultrapatching_right_exact}, this map fits in the commutative diagram
\[
\begin{tikzcd}
    \mathcal{P}(\{\T^{p}(M(k+p-1,Q_n,\F)^*)\})
        \arrow[d,two heads]
        \arrow[r,two heads] &
    \Image(R^{\psi_{k+p-1}}_\infty(k+p-1)^*\otimes_\Oh\F
           \to
           \widetilde{R^{\psi_{k+p-1}}_\infty(k+p-1)^*}\otimes_\Oh\F)
        \arrow[d]  \\
    \mathcal{P}(\{\T^{p}(M(k,Q_n,\F)^*)\})
        \arrow[r,two heads] &
    \Image(R^{\psi_{k}}_\infty(k)^*\otimes_\Oh\F
           \to
           \widetilde{R^{\psi_{k}}_\infty(k)^*}\otimes_\Oh\F),
\end{tikzcd}
\]
so the map between images must be surjective. 

Finally, the description of 
\(
\ker\bigl(
    \widetilde{R^{\psi_{k+(p-1)}}_\infty(k)}\otimes_\Oh\F
    \to
    \widetilde{R^{\psi_{k}}_\infty(k)}\otimes_\Oh\F
\bigr)
\)
follows from part~(1), the exact sequence~\eqref{eqn_Sk_dual_ses}, and
Proposition~\ref{prop_ultrapatching_right_exact}.

    \end{enumerate}
\end{proof}

\begin{example} \label{example_surjection_of_patched_rings}
    When $\rbar$ and the weights $k$ are as in Example~\ref{example_crystalline_deformation_rings}, the surjections in Lemma~\ref{lem_Hecke_algebras_patch_to_deformation_rings}(3) can be described very explicitly. In these cases we have
    \[
        R_{\rbar}^{\Box,\psi_k}(k) = R_{\rbar}^{\Box,\psi_k}(k)[\alpha_p],
    \]
    and therefore
    \[
        R_\infty^{\psi_k}(k)\otimes_\Oh\F
        \;=\;
        R^{\Box,\psi_k}_\rbar(k)\, \widehat{\otimes}_\Oh B
        \;=\;
        R^{\Box,\psi_k}_\rbar(k)[\alpha_p]\, \widehat{\otimes}_\Oh B
        \;=\;
        \widetilde{R_\infty^{\psi_k}(k)}\otimes_\Oh\F.
    \]
    Thus Lemma~\ref{lem_Hecke_algebras_patch_to_deformation_rings}(3) yields a surjection
    \[
        R^{\Box,\psi_{k_0+p-1}}_\rbar(k_0+p-1)\widehat{\otimes}_\Oh B
        \;\twoheadrightarrow\;
        R^{\Box,\psi_{k_0}}_\rbar(k_0)\widehat{\otimes}_\Oh B.
    \]
    When $k_0$ from Example~\ref{example_crystalline_deformation_rings} is not $2$, we have
    \[
        R^{\Box,\psi_{k_0+p-1}}_\rbar(k_0+p-1)\widehat{\otimes}_\Oh B
        \;\xrightarrow{\sim}\;
        \F\llbracket x_1,x_2,x_3,y,\alpha_p\rrbracket/(y\alpha_p)
        \widehat{\otimes}_\F B
    \]
    and
    \[
        R^{\Box,\psi_{k_0}}_\rbar(k_0)\widehat{\otimes}_\Oh B
        \;\xrightarrow{\sim}\;
        \F\llbracket x_1,x_2,x_3,\alpha_p\rrbracket
        \widehat{\otimes}_\F B.
    \]
    The surjection sends each framing variable $x_i$ to $x_i$, sends $\alpha_p$ to $\alpha_p$, and restricts to the identity on $B$. The element $y$ must map to~$0$, since its image must an annihilator of $\alpha_p$ in 
    $\F\llbracket x_1,x_2,x_3,\alpha_p\rrbracket\widehat{\otimes}_\F B$, and this is $0$.
    Hence
    \[
        (y)
        \;=\;
        \ker\!\left(
            \widetilde{R^{\psi_{k_0+(p-1)}}_\infty(k_0)}\otimes_\Oh\F
            \longrightarrow
            \widetilde{R^{\psi_{k_0}}_\infty(k_0)}\otimes_\Oh\F
        \right)
        \;=\;
        \mathcal{P}\!\left(
            \{S(k_0+(p-1),Q_n)^{\vee,\Box}\}_{n\ge 1}
        \right).
    \]
    When $k_0=2$ or when $\rbar|_{I_p}\sim \begin{pmatrix}
        \epsilon & *\\
        0 & 1
    \end{pmatrix}$ but $\rbar$ itself is not up to twist an extension of $1$ by $\epsilon$, the surjection is an isomorphism, and hence the kernel is trivial.
\end{example}

The ring $S_\infty$ encodes the patching data. To pass from the infinite level (the patched modules) to the finite level (the characteristic $p$ Hecke algebras), we quotient by the maximal ideal $\mathfrak{m}_{S_\infty}$. The following lemma makes this precise.

\begin{lem} \label{lem_quotient_of_patched_modules}
    Assume that $k\ge 2$. Then
    \[
        \widetilde{R^{\psi_{k}}_\infty(k)^*}\otimes_{S_\infty}S_\infty/\mathfrak{m}_{S_\infty}
        \xrightarrow{\;\sim\;} \T(M(k,\varnothing,\Oh)^*)\otimes_\Oh\F,
    \]
    and
    \[
        \ker\Bigl( \widetilde{R^{\psi_{k+(p-1)}}_\infty(k)}\otimes_\Oh\F \to \widetilde{R^{\psi_{k}}_\infty(k)}\otimes_\Oh\F\Bigr)
        \otimes_{S_\infty}S_\infty/\mathfrak{m}_{S_\infty} 
        \xrightarrow{\sim} S(k,\varnothing)^\vee.
    \]
    Moreover, if $k\ge p^2$ when $k(\rbar) = 1,$ then
    \[
        \Image\Bigl(R^{\psi_{k}}_\infty(k)^*\otimes_\Oh\F
        \;\longrightarrow\;
        \widetilde{R^{\psi_{k}}_\infty(k)^*}\otimes_\Oh\F\Bigr)
        \otimes_{S_\infty}S_\infty/\mathfrak{m}_{S_\infty}
        \xrightarrow{\sim} \T^{p}(M(k,\varnothing,\F)^*).
    \]
\end{lem}

\begin{proof}
    By Lemma~\ref{lem_Hecke_algebras_patch_to_deformation_rings} we have
\[
    \widetilde{R^{\psi_{k}}_\infty(k)^*}
    \;\xrightarrow{\sim}\;
    \mathcal{P}(\{\T(M(k,Q_n,\Oh)^*)\})
    = \varprojlim_\mathfrak{a} \mathcal{U}(\{\T(M(k,Q_n,\Oh)^*)/\mathfrak{a}\}).
\]
By Proposition \ref{prop_ultrapatching_right_exact} and Lemma~\ref{lem_finite_free_over_group_alg},
\begin{multline*}
    \widetilde{R^{\psi_{k}}_\infty(k)^*}\otimes_{S_\infty}S_\infty/\mathfrak{m}_{S_\infty}
    \;\xrightarrow{\sim}\;
    \varprojlim_\mathfrak{a} \mathcal{U}(\{\T(M(k,Q_n,\Oh)^*)^{\Box}/\mathfrak{a}+\mathfrak{m}_{S_\infty}\}) \\
    \xrightarrow{\sim}\;
    \varprojlim_\mathfrak{a} \mathcal{U}(\{\T(M(k,\varnothing,\F)^*)/\mathfrak{a}\})
    \;\xrightarrow{\sim}\;
    \T(M(k,\varnothing,\F)^*),
\end{multline*}
so the first isomorphism in the lemma follows. 

Next, since all terms in the short exact sequence
\[
0\to \ker\left( \widetilde{R^{\psi_{k+(p-1)}}_\infty(k)}\otimes_\Oh\F \to \widetilde{R^{\psi_{k}}_\infty(k)}\otimes_\Oh\F\right) 
\to \widetilde{R^{\psi_{k+(p-1)}}_\infty(k)}\otimes_\Oh\F 
\to \widetilde{R^{\psi_{k}}_\infty(k)}\otimes_\Oh\F \to 0
\]
are finite free $S_\infty\otimes_\Oh\F$-modules by Proposition \ref{prop_patching_the_coherent_cohomology}, exactness is preserved after tensoring with $S_\infty/\mathfrak{m}_{S_\infty}$. Hence we obtain
\begin{multline*}
0\to \ker\left( \widetilde{R^{\psi_{k+(p-1)}}_\infty(k)}\otimes_\Oh\F \to \widetilde{R^{\psi_{k}}_\infty(k)}\otimes_\Oh\F\right)\otimes_{S_\infty}S_\infty/\mathfrak{m}_{S_\infty} \\
\to \T(M(k+(p-1),\varnothing,\F))
\to \T(M(k,\varnothing,\F))\to 0,
\end{multline*}
so the kernel is identified with $S(k,\varnothing)^\vee$.

Now assume $k\ge p^2$ if $k(\rbar)=1.$ Consider the commutative diagram
\[
\begin{tikzcd}
    R^{\psi_k}_\infty(k)^* \otimes_{S_\infty}S_\infty/\mathfrak{m}_{S_\infty}
    \arrow[r]\arrow[d,two heads] &
    \widetilde{R^{\psi_k}_\infty(k)^*}\otimes_{S_\infty}S_\infty/\mathfrak{m}_{S_\infty}
    \arrow[d,"\wr"] \\
    \T^{p}(M(k,\varnothing,\Oh)^*)\otimes_\Oh\F
    \arrow[r]\arrow[d,two heads] &
    \T(M(k,\varnothing,\F)^*)\\
    \T^{p}(M(k,\varnothing,\F)^*)\arrow[ur,hook]&
\end{tikzcd}
\]
Then the image of
\(
    R^{\psi_k}_\infty(k)^* \otimes_{S_\infty}S_\infty/\mathfrak{m}_{S_\infty}
\)
inside
\(
    \widetilde{R^{\psi_{k}}_\infty(k)^*}\otimes_{S_\infty}S_\infty/\mathfrak{m}_{S_\infty}
\)
is therefore isomorphic to $\T^{p}(M(k,\varnothing,\F)^*)$. On the other hand, note that
\(
    \Image(R^{\psi_{k}}_\infty(k)^*\otimes_\Oh\F\to \widetilde{R^{\psi_{k}}_\infty(k)^*}\otimes_\Oh\F)
\)
is finite free over $S_\infty\otimes_\Oh\F$ by Lemma~\ref{lem_Hecke_algebras_patch_to_deformation_rings}, so the inclusion
\[
    \Image(R^{\psi_{k}}_\infty(k)\otimes_\Oh\F \to \widetilde{R^{\psi_{k}}_\infty(k)^*}\otimes_\Oh\F)
    \hookrightarrow
    \widetilde{R^{\psi_{k}}_\infty(k)}\otimes_\Oh\F
\]
remains injective after tensoring with \(S_\infty/\mathfrak{m}_{S_\infty}\). Hence
\[
    \Image(R^{\psi_{k}}_\infty(k)^*\otimes_\Oh\F\to \widetilde{R^{\psi_{k}}_\infty(k)^*}\otimes_\Oh\F)
    \otimes_{S_\infty}S_\infty/\mathfrak{m}_{S_\infty}
\]
coincides $$\Image\left(R^{\psi_k}_\infty(k)^* \otimes_{S_\infty}S_\infty/\mathfrak{m}_{S_\infty}\to \widetilde{R^{\psi_{k}}_\infty(k)^*}\otimes_{S_\infty}S_\infty/\mathfrak{m}_{S_\infty}\right),$$
so the claim follows.
\end{proof}

\section{Applications}
In this section, we prove Theorem Theorem \ref{thm_crystalline_deformation_ring_Cohen_Macaulay} 
and Theorem \ref{thm_crystalline_stronger_topology} with Assumption 
\ref{assumption_semisimplification} replaced by the weaker Assumption 
\ref{assumption_weaker}. We also prove the multiplicity one theorem for 
Serre's quaternionic modular forms. 

Throughout, for $R=\Oh$ or $\F$, we let $M(k,Q,R)^*$ be either $M(k,Q,R)$ 
or $M(k,Q,R)^\no$. We let $R_\rbar^{\Box,\psi}(k)^*$ be either 
$R_\rbar^{\Box,\psi}(k)$ or $R_\rbar^{\Box,\psi}(k)^\no$, and 
$k(\rbar)^*$ be either $k(\rbar)$ or $k(\rbar)^{\no}$. The choice of $*$ 
is understood to be fixed within each statement.

\subsection{The Cohen--Macaulay Property of Crystalline Deformation Rings}

\begin{thm}[Emerton]\label{thm_R_ap_cohen_macaulay}
    Suppose that $\rbar:G_{\Q_p}\to \GL_2(\F)$ satisfies Assumption \ref{assumption_weaker}. Then $R^{\Box,\psi}_\rbar(k)[\alpha_p]$ is Cohen--Macaulay for every $k \ge 2$.
\end{thm}

\begin{proof}
    If Assumption \ref{assumption_global_rhobar} holds, we have the patching set-up as in \S \ref{section_patching}. Since $\widetilde{R^{\psi_k}_\infty(k)}$ is finite free over $S_\infty$ by
    Lemma~\ref{lem_Hecke_algebras_patch_to_deformation_rings},
    it has the same dimension and depth as $S_\infty$.
    As $S_\infty$ is regular, its dimension and depth coincide, and hence $\widetilde{R^{\psi_k}_\infty(k)}$ is Cohen--Macaulay. Since $\widetilde{R^{\psi_k}_\infty(k)}$ is $\Oh$-flat
    (Proposition~\ref{prop_R_inf_tilde_property}), its mod-$\varpi$ reduction
    $$
      \widetilde{R^{\psi_k}_\infty(k)}/\varpi \xrightarrow{\sim}
      \left(R^{\Box,\psi_k}_\rbar(k)[\alpha_p]/\varpi\right) \widehat{\otimes}_{\F} B
      \xrightarrow{\sim} \bigoplus_{u\in U(\rbar)} \left(R_u/\varpi\right)\widehat{\otimes}_{\F} B
    $$
    is also Cohen--Macaulay. In particular, each direct summand $R_u/\varpi \widehat{\otimes}_\F B$ is Cohen--Macaulay. By Remark~\ref{rmk_vexing_prime_local_deformation_ring} and
    Proposition~\ref{prop_completed_tensor_product_in_char_p},
    the ring $B$ is Cohen--Macaulay. It then follows from Proposition \ref{prop_completed_tensor_product_in_char_p} that $R_u/\varpi$ is Cohen--Macaulay. Since $R_\rbar^{\Box,\psi_k}[\alpha_p]$ is $\Oh$-flat and decomposes as
    $$
      R_\rbar^{\Box,\psi_k}(k)[\alpha_p] \xrightarrow{\sim} \bigoplus_{u\in U(\rbar)} R_u
    $$
    by Proposition \ref{prop_R_ap_decomposition}, we deduce that $R_\rbar^{\Box,\psi_k}[\alpha_p]$ is Cohen--Macaulay. For a general determinant $\psi$, the claim follows from Remark \ref{rmk_change_det_of_deformation_ring}. 
 
    If Assumption \ref{assumption_semisimplification} holds, then by Remark \ref{rmk_semisimple_reps_satisfy_assumption_global} the semisimplification $\rbar^{\rm ss}$ of $\rbar$ satisfies Assumption \ref{assumption_global_rhobar}. Thus $R^{\Box,\psi}_{\rbar^{\rm ss}}[\alpha_p]$ is Cohen--Macaulay, and it follows from Proposition \ref{prop_semisimplification_CM} that $R^{\Box,\psi}_\rbar[\alpha_p]$ is Cohen--Macaulay. 
\end{proof}

\begin{prop}
    The homomorphism in Lemma \ref{lem_Hecke_algebras_patch_to_deformation_rings}
$$R^{\psi_k}_\infty(k)^*\to \widetilde{R^{\psi_k}_\infty(k)^*}$$
is an isomorphism if and only if
$$R^{\Box,\psi_k}_\rbar(k)^*\hookrightarrow R^{\Box,\psi_k}_\rbar(k)^*[\alpha_p]$$
is an isomorphism.
\end{prop}

\begin{proof}
    One direction is straightforward from the definition of $R^{\psi_k}_\infty(k)^*$ 
and $\widetilde{R^{\psi_k}_\infty(k)^*}$. For the other direction, it is enough 
to show that
$$R^{\Box,\psi_k}_\rbar(k)^*\hookrightarrow R^{\Box,\psi_k}_\rbar(k)[\alpha_p]^*$$
is surjective. Since both rings are complete $\Oh$-modules, by Nakayama's lemma 
it suffices to show that
$$R^{\Box,\psi_k}_\rbar(k)^*\otimes_\Oh\F\to 
R^{\Box,\psi_k}_\rbar(k)[\alpha_p]^*\otimes_{\Oh}\F$$
is surjective. This follows from the explicit presentation of $B$ given in 
Lemma \ref{rmk_vexing_prime_local_deformation_ring}.
\end{proof}

\begin{prop}\label{prop_R_equal_R_ap}
    Suppose that $\rbar$ satisfies Assumption \ref{assumption_weaker}. 
    Let $k\ge 2$ and assume that $k(\rbar)\neq 1$. Then the inclusion
    $$R^{\Box,\psi}_\rbar(k)\hookrightarrow R^{\Box,\psi}_\rbar(k)[\alpha_p]$$
    is an isomorphism if and only if $k < pk(\rbar)$.
\end{prop}

\begin{proof}
    Suppose that $\rbar$ satisfies Assumption \ref{assumption_global_rhobar}. By Lemma \ref{lem_Hecke_algebras_patch_to_deformation_rings}, the map 
    $$R^{\psi_k}_\infty(k)^*\to \widetilde{R^{\psi_k}_\infty(k)^*}$$ 
    is an isomorphism if and only if $\dim_\F M(j(k),\varnothing,\F)^*=0$. From Remark \ref{rmk_anemic_full_iso_weight_range}, this happens if and only if $k<k(\rbar)^*p$ when $k(\rbar)\neq 1.$ By the proposition above, we then have 
    $$R^{\Box,\psi_k}_\rbar(k)^*\hookrightarrow R^{\Box,\psi_k}_\rbar(k)^*[\alpha_p]$$ 
    is an isomorphism if and only if $k<k(\rbar)^*p$ when $k(\rbar)\neq 1.$ By a standard twisting argument, we can then replace $\psi_k$ with an unramified twist $\psi$ of $\psi_k$. 

    If $\rbar$ satisfies Assumption \ref{assumption_semisimplification}, we may assume it is non-semisimple, and thus $\rbar$ is ramified so that $k(\rbar)>1$ automatically. Its semisimplification satisfies Assumption \ref{assumption_global_rhobar}. 
    When $\chi_2$ is unramified, by Proposition \ref{prop_versal_universal_ap_in_R_equivalent}, Proposition \ref{prop_nonordinary_Serre_weight_local_meaning}, and Definition \ref{def_nonordinary_serre_weight}, the two deformation rings are equal if and only if 
    $$k(\rbar)\le k<k(\rbar)^{\no} = pk(\rbar).$$
    When $\chi_1$ and $\chi_2$ are both ramified, by Proposition \ref{prop_versal_universal_ap_in_R_equivalent} and what we have just proved for $R_\rbarss^{\Box,\psi}(k),$ this is equivalent to 
    $$k(\rbarss)\le k<pk(\rbarss).$$ 
    One checks by Definition \ref{def_Serre_weight} that $k(\rbar) = k(\rbarss).$
    When $\chi_1$ is unramified and $\chi_2$ is ramified, by Proposition \ref{prop_versal_universal_ap_in_R_equivalent} and what we have just proved for $R_\rbarss^{\Box,\psi}(k)^{\no},$ this is equivalent to 
    $$k(\rbarss)^{\no}\le k<pk(\rbarss).$$ 
    By the proof of Proposition \ref{prop_nonordinary_Serre_weight_local_meaning}, we have $k(\rbarss)^{\no}=k(\rbar)$, and the proof is complete.
\end{proof}

\begin{proof}[Proof of Theorem \ref{thm_crystalline_deformation_ring_Cohen_Macaulay}]
    When $k(\rbar)\neq 1,$ this follows directly from Theorem \ref{thm_R_ap_cohen_macaulay} and Proposition \ref{prop_R_equal_R_ap}. When $k(\rbar) = 1,$ we need to show $R^{\Box,\psi}_\rbar(1)$ is Cohen--Macaulay. But in this case, this is exactly the deformation ring parametrizing all the unramified lifts of $\rbar$ with a fixed determinant $\psi,$ which concerns lifting the Frobenius element only. Hence, it is formally smooth and thus Cohen--Macaulay. 
\end{proof}

\subsection{Density of Crystalline Points in Characteristic \texorpdfstring{$p$}{p}}
Let $k\ge2$ be an integer that is congruent to $k(\rbar)$ modulo $p-1$. Then the characters $\psi_{k}$ and $\psi_{k(\rbar)}$ are congruent modulo $\varpi,$ the deformation ring $R^{\Box,\psi_k}_\rbar\otimes_\Oh\F$ can be identified with $R^{\Box,\psi_{k(\rbar)}}_\rbar\otimes_\Oh\F.$ Therefore, the special fiber of the crystalline deformation ring $R^{\Box,\psi_{k}}_\rbar(k)$ can  be viewed as quotients of $R^{\Box,\psi_{k(\rbar)}}_\rbar\otimes_\Oh\F$, cut out by the ideal $$I_k:=\ker(R_\rbar^{\Box,\psi}\otimes_\Oh\F \surj R_\rbar^{\Box,\psi}(k)\otimes_\Oh\F).$$ In the rest of the subsection, we will write all the determinants as $\psi.$ 

We first prove Theorem \ref{thm_liminf_of_crystalline_is_everything} and then deduce 
Theorem \ref{thm_crystalline_stronger_topology} from it. Thus we work with the 
patching setup in \S \ref{section_patching}. 
Moreover, since Theorem \ref{thm_liminf_of_crystalline_is_everything} is asymptotic, 
it suffices to assume that the weights $k$ are large enough for the conditions in 
Lemma \ref{lem_Hecke_algebras_patch_to_deformation_rings} to hold. In the rest of this 
subsection, we will always assume $k \ge p^2$.

Consider the composition of maps $$R_\rbar^{\Box,\psi}\surj R_\rbar^{\Box,\psi}(k)\hookrightarrow R_\rbar^{\Box,\psi}(k)[\alpha_p]\xrightarrow{\sim}R_\rbar^{\Box,\psi}(k)^{\ord}[\alpha_p]\oplus R_\rbar^{\Box,\psi}(k)^{\no}[\alpha_p]\surj R_\rbar^{\Box,\psi}(k)[\alpha_p]^{\no}.$$ Tensor it with $\F$ and we get $$R_\rbar^{\Box,\psi}\otimes_\Oh\F \surj R_\rbar^{\Box,\psi}(k)\otimes_\Oh\F \to R_\rbar^{\Box,\psi}(k)^{\no}[\alpha_p]\otimes_\Oh\F.$$ Let $J_k$ be the kernel of the map. It is clear that $J_k\supseteq I_k.$
Since the map $$R_\rbar^{\Box,\psi}\otimes_\Oh\F\to R_\rbar^{\Box,\psi}(k)^{\no}[\alpha_p]\otimes_\Oh\F$$ factors through $R_\rbar^{\Box,\psi}(k)^{\no}\otimes_\Oh\F,$ we have the identification  $$R_\rbar^{\Box,\psi}\otimes_\Oh\F/J_k\xrightarrow{\sim}\Image(R_\rbar^{\Box,\psi}(k)^{\no}\otimes_\Oh\F\to R_\rbar^{\Box,\psi}(k)^{\no}[\alpha_p]\otimes_\Oh\F).$$
Recall that $$R_\infty^{\psi}\otimes_\Oh\F = (R^{\Box,\psi}_\rbar\otimes_\Oh\F) \widehat{\otimes}_\F B$$ where $B$ is defined in \eqref{eqn_def_B}. Define $J_k^e$ to be the extension of $J_k$ in $R_\infty^{\psi}\otimes_\Oh\F.$ Our idea is to study the ideals $J_k^e$ and transfer the information back to $J_k.$ There are two ways to do this. The inclusion $$R^{\Box,\psi}_\rbar\otimes_\Oh\F \hookrightarrow R^{\psi}_\infty\otimes_\Oh\F$$ has a retraction given by reduction modulo the maximal ideal $\mathfrak{m}_B$ of $B,$ meaning that we have the commutative diagram \begin{equation}\label{eqn_retraction}
\begin{tikzcd}
    R^{\Box,\psi}_\rbar\otimes_\Oh\F \arrow[r,hook]\arrow[rd,"{\rm id}"] & R^{\psi}_\infty\otimes_\Oh\F\arrow[d,"\mod\mathfrak{m}_B"] \\&R^{\Box,\psi}_\rbar\otimes_\Oh\F
\end{tikzcd}.
\end{equation} The same holds true for $\Image\left(R_\rbar^{\Box,\psi}(k)^{\no}\otimes_\Oh\F \to R_\rbar^{\Box,\psi}(k)^{\no}[\alpha_p]\otimes_\Oh\F\right)$ 
and $$\Image\left(R_\rbar^{\Box,\psi}(k)^{\no}\otimes_\Oh\F \to R_\rbar^{\Box,\psi}(k)^{\no}[\alpha_p]\otimes_\Oh\F\right)\widehat{\otimes}_\F B,$$ i.e., there are commutative diagrams 
\begin{equation}\label{eqn_retraction_k}\begin{tikzcd} 
    \Image(R^{\Box,\psi}_\rbar(k)^{\no}\otimes_\Oh\F \to R^{\Box,\psi}_\rbar(k)^{\no}[\alpha_p]\otimes_\Oh\F)\arrow[r,hook]\arrow[rd,"{\rm id}"] & \Image(R^{\Box,\psi}_\rbar(k)^{\no}\otimes_\Oh\F \to R^{\Box,\psi}_\rbar(k)^{\no}[\alpha_p]\otimes_\Oh\F) \widehat{\otimes}_\F B\arrow[d,"\mod\mathfrak{m}_B"] \\&\Image(R^{\Box,\psi}_\rbar(k)^{\no}\otimes_\Oh\F \to R^{\Box,\psi}_\rbar(k)^{\no}[\alpha_p]\otimes_\Oh\F)
\end{tikzcd}\end{equation} for every weight $k\ge p^2.$ On the one hand, we have $J_k^{ec} = J_k$. To see this, by Proposition \ref{prop_completed_tensor_product_ideal}, there is a short exact sequence $$0\to J_k^e\to R_\infty^{\psi}\otimes_\Oh\F\to (R^{\Box,\psi}_\rbar/J_k + (\varpi))\widehat{\otimes}_\F B\to 0.$$ In other words, the ideal $J_k^e$ is the kernel of 
$$R_\infty^{\psi}\otimes_\Oh\F\surj (R^{\Box,\psi}_\rbar/J_k + (\varpi))\widehat{\otimes}_\F B\xrightarrow{\sim}\Image(R^{\Box,\psi}_\rbar(k)^{\no}\otimes_\Oh\F \to R^{\Box,\psi}_\rbar(k)^{\no}[\alpha_p]\otimes_\Oh\F)\widehat{\otimes}_\F B.$$
Now $J_k^{ec} = J_k$ follows from the fact that the horizontal map in the diagram \eqref{eqn_retraction_k} is injective. On the other hand, we have $(J_k^e+\mathfrak{m}_B \mod \mathfrak{m}_B) \xrightarrow{\sim} J_k.$ This is because \begin{align*}
    \frac{R^{\Box,\psi}_\rbar\otimes_\Oh\F}{J_k^e \mod \mathfrak{m}_B}\xrightarrow{\sim}R^{\psi}_\infty\otimes_\Oh\F/(J_k^e+\mathfrak{m}_B) &\xrightarrow{\sim}\frac{R^{\psi}_\infty\otimes_\Oh\F/J_k^e}{(J_k^e+\mathfrak{m}_B)/J_k^e}\\&\xrightarrow{\sim} \frac{\Image(R^{\Box,\psi}_\rbar(k)^{\no}\otimes_\Oh\F \to R^{\Box,\psi}_\rbar(k)^{\no}[\alpha_p]\otimes_\Oh\F)\widehat{\otimes}_\F B}{\mathfrak{m}_B} \\&\xrightarrow{\sim} \Image(R^{\Box,\psi}_\rbar(k)^{\no}\otimes_\Oh\F \to R^{\Box,\psi}_\rbar(k)^{\no}[\alpha_p]\otimes_\Oh\F).
\end{align*} As a result, we may regard $J_k$ either as the restriction of $J_k^e$ or as its reduction modulo~$\mathfrak{m}_B$. Each point of view has its own advantage, and we will adopt one or the other depending on our needs.

\begin{prop} \label{prop_J_k_e_defining_ideal}
    The ideal $J_k^e$ is the kernel of $$R_\infty^\psi\otimes_\Oh\F\surj \Image(R_\infty^{\psi}(k)^{\no}\otimes_\Oh\F\to \widetilde{R_\infty^{\psi}(k)^{\no}}\otimes_\Oh\F)$$
\end{prop}

\begin{proof}
    By the universal property of completed tensor product, there is a unique map $$\Image(R^{\Box,\psi}_\rbar(k)^\no\otimes_\Oh\F \to R^{\Box,\psi}_\rbar(k)^\no[\alpha_p]\otimes_\Oh\F)\widehat{\otimes}_\F B\to \Image(R_\infty^\psi(k)^\no\otimes_\Oh\F\to \widetilde{R_\infty^\psi(k)^\no}\otimes_\Oh\F)$$ which fits into the commutative diagram $$\begin{tikzcd}
        R_\infty^\psi\otimes_\Oh\F \arrow[r,two heads]\arrow[dr,two heads] & \Image(R^{\Box,\psi}_\rbar(k)^\no\otimes_\Oh\F \to R^{\Box,\psi}_\rbar(k)^\no[\alpha_p]\otimes_\Oh\F)\widehat{\otimes}_\F B  \arrow[r]\arrow[d]& \widetilde{R_\infty^{\psi}(k)^\no}\otimes_\Oh\F\\
         & \Image(R_\infty^\psi(k)^\no\otimes_\Oh\F\to \widetilde{R_\infty^\psi(k)^\no}\otimes_\Oh\F)\arrow[ur,hook] & 
    \end{tikzcd}.$$ The vertical map is therefore automatically surjective.  As explained above,
our goal is to show that it is an isomorphism.  To prove injectivity, it
suffices to check that the composite $$\Image(R^{\Box,\psi}_\rbar(k)^\no\otimes_\Oh\F \to R^{\Box,\psi}_\rbar(k)^\no[\alpha_p]\otimes_\Oh\F)\widehat{\otimes}_\F B  \to \widetilde{R_\infty^{\psi}(k)^\no}\otimes_\Oh\F$$ is an injection. By definition, the ring $\widetilde{R_\infty^{\psi}(k)^\no}\otimes_\Oh\F$ is $$(R^{\Box,\psi}_\rbar(k)^\no[\alpha_p]\otimes_\Oh\F)\widehat{\otimes}_\F B.$$ Now the injectivity follows from Proposition \ref{prop_completed_tensor_product_in_char_p}.
\end{proof}

From now on, by $\varprojlim_k$ we mean taking the inverse limit over positive integers $k\ge p^2$ that are congruent to $k(\rbar)$ modulo $p-1.$ Consider the ring $$\varprojlim_k \Image(R_\infty^{\psi}(k)^\no\otimes_\Oh\F\to \widetilde{R_\infty^{\psi}(k)^\no}\otimes_\Oh\F),$$ where the transition maps are the surjections constructed in Lemma \ref{lem_Hecke_algebras_patch_to_deformation_rings}. So this is a surjective inverse system of finite free $S_\infty\otimes_\Oh\F$ modules. In particular, both Lemma \ref{lem_inverse_limit_exactness} and Lemma \ref{lem_inverse_limit_tensor_product} apply. By Proposition \ref{prop_J_k_e_defining_ideal}, we can rewrite this ring as $$\varprojlim_k R_\infty^{\psi}\otimes_\Oh\F/J_k^e.$$ The ring homomorphism $$R_\infty^{\psi}\otimes_\Oh\F\to \varprojlim_k R_\infty^{\psi}\otimes_\Oh\F/J_k^e$$ has kernel $$\mathfrak{a} := \bigcap_{k\equiv k(\rbar)\pmod{p-1}}J_k^e.$$

\begin{prop}\label{prop_surjection_R_inf_to_inverse_limit}
    The ring homomorphism $$R_\infty^{\psi}\otimes_\Oh\F\to \varprojlim_k R_\infty^{\psi}\otimes_\Oh\F/J_k^e$$ is surjective. In particular, the ring $\varprojlim_k R_\infty^{\psi}\otimes_\Oh\F/J_k^e$ is Noetherian. 
\end{prop}

\begin{proof}
    Modulo the maximal ideal of $S_\infty,$ the map gives rise to a ring homomorphism $$R_\varnothing^{\univ,\psi}\otimes_\Oh\F \to (\varprojlim_k R_\infty^{\psi}\otimes_\Oh\F/J_k^e)\otimes_{S_\infty/\varpi}S_\infty/\mathfrak{m}_{S_\infty}\xrightarrow{\sim}\varprojlim_k R_\infty^{\psi}\otimes_\Oh\F/J_k^e\otimes_{S_\infty/\varpi}S_\infty/\mathfrak{m}_{S_\infty},$$ where the isomorphism follows from Lemma \ref{lem_inverse_limit_tensor_product}. By Lemma \ref{lem_quotient_of_patched_modules} and Lemma \ref{lem_U_ell_in_anemic}, we may identify the right-hand term with $$\varprojlim_k \T^{pN_\varnothing}(M(k,\varnothing,\F)^\no).$$ By Deo's theorem below, the homomorphism is surjective modulo $\mathfrak{m}_{S_\infty}$. Once we can show that $\varprojlim_k R_\infty^{\psi}\otimes_\Oh\F/J_k^e$ is a complete $S_\infty$-module, then the proposition follows from Nakayama's lemma for complete modules. To see this, we complete $\varprojlim_k R_\infty^{\psi}\otimes_\Oh\F/J_k^e$: $$\varprojlim_n\left((\varprojlim_k R_\infty^{\psi}\otimes_\Oh\F/J_k^e)\otimes_{S_\infty/\varpi}S_{\infty}/\mathfrak{m}_{S_\infty}^n\right).$$ Again by Lemma \ref{lem_inverse_limit_tensor_product} and commutativity of inverse limits, this is isomorphic to $$\varprojlim_n\varprojlim_k (R_\infty^\psi\otimes_\Oh\F/J_k^e\otimes_{S_\infty/\varpi}S_\infty/\mathfrak{m}_{S_\infty}^n) \xrightarrow{\sim} \varprojlim_k\varprojlim_n (R_\infty^\psi\otimes_\Oh\F/J_k^e\otimes_{S_\infty/\varpi}S_\infty/\mathfrak{m}_{S_\infty}^n).$$ 
    For each $k$, we have $$R_\infty^{\psi}\otimes_\Oh\F/J_k^e\xrightarrow{\sim} \Image(R_\infty^{\psi}(k)^\no\otimes_\Oh\F\to \widetilde{R_\infty^{\psi}(k)^\no}\otimes_\Oh\F),$$ which is a quotient of $R_\infty^\psi(k)\otimes_\Oh\F.$ Since $R_\infty^{\psi}(k)\otimes_\Oh\F$ is finite over $S_\infty/\varpi$ by the proof of Proposition \ref{prop_patching_the_coherent_cohomology}, $R_\infty^{\psi}\otimes_\Oh\F/J_k^e$ is a finite $S_\infty$-module, in
particular complete. Therefore, the double limit above reduces to $\varprojlim_k R_\infty^{\psi}\otimes_\Oh\F/J_k^e$, showing that this is a complete $S_\infty$-module. 
\end{proof}

\begin{thm}[{\cite[Theorem 1]{deo_2017_structure_of_hecke_algebras}}]\label{thm_Deo}
    The ring homomorphism $$R_\varnothing^{\univ,\psi}\otimes_\Oh\F\to \varprojlim_k \T^{pN_\varnothing}(M(k,\varnothing,\F)^\no)$$ is a surjection. Assume that $p\ge 5.$ Then both rings have Krull dimension at least $2. $
\end{thm}

\begin{remark}
This is a slightly modified version of Deo's theorem. His original statement assumes that $p\nmid \phi(N)$ and is stated for the Hecke algebra $\T^{pN_\varnothing}(M(k,\varnothing,\F))$. 

In our setup \S \ref{subsection_hecke_modules_and_algebras}, we work with level structure $\Gamma_1$ equipped with a fixed Nebentypus character $\chi$ whose order is prime to $p$. Under these assumptions, the pseudo-deformation map constructed by Deo in \cite[Lemma 4]{deo_2017_structure_of_hecke_algebras} still has constant determinant. Consequently, his proof of \cite[Lemma 5]{deo_2017_structure_of_hecke_algebras} applies in our situation without the extra assumption $p\nmid \phi(N)$.

Moreover, since $\T^p(M(k,\varnothing,\F))$ is a finite $\F$-algebra, the Mittag--Leffler condition is satisfied, giving a surjection
\[
\varprojlim_k \T^p(M(k,\varnothing,\F)) \;\;\surj\;\; \varprojlim_k \T^p(M(k,\varnothing,\F)^\no).
\]
Hence, the surjection in Theorem~\ref{thm_Deo} follows from Deo's original result. One can also check that all of Deo's arguments remain valid for the quotient ring $\varprojlim_k \T^p(M(k,\varnothing,\F)^\no)$, since the crucial step \cite[Proposition 17 (1)]{deo_2017_structure_of_hecke_algebras} is established for this quotient by Jochnowitz \cite[Theorem 6.3]{jochnowitz1982study}.
\end{remark} 

Recall from \S \ref{subsection_setup_for_patching} that $S_\infty = \Oh\llbracket X_{v,i,j},y_1,\ldots,y_r \rrbracket_{v\in T(\rhobar)\sqcup\{p\},i,j=1,2,vij\neq p}.$ Thus $S_\infty$ is a formal power series ring in $(\dim S_\infty - 1)$ variables where $$\dim S_\infty - 1 = 4\# (T(\rhobar)\sqcup\{p\}) - 1+ r = 4\#T(\rhobar) + r + 3.$$ Let $d_1,\ldots,d_{\dim S_\infty-1}$ denote the image of these variables in $S_\infty\otimes_\Oh\F.$

\begin{prop}\label{prop_regular_sequence_in_inverse_limit}
    The sequence $d_1,\ldots,d_{\dim S_\infty-1}$ is a regular sequence in $\varprojlim_k R_\infty^{\psi}\otimes_\Oh\F/J_k^e.$
\end{prop}

\begin{proof}
    We begin by showing that $\varprojlim_k R_\infty^{\psi}\otimes_\Oh\F/J_k^e$ is $d_1$-torsion free. For each $k,$ the module $R_\infty^{\psi}\otimes_\Oh\F/J_k^e = \Image(R_\infty^{\psi}(k)\otimes_\Oh\F\to \widetilde{R_\infty^{\psi}(k)}\otimes_\Oh\F)$ is finite free over $S_\infty\otimes_\Oh\F$ by Lemma \ref{lem_Hecke_algebras_patch_to_deformation_rings}. Hence, multiplication by $d_1$ gives a short exact sequence $$0\to R_\infty^{\psi}\otimes_\Oh\F/J_k^e \xrightarrow{\cdot d_1} R_\infty^{\psi}\otimes_\Oh\F/J_k^e \to R_\infty^{\psi}\otimes_\Oh\F/(J_k^e+(d_1))\to 0.$$ The system $\{R_\infty^{\psi}\otimes_\Oh\F/J_k^e\}_{k\equiv k(\rbar)\pmod{p-1}}$ is a surjective system of finite free $S_\infty\otimes_\Oh\F$-modules, 
    so Lemma \ref{lem_inverse_limit_exactness} implies that the inverse limit preserves this exact sequence. Thus the inverse limit $\varprojlim_k R_\infty^{\psi}\otimes_\Oh\F/J_k^e$ is $d_1$-torsion free and we obtain an isomorphism $$(\varprojlim_k R_\infty^{\psi}\otimes_\Oh\F/J_k^e)/(d_1)\xrightarrow{\sim}\varprojlim_k R_\infty^{\psi}\otimes_\Oh\F/(J_k^e+(d_1)).$$ Now repeat the argument with the modules $\varprojlim_k R_\infty^{\psi}\otimes_\Oh\F/(J_k^e+(d_1)),$ using that each is finite free over $S_\infty\otimes_\Oh\F/(d_1).$ This shows that the quotient of the inverse limit by $d_1$ is $d_2$-torsion free. Iterating the procedure for $d_3,\ldots,d_{\dim S_\infty-1}$ shows inductively that the sequence $d_1,\ldots,d_{\dim S_\infty - 1}$ is regular on $\varprojlim_k R_\infty^{\psi}\otimes_\Oh\F/J_k^e$ as desired. 
\end{proof}

\begin{lem} \label{lem_intersection_J_k_extension_nilpotent}
    Assume that $p\ge 5.$ The ideal $$\mathfrak{a} = \ker\left(R_\infty^{\psi}\otimes_\Oh\F \surj  \varprojlim_k R_\infty^{\psi}\otimes_\Oh\F/J_k^e\right)$$
    is contained in the nilradical of $R_\infty^{\psi}\otimes_\Oh\F.$
\end{lem}

\begin{proof}
    By Proposition \ref{prop_surjection_R_inf_to_inverse_limit} and its proof, we have a commutative diagram $$\begin{tikzcd}
        R_\infty^{\psi}\otimes_\Oh\F \arrow[r,two heads]\arrow[d,two heads]& \varprojlim_k R_\infty^{\psi}\otimes_\Oh\F/J_k^e\arrow[d,two heads]\\
        R_\varnothing^{\univ,\psi}\otimes_\Oh\F \arrow[r,two heads] & \varprojlim_k \T^{pN_\varnothing}(M(k,\varnothing,\F))
    \end{tikzcd}.$$ The vertical maps are reductions modulo the sequence $d_1,\ldots,d_{\dim S_\infty -1}.$ Applying Proposition \ref{prop_regular_sequence_in_inverse_limit}, Theorem \ref{thm_Deo} and Krull's principal ideal theorem, we conclude that the ring $\varprojlim_k R_\infty^{\psi}\otimes_\Oh\F/J_k^e$ has Krull dimension at least $$ 2 + \dim S_\infty - 1 = 2 + 4\#T(\rhobar) + r + 3 = 4\#T(\rhobar) + r + 5.$$ By Proposition \ref{prop_R_infty_special_fiber}, this equals the dimension of $R_\infty^{\psi}\otimes_\Oh\F$. Hence, the kernel of the top horizontal surjection cannot contain any non-zero-divisor by Krull's principal ideal theorem. Proposition \ref{prop_R_infty_special_fiber} implies that all zero-divisors in $R_\infty^{\psi}\otimes_\Oh\F$ are nilpotent, so the kernel $\mathfrak{a}$ is contained in its nilradical.
\end{proof}

\begin{proof}[Proof of Theorem \ref{thm_liminf_of_crystalline_is_everything}]
    Since $J_k^e \supseteq J_{k + p -1}^{e}$ for every $k\ge 2$, we have $$J_k^{ec} \supseteq J_{k + p -1}^{ec}{\rm \quad and\quad}(J_k^e \mod \mathfrak{m}_B) \supseteq (J_{k + p -1}^{e} \mod\mathfrak{m}_B).$$ By the discussion above Proposition \ref{prop_J_k_e_defining_ideal}, both give us $J_k\supseteq J_{k+p-1}$ for every $k\ge 2.$ The surjection $$\Image(R_\rbar^{\Box,\psi}(k + p-1)\otimes_\Oh\F \to R_\rbar^{\Box,\psi}(k + p-1)[\alpha_p]\otimes_\Oh\F)\surj \Image(R_\rbar^{\Box,\psi}(k )\otimes_\Oh\F \to R_\rbar^{\Box,\psi}(k )[\alpha_p]\otimes_\Oh\F)$$ is induced from the inclusion of defining ideals $J_k\supseteq J_{k+p-1}.$ 
    
    For the second part, we use the two perspectives on $J_k$ discussed before Proposition \ref{prop_J_k_e_defining_ideal}. First, the diagram \eqref{eqn_retraction} factors through $\mathfrak{a}$, i.e., we have a commutative diagram $$\begin{tikzcd}
    R^{\Box,\psi}_\rbar\otimes_\Oh\F \arrow[r,hook]\arrow[rd,"{\rm id}"] & R^{\psi}_\infty\otimes_\Oh\F/\mathfrak{a}\arrow[d,"\mod\mathfrak{m}_B",two heads] \\&R^{\Box,\psi}_\rbar\otimes_\Oh\F
\end{tikzcd}.$$ To see this, since $R_\rbar^{\Box,\psi}\otimes_\Oh\F$ is an integral domain by Theorem \ref{thm_unrestricted_deformation_ring_at_p}, the ideal $\mathfrak{m}_B$ is a prime ideal in $R_\infty^{\psi}\otimes_\Oh\F$. By Lemma \ref{lem_intersection_J_k_extension_nilpotent}, the ideal $\mathfrak{a}$ is contained in the nilradical of $R_\infty^{\psi}\otimes_\Oh\F$. Therefore, its contraction $\mathfrak{a}^c$ in $R_\rbar^{\Box,\psi}\otimes_\Oh\F$ is zero and $\mathfrak{a}$ is contained in $\mathfrak{m}_B$. This proves that the horizontal map is injective and the vertical map factors through the quotient by $\mathfrak{a}$. 
We can rewrite the diagrams in \eqref{eqn_retraction_k} using the defining ideals and take the inverse limit with the surjective transition maps to construct the following commutative diagram  
$$\begin{tikzcd}
    \varprojlim_k R^{\Box,\psi}_\rbar\otimes_\Oh\F/J_k \arrow[r,hook]\arrow[rd,"{\rm id}"]& \varprojlim_k R_\infty^{\psi}\otimes_\Oh\F/J_k^e \arrow[d,two heads]\\
     & \varprojlim_k R^{\Box,\psi}_\rbar\otimes_\Oh\F/J_k
\end{tikzcd}.$$ The horizontal map is injective and the vertical map is surjective because they compose to the identity map. 
We can piece these commutative diagrams together and get: $$\begin{tikzcd}
        R_\rbar^{\Box,\psi}\otimes_\Oh\F \arrow[r]\arrow[d,hook]\arrow[d,hook]\arrow[dd,"{\rm id}",bend right = 70] &\varprojlim_k R^{\Box,\psi}_\rbar\otimes_\Oh\F/J_k \arrow[d,hook]\arrow[dd,"{\rm id}",bend left = 70]\\
    R_\infty^{\psi}\otimes_\Oh\F/\mathfrak{a} \arrow[r,"\sim"]\arrow[d,two heads,"\mod \mathfrak{m}_B"]& \varprojlim_k R_\infty^{\psi}\otimes_\Oh\F/J_k^e\arrow[d,two heads]\\
    R_\rbar^{\Box,\psi}\otimes_\Oh\F \arrow[r]&\varprojlim_k R^{\Box,\psi}_\rbar\otimes_\Oh\F/J_k
    \end{tikzcd}.$$ Now from the top square, we see the top horizontal map is injective and from the bottom square, we see the bottom horizontal map is surjective. However, the top and bottom horizontal maps are the same map, meaning that they are both isomorphisms. 
\end{proof}

\begin{proof}[Proof of Theorem \ref{thm_crystalline_stronger_topology}]
    If $\rbar$ is semisimple, the result follows from Theorem \ref{thm_liminf_of_crystalline_is_everything} 
    and Proposition \ref{prop_strong_top_inverse_limit_equivalence}. 
    Suppose that $\rbar$ is not semisimple and satisfies Assumption \ref{assumption_semisimplification}. 
    In the proof of Theorem \ref{thm_liminf_of_crystalline_is_everything}, the ideal
    \[
        J_k = \ker\bigl(R_\rbarss^{\Box,\psi}\otimes_\Oh\F \to R_\rbarss^{\Box,\psi}(k)^\no[\alpha_p]\otimes_\Oh\F\bigr)
    \]
    is used. Hence, the condition in Proposition 
    \ref{prop_thm_crystalline_stronger_topology_reduction_to_semisimple_case} is satisfied, 
    and the theorem follows.
\end{proof}

\begin{remark} \label{rmk_speed_of_convergence}
    Our method does not shed much light on how fast the Hilbert functions converge in Theorem \ref{thm_Hilbert_function} even for the representations $\rbar$ satisfying Assumption \ref{assumption_global_rhobar}. For one thing, the argument of our proof essentially boils down to a dimension count of Noetherian rings. For the other, our method studies $\Image(R^{\Box,\psi}_\rbar(k)\otimes_\Oh\F\to R^{\Box,\psi}_\rbar(k)^\no[\alpha_p]\otimes_\Oh\F)$ by patching $\{\T^{p}(M(k,Q_n,\F)^\no)^{\Box}\}_n.$ But we are actually interested in $R^{\Box,\psi}_\rbar(k)\otimes_\Oh\F$ which is the patched ring for $\{\T^{p}(M(k,Q_n,\Oh))^{\Box}\otimes_\Oh\F\}_n.$ And Corollary \ref{cor_anemic_equal_full_criterion} together with its remark state that the two rings are different as long as $k\ge k(\rbar)^\no p$ and moreover that the difference grows with $k$!
    Thus, even an effective proof of Theorem \ref{thm_liminf_of_crystalline_is_everything} would leave open the further problem of relating the resulting bounds on $k$ to those in Theorem \ref{thm_Hilbert_function}. At present, the speed of convergence of the Hilbert functions remains unclear to the author. Nonetheless, numerical computations in the author's thesis suggest that the behavior may have a particularly interesting pattern, and the author hopes to return to this question and formulate a more concrete conjecture in future work.
\end{remark}

\subsection{A Multiplicity-One Theorem for Serre's Modular Forms} \label{subsection_Serre_mult_one}
For a global Galois representation $\rhobar:G_\Q\to \GL_2(\F),$  
we denote by $\rhobar(i)$ the Tate twist of $\rhobar$ by $\epsilon^i.$ 

\begin{thm} \label{thm_Serre_mult_one}
    Let $$\rhobar:G_{\Q}\to \GL_2(\F)$$ be a global Galois representation. Assume the following:
\begin{enumerate}
    \item $p \ge 5$;
    \item $\rhobar$ is odd;
    \item $\rhobar|_{\Gal(\Q(\zeta_p)/\Q)}$ is absolutely irreducible;
    \item $N(\rhobar) \ge 5$;
    \item $\rbar:=\rhobar|_{G_{\Q_p}}$ is not a twist of 
    \(
        \begin{pmatrix}
            \epsilon & * \\
            0 & 1
        \end{pmatrix}
    \)
    with $*$ trivial or peu ramifi\'ee.
\end{enumerate}
Then for every weight $k\ge 1$, the $\F$-vector space
\[
    S(k,\varnothing)
    \bigl[\mathfrak{m}_\varnothing^{\{pN_\varnothing\}}\bigr]
\]
is at most one-dimensional.  
The same conclusion holds when $\rbar$ is a twist of the above shape,
provided that there are no vexing primes.
\end{thm}

\begin{proof}
    First of all, by Lemma \ref{lem_U_ell_in_anemic}, we can replace $\mathfrak{m}_\varnothing^{\{pN_\varnothing\}}$ with $$\mathfrak{m}_\varnothing^{\{p\}} := (T_\ell- a_\ell(f_0),\langle\ell\rangle-\chi(\ell),U_q - a_q(f_0),U_{q_0} - \alpha_{q_0})_{\ell\nmid pN_\varnothing,q|N(\rhobar)}.$$ By twisting (Proposition \ref{prop_Sk_theta_twist}), it suffices to prove the statement for $$S(k,U_1(N_\varnothing),\chi)_{\mathfrak{m}_\varnothing^{\{pN(\rhobar)\}}(i)}[\mathfrak{m}_\varnothing^{\{p\}}(i)]$$ for all weights $p+1\le k \le 2p+1$ and all Tate twists $\rhobar(i)$ with $0\le i\le p-2.$ Replacing $\rhobar(i)$ by $\rhobar,$ we can drop the twist from the notation. Within this range of weights, we have the identification $$S(k,\varnothing)\xrightarrow{\sim} M(k,\varnothing,\F)/M(k-(p-1),\varnothing,\F).$$ 
    When $k = 2p,$ by Proposition \ref{prop_Serre_mod_forms_k_pk}, we can replace $k = 2p$ by $k = 2.$ Then by Proposition \ref{prop_Sk_theta_twist}, we can replace $k = 2$ by $k = p+3\le 2p-1$ (since $p\ge 5$). Similarly, for $k=2p+1,$ we first apply Proposition \ref{prop_Sk_theta_twist} and then Proposition \ref{prop_Serre_mod_forms_k_pk} to replace $k=2p+1$ by $k=p+2.$  
    Hence, we only need to prove the statement for $S(k,\varnothing)[\mathfrak{m}_\varnothing^{\{p\}}] 
    $ for all weights $$p+1\le k\le 2p-1.$$ 

    If $M(k - (p-1),\varnothing,\F) = 0,$ then we have \begin{align*}
    S(k,\varnothing)[\mathfrak{m}_\varnothing^{\{p\}}] &= M(k,\varnothing,\F)[\mathfrak{m}_\varnothing^{\{p\}}].
    \end{align*} Suppose that $M(k,\varnothing,\F)[\mathfrak{m}_\varnothing^{\{p\}}]$
    is at least two-dimensional, then there is a modular form $g$ in $M(k,\varnothing,\F)$ whose $q$-expansion has the shape $\sum_{n\ge 1}b_nq^{np}$. By Theorem \ref{thm_katz_theta},  such a form lies in $\ker(\theta) = \Image V_p$. However, since $V_p$ is $\T^{p}$-equivariant and raises the filtration of a form to its $p$-th power, we have $$\Image V_p \cap M(k,\varnothing,\F) = 0.$$ Thus the $\F$-dimension of $S(k,\varnothing)[\mathfrak{m}_\varnothing^{\{p\}}]$ is at most one, as desired. 

    If $M(k-(p-1),\varnothing,\F)\neq 0,$ then the weight $k$ satisfies $$1\le k(\rbar) \le k - (p-1) \le p.$$ We will use Galois deformation theory to conclude the proof. 
    
    Suppose that $\rbar$ is absolutely irreducible, i.e. case (1) in the definition of the Serre weight, then $\rbar|_{I_p}$ takes the form $$\epsilon_2^{k(\rbar)-1}\oplus \epsilon_2^{p(k(\rbar)-1)}.$$ In this case, we have $2\le k(\rbar)\le p$ and $k = k(\rbar) + p-1.$ By Example \ref{example_surjection_of_patched_rings}, the patched module $\mathcal{P}\left(\left\{S(k(\rbar)+p-1,Q_n)^{\vee,\Box} \right\}_{n\ge 1}\right)$ is either trivial or generated by one element over the ring $$\widetilde{R_\infty^{\psi_{k(\rbar)+p-1}}(k(\rbar)+p-1)}\otimes_\Oh\F=R_\infty^{\psi_{k(\rbar)+p-1}}(k(\rbar)+p-1)\otimes_\Oh\F.$$ Modulo $\mathfrak{m}_{S_\infty}$ and invoke Lemma \ref{lem_quotient_of_patched_modules}, we see that $S(k,\varnothing)^{\vee}$ is generated by at most one element over $$\T(M(k,\varnothing,\F)) = \T^{p}(M(k,\varnothing,\F)).$$ By Nakayama's lemma, it follows that $$\dim_\F S(k,\varnothing)^{\vee}/\mathfrak{m}_\varnothing^{\{p\}}S(k,\varnothing)^{\vee} = 1.$$ Thus we have $$\dim_\F S(k,\varnothing)[\mathfrak{m}_\varnothing^{\{p\}}] = \dim_\F\left(S(k,\varnothing)[\mathfrak{m}_\varnothing^{\{p\}}]\right)^\vee = \dim_\F S(k,\varnothing)^{\vee}/\mathfrak{m}_\varnothing^{\{p\}}S(k,\varnothing)^{\vee} = 1.$$ 

    Suppose that $\rbar$ is reducible and tamely ramified, i.e. case (2)(a) in the definition of the Serre weight, then $\rbar|_{I_p}$ takes the form $\epsilon^{k(\rbar)-1} \oplus 1.$ In this case, the weights satisfy $1\le k(\rbar)\le p-1$ and $$k = \begin{cases}
        k(\rbar) + p-1 & \text{if } k(\rbar) \neq 1\\
        2p-1 & \text{if } k(\rbar) = 1.
    \end{cases}$$ We have \begin{align*}
        \dim_\F S(k,\varnothing) &= \dim_\F M(k,\varnothing,\F) - \dim_\F M(k-(p-1),\varnothing,\F).
    \end{align*} By the Eichler--Shimura isomorphism and the fact that torsion in $H^1$ is away from a non-Eisenstein maximal ideal (for details see \cite[\S6.1]{BergdallPollack2019slopes}), we have \begin{multline*}
        2\dim_\F S(k,\varnothing) =  \dim_\F H^1(\Gamma_1(N_\varnothing),\chi\otimes\Sym^{k-2}\F^2)_{\mathfrak{m}_\varnothing^{\{pN(\rhobar)\}}} \\- \dim_\F H^1(\Gamma_1(N_\varnothing),\chi\otimes\Sym^{k-(p-1)-2}\F^2)_{\mathfrak{m}_\varnothing^{\{pN(\rhobar)\}}}.
    \end{multline*} By Serre's formula of the semisimplification of the $\GL_2(\F_p)$ representations (see for example \cite[\S 2.3]{reduzzi2015weight}), the semisimplification of $\Sym^{k-2}\F^2$ is $$\det\otimes \Sym^{k-p-3}\F^2 \oplus \Sym^{k-(p+1)}\F^2 \oplus {\det}^{k-2}\otimes \Sym^{2p-k}\F^2.$$
    However, when $k(\rbar)\neq 2,$ the Serre weights of $\rbar$ (in the sense of \cite{buzzard2010serre}) are $$W(\rbar) := \begin{cases}
    \{\Sym^{p-2}\F^2\} & \text{if } k(\rbar) =1\\
    \{ \Sym^{k(\rbar)-2}\F^2, {\det}^{k(\rbar)-1}\otimes\Sym^{p-1-k(\rbar)}\F^2 \} & \text{if } 3 \le k(\rbar) \le p-2\\
    \{\Sym^{p-3}\F^2,\det^{p-2}\otimes\Sym^{p-1}\F^2, \det^{p-2}\} & \text{if }k(\rbar) = p-1
    \end{cases}.$$ Comparing this list with the decomposition above, we see that among the three Jordan--H\"older factors of $\Sym^{k-2}\F^2,$ only $\Sym^{k-(p+1)}\F^2$ can occur as a Serre weight of $\rbar.$ So we have $$\dim_\F H^1(\Gamma_1(N_\varnothing),\chi\otimes\Sym^{k-2}\F^2)_{\mathfrak{m}_\varnothing^{\{pN(\rhobar)\}}} = \dim_\F H^1(\Gamma_1(N_\varnothing),\chi\otimes\Sym^{k-(p+1)}\F^2)_{\mathfrak{m}_\varnothing^{\{pN(\rhobar)\}}}.$$ and thus $\dim_\F S(k,\varnothing) = 0.$
    
    Suppose that $\rbar$ is reducible and wildly ramified, i.e., case (2)(b) in the definition of the Serre weight. Then $\rbar|_{I_p}$ takes the form $$\begin{pmatrix}
        \epsilon^{k(\rbar)-1}  & *\\
        0 & 1
    \end{pmatrix}.$$ In this case, we have $2\le k(\rbar) \le p$ and $k = k(\rbar) + p-1.$ (In particular, $*$ is not tr\`es ramifi\'ee.) When $k(\rbar)\neq 2$, the Serre weights of $\rbar$ are $$W(\rbar) =
        \{\Sym^{k(\rbar)-2}\}.$$ The same reasoning as above shows that $\dim_\F S(k,\varnothing) = 0.$

    At last, we need to treat the remaining case: the residual representation $\rbar$ satisfies $$\rbar|_{I_p} \sim \begin{pmatrix}
        \epsilon & *\\
        0 & 1
    \end{pmatrix},$$ where $*$ is not tr\`es ramifi\'ee. We have $k = p+1$ and $k(\rbar) = 2.$ If $\rbar$ is not up to twist an extension of $1$ by $\epsilon,$ then the same argument using Example \ref{example_surjection_of_patched_rings} applies. For the rest of the proof, by replacing $\rbar$ with a twist by an unramified character, we assume that $$\rbar\sim \begin{pmatrix}
    \epsilon & *\\
    0 & 1
    \end{pmatrix},$$ where $*$ is peu ramifi\'ee or trivial. 
    In this case, if there are no vexing primes, then by \cite[Lemma 3.3]{Gee_Savitt_2011_quaternion_Serre_weights} and the construction of patching, the patched module $\mathcal{P}\left(\left\{S(k,Q_n)^{\vee,\Box} \right\}_{n\ge 1}\right)$ is a maximal Cohen--Macaulay module over a formal power series ring over the special fiber of the non-crystalline component of the semistable deformation ring $R_{\rbar,{\rm st}}^{\Box,\psi_k}(k(\rbar))$ whose $\overline{\Q}_p$-points parametrize semistable lifts of $\rbar$ with Hodge--Tate weights $\{0,1\},$ 
    and determinant $\psi_k$. One can use the Breuil--M\'ezard conjecture in this case, proved in \cite{Paskunas_2015_BM_conj} if $\rbar$ non-split and in \cite{HuTan2013TheBC} if $\rbar$ splits to conclude that the non-crystalline component is a regular local ring. Following \cite{Diamond_1997_mult_one} and \cite{Fujiwara2006DeformationRA}, one then applies Auslander--Buchsbaum theorem to conclude that $\mathcal{P}\left(\left\{S(k,Q_n)^{\vee,\Box} \right\}_{n\ge 1}\right)$ is free over the formal power series ring over the special fiber of the non-crystalline component of $R_{\rbar,\rm st}^{\Box,\psi_k}(k(\rbar)).$ By interpreting Serre's modular forms as certain invariant functions from $D^{\times}$ to $\F$ where $D$ is a quaternion algebra over $\Q$ that is ramified exactly at $\{p,\infty\},$ one can make sense of quaternionic modular forms in characteristic 0. Then combine generic multiplicity one and freeness to conclude multiplicity one in characteristic $p.$ For more details in this case, see \cite[\S 7.5.13]{EGH_2022_IHES_notes}.
\end{proof}

\bibliographystyle{amsalpha}
\bibliography{/Users/irreducible/Desktop/postdoc/tex/ref}

\end{document}